\DeclareSymbolFont{AMSb}{U}{msb}{m}{n}
\documentclass[11pt,a4paper,leqno,noamsfonts]{amsart}
   \makeatletter
   \renewcommand\@biblabel[1]{#1.}
   \makeatother
\usepackage[english]{babel}
\usepackage[dvipsnames]{xcolor}
\usepackage{graphicx,pifont,soul,physics} 
\usepackage[utopia]{mathdesign}
\usepackage[a4paper,top=3cm,bottom=3cm,
           left=3cm,right=3cm,marginparwidth=60pt]{geometry}
\usepackage[utf8]{inputenc}
\usepackage{braket,caption,comment,mathtools,stmaryrd,ytableau}
\usepackage[usestackEOL]{stackengine}
\usepackage{multirow,booktabs,microtype,relsize}
\usepackage[colorlinks,bookmarks]{hyperref} %
      \hypersetup{colorlinks,%
            citecolor=britishracinggreen,%
            filecolor=black,%
            linkcolor=cobalt,%
            urlcolor=cornellred}
      \numberwithin{equation}{section}
\usepackage[capitalise]{cleveref}
\usepackage[colorinlistoftodos]{todonotes}
\definecolor{antiquewhite}{rgb}{0.98, 0.92, 0.84}
\definecolor{buff}{rgb}{0.94, 0.86, 0.51}
\definecolor{palecopper}{rgb}{0.85, 0.54, 0.4}
\definecolor{fluorescentyellow}{rgb}{0.8, 1.0, 0.0}
\definecolor{bole}{rgb}{0.47, 0.27, 0.23}

\usepackage{amsmath,float}
\usetikzlibrary{decorations.markings}

\definecolor{cornellred}{rgb}{0.7, 0.11, 0.11}
\definecolor{britishracinggreen}{rgb}{0.0, 0.26, 0.15}
\definecolor{cobalt}{rgb}{0.0, 0.28, 0.67}
\DeclareSymbolFont{usualmathcal}{OMS}{cmsy}{m}{n}
\DeclareSymbolFontAlphabet{\mathcal}{usualmathcal}

\newcommand{\A}{{\mathbb{A}}}

\newcommand{\C}{{\mathbb{C}}}

\newcommand{\N}{{\mathbb{N}}}

\newcommand{\Q}{{\mathbb{Q}}}

\newcommand{\T}{{\mathbb{T}}}

\newcommand{\Z}{{\mathbb{Z}}}

\DeclareMathOperator{\Hilb}{Hilb}
\DeclareMathOperator{\tor}{Tor}

\DeclareMathOperator{\Quot}{Quot}

\DeclareMathOperator{\red}{red}

\DeclareMathOperator{\Coh}{Coh}

\DeclareMathOperator{\Span}{Span}

\DeclareMathOperator{\Ob}{Ob}

\DeclareMathOperator{\GL}{GL}

\DeclareFontFamily{OT1}{rsfs}{}
\DeclareFontShape{OT1}{rsfs}{n}{it}{<-> rsfs10}{}
\DeclareMathAlphabet{\curly}{OT1}{rsfs}{n}{it}
\renewcommand\hom{\mathscr{H}\kern-0.3em\mathit{om}}

\newcommand\Hom{\operatorname{Hom}}

\DeclareMathOperator{\lHom}{\mathscr{H}\kern-0.3em\mathit{om}}
\DeclareMathOperator{\RRlHom}{\mathbf{R}\kern-0.025em\mathscr{H}\kern-0.3em\mathit{om}}

\DeclareMathOperator{\lExt}{{\mathscr{E}\kern-0.2em\mathit{xt}}}

\newcommand\Spec{\operatorname{Spec}}
\newcommand\Supp{\operatorname{Supp}}

\newcommand\id{\operatorname{id}}

\newcommand{\Calo}{\mathscr O}

\newcommand{\Calk}{\mathscr K}



\usepackage[all]{xy}
\usepackage{tikz}
\usepackage{tikz-cd}
\usepackage{adjustbox}
\usepackage{rotating}
\usepackage{comment}

\usetikzlibrary{matrix,shapes,intersections,arrows,decorations.pathmorphing}
\tikzset{commutative diagrams/arrow style=math font}
\tikzset{commutative diagrams/.cd,
mysymbol/.style={start anchor=center,end anchor=center,draw=none}}

\tikzset{
shift up/.style={
to path={([yshift=#1]\tikztostart.east) -- ([yshift=#1]\tikztotarget.west) \tikztonodes}
}
}

\theoremstyle{definition}

\newtheorem*{lemma*}{Lemma}
\newtheorem*{theorem*}{Theorem}
\newtheorem*{example*}{Example}
\newtheorem*{fact*}{Fact}
\newtheorem*{notation*}{Notation}
\newtheorem*{definition*}{Definition}
\newtheorem*{prop*}{Proposition}
\newtheorem*{remark*}{Remark}
\newtheorem*{corollary*}{Corollary}

\newtheorem*{conventions*}{Conventions}
\newtheorem*{convention*}{Convention}

\newtheorem{definition}{Definition}[section]

\newtheorem{example}[definition]{Example}

\newtheorem{notation}[definition]{Notation}
\newtheorem{remark}[definition]{Remark}

\newtheoremstyle{thm} 
        {3mm}
        {3mm}
        {\slshape}
        {0mm}
        {\bfseries}
        {.}
        {1mm}
        {}
\theoremstyle{thm}
\newtheorem{theorem}[definition]{Theorem}
\newtheorem{corollary}[definition]{Corollary}
\newtheorem{lemma}[definition]{Lemma}
\newtheorem{prop}[definition]{Proposition}

\newtheoremstyle{ex} 
        {3mm}
        {3mm}
        {}
        {0mm}
        {\scshape}
        {.}
        {1mm}
        {}
\theoremstyle{ex}

\newtheoremstyle{sol} 
        {3mm}
        {3mm}
        {}
        {0mm}
        {\scshape}
        {.}
        {1mm}
        {}
\theoremstyle{sol}

\newtheorem*{Acknowledgments*}{Acknowledgments}

\newcommand{\virg}{``}

\newcommand{\mm}{{\mathfrak{m}}}
\newcommand{\Calt}{{\mathscr{T}}}
\newcommand{\Cale}{{\mathscr{E}}}

\newcommand{\Calb}{{\mathscr{B}}}

\newcommand{\Calf}{{\mathscr{F}}}

\newcommand{\Calm}{{\mathscr{M}}}
\newcommand{\Caln}{{\mathcal{N}}}

\newcommand{\Calh}{{\mathscr{H}}}

\newcommand{\Calg}{{\mathscr{G}}}

\newcommand{\Calr}{{\mathscr{R}}}

\newcommand{\Calu}{{\mathscr{U}}}

\DeclareMathOperator{\sym}{Sym}

\newcommand{\sfrac}[2]{
	{\raise0.8ex\hbox{$#1$} \!\mathord{\left/
			{\vphantom {#1 #2}}\right.\kern-\nulldelimiterspace}
		\!\lower0.8ex\hbox{$#2$}}
} 
\newcommand{\ssum}[2]{ 
	\underset{#1}{\overset{#2}{\sum}} 
}

\newcommand{\Quadrant}[2]{
	\foreach \i in {0,...,{#1}}
	\draw[thick] (\i,0)--(\i,{#2+1});
	\foreach \j in {1,...,{#2}}
	\draw[thick] (0,\j)--({#1},\j);
	\draw[thick] (0,1)--(0,0)--({#1},0)--({#1},1);
	\draw[thick] (0,{#2})--(0,{#2+1})--({#1},{#2+1})--({#1},{#2});
}

\DeclareMathOperator{\OP}{OP}

\DeclareMathOperator{\reg}{reg}

\DeclareMathOperator{\Irr}{Irr}

\DeclareMathOperator{\gen}{gen}

\DeclareMathOperator{\exc}{Exc}

\DeclareMathOperator{\SL}{SL}

\newtheorem{innercustomthm}{Theorem}
\newenvironment{customthm}[1]
{\renewcommand\theinnercustomthm{#1}\innercustomthm}
{\endinnercustomthm}

\usetikzlibrary{patterns}
\DeclareMathAlphabet\BCal{OMS}{cmsy}{b}{n}

\title{Moduli spaces of $\Z/k\Z$-constellations over $\A^2$}

\author{Michele Graffeo}

\begin{document}
\maketitle
	\begin{abstract}Let $\rho:\Z/k \Z\rightarrow \SL(2,\C)$ be a representation of a finite abelian group and let $\Theta^{\gen}\subset \Hom_\Z(R(\Z/k\Z),\Q)$ be the space of generic stability conditions on the set of $G$-constellations. We provide a combinatorial description of all the chambers $C\subset\Theta^{\gen}$ and prove that there are $k!$ of them. Moreover, we introduce the notion of simple chamber and we show that, in order to know all toric $G$-constellations, it is enough to build all simple chambers. We also prove that there are $k\cdot 2^{k-2} $ simple chambers. Finally, we provide an explicit formula for the tautological bundles $\Calr_C$ over the moduli spaces $\Calm _C$ for all chambers $C\subset \Theta^{\gen}$ which   only depends upon the chamber stair which is a  combinatorial object attached to the chamber $C$.
	\end{abstract}
\tableofcontents
\section{Introduction}Given a Gorenstein singular variety $X$, a crepant resolution is a proper birational morphism
	$$Y\xrightarrow{\varepsilon}X$$
	where $Y$ is smooth and the canonical bundle is preserved, i.e. $\omega_Y=\varepsilon^*\omega_X$.
	
	 It was proven by Watanabe in \cite{WATANABE} that the singularities of the form $\A^n/G$, where $G\subset\SL(n,\C)$ is a finite subgroup, are Gorenstein.  Their crepant resolutions appear in several fields of Algebraic Geometry and Mathematical Physics, for example see \cite{BRUZZO,ITOREID,REID} and the references therein.
	
	Even though, in general, crepant resolutions may not exist, their existence is guaranteed in dimension 2 and 3: see \cite{DUFREE} for dimension 2, and see Roan \cite{ROAN1,ROAN2}, Ito \cite{ITOCREP} and Markushevich \cite{MARKU} for dimension 3. In particular, the 3-dimensional case was solved by a case by case analysis, taking advantage of the fact that the conjugacy classes of finite subgroups of $\SL(3,\C)$ were listed, for example in \cite{YU}.
	
	More recently, in \cite{BKR}, Bridgeland, King and Reid proved in one shot that a resolution always exists in dimension 3. The resolution that they proposed is made in terms of $G$-clusters, i.e. $G$-equivariant zero-dimensional subschemes $Z$ of $\A^n$ such that $H^0(Z,\Calo_Z)\cong\C[G]$ as $G$-modules (\Cref{cluster}). In particular, in \cite{BKR} it was proved that there exists a crepant resolution 
	$$G\mbox{-}\Hilb(\A^3)\rightarrow \A^3/G$$
	where $G$-$\Hilb(\A^3)$ is the fine moduli space of $G$-clusters. Notice that this result had already been obtained for abelian actions by Nakamura in \cite{NAKAMURA}.
	
	In \cite{CRAWISHII} Craw and Ishii generalized the notion of $G$-cluster to that of $G$-constellation, i.e. a coherent $G$-sheaf $\Calf$ such that $H^0(\A^n,\Calf)\cong \C[G]$ as $G$-modules (\Cref{constellation}). Moreover, in the case of $G$ abelian the authors in \cite{CRAWISHII} introduced a notion of $\theta$-stability for $G$-constellations (\Cref{stability}), following the ideas in King \cite{KING}.  They proved that, for any abelian subgroup $G\subset\SL(3,\C)$ and for any  crepant resolution $Y\xrightarrow{\varepsilon} \A^3/G$ there exists at least a generic stability condition $\theta$ and an isomorphism $\Calm_\theta \xrightarrow{\varphi} Y $ such that the composition $\varepsilon\circ\varphi$ agrees with the restriction
	$$\Calm_\theta\rightarrow\A^3/G$$
 of the Hilbert--Chow morphism, to the irreducible component $\Calm_\theta$ of the fine moduli space of $\theta$-stable $G$-constellations containing free orbits. Moreover, they conjectured that the same is true for any finite subgroup of $\SL(3,\C)$. Recently, this conjecture has been affirmatively solved by Yamagishi in \cite{prova2}. 
	
	It turns out that the space of generic stability conditions $\Theta^{\gen}\subset \Theta$ is a disjoint union of connected components called chambers. Moreover, in each chamber $C$, the notion of stability is constant, i.e. for any $\theta,\theta'\in C$, a $G$-constellation is $\theta$-stable if and only if it is $\theta'$-stable.
	
	In this paper I will focus on the 2-dimensional abelian case, i.e. the case when $G\subset\SL(2,\C)$ is a finite abelian, and hence cyclic, subgroup. In the literature the singularity $\A^2/G$ is sometimes called the $A_{|G|-1}$ singularity. This case is particularly simple from the point of view of the resolution because we know, from classical surface theory, that there is a unique minimal crepant resolution. Therefore, all the moduli spaces $\Calm_\theta$ are isomorphic as quasi-projective varieties. As a consequence, in order to distinguish two chambers it is enough to study their universal families $\Calu_C \in\Ob\Coh(\Calm_C\times \A^2)$.
	The first main result in the paper is the following. 
	\begin{customthm}{\ref{TEO1}} If $G\subset \SL(2,\C)$ is a finite abelian subgroup of cardinality $k=|G|$, then the space of generic stability conditions $\Theta^{\gen}$ is the disjoint union of $k!$ chambers.
	\end{customthm}
	
	The result in \Cref{TEO1} can be also recovered, via different arguments, from the theory developed by Kronheimer in \cite{KRON} (See also \cite[Chapter 3-\S 3]{CASSLO} for  the algebraic interpretation), but the approach to the abelian case here is different and it helps to prove the other results. 
	
	In order to prove \Cref{TEO1}, I will give an exhaustive combinatorial description of the toric points of the spaces $\Calm_\theta$ in terms of very classical combinatorial objects, namely skew Ferrers diagrams. Such diagrams are standard tools in many branches of mathematics, e.g. enumerative geometry, group theory, commutative algebra etc (for example \cite{BRIANCON,FULTREP,ANDREA}).
	
	Next, I will introduce the notion of simple chamber (\Cref{simplechamber}) and I will show that,   for any indecomposable $G$-constellation $\Calf$, there exists at least a simple chamber $C$ such that $\Calf$ is $\theta$-stable for all $\theta\in C$. This property makes simple chambers useful, because knowing them is the same as knowing all the $ G $-constellations. In order to define simple chambers, I will need to construct chamber stairs (\Cref{chamberstair}), combinatorial objects that I will use to encode all the data of a chamber $C$. 
	
	The second theorem I prove is the following. 
	\begin{customthm}{\ref{teosimple}}If $G\subset \SL(2,\C)$ is a finite abelian subgroup of cardinality $k=|G|$, then the space of generic stability conditions $\Theta^{\gen}$ contains $k\cdot 2^{k-2}$ simple chambers.
	\end{customthm}
	Finally, in \Cref{costruzione} I will give a commutative algebra construction that allows one to write an explicit formula for the tautological bundle 
	$$\Calr_\theta\in\Ob\Coh(\Calm_\theta),$$
	i.e. the pushforward of the universal family $\Calu_\theta \in\Ob\Coh(\Calm_\theta\times \A^2)$ via the first projection. This construction can be easily implemented using some software such as Macaulay2 \cite{M2}. 
	Moreover, it provides a realization of all the moduli spaces $\Calm_\theta$ as a $G$-invariant subvariety of $\Quot_{\Calk_C}^{|G|}(\A^2)$ where $\Calk_C\in\Ob\Coh(\A^2)$ is an ideal sheaf dependent only upon the chamber $C$ such that $\theta\in C$ (see \Cref{quot}). This solves, in 2-dimensions, a problem related to the one raised by Nakamura in \cite[Problem 6.5.]{NAKAMURA} and it also implies that to give a chamber is equivalent to give its chamber stair (\Cref{chamberstair}).

	This paper gives some contributions to the solution of several open problems regarding the subject, and provides some techniques that seem to be applicable to more general situations, such as some non-abelian, even 3-dimensional, case for example following the ideas in \cite{NOLLA1,NOLLA2}.
	
	After providing, in the first section, some technical preliminaries and some known facts, I will devote the second section to a brief description of the singularity $ \A ^ 2 / G $ and to its minimal resolution.
	
	In the third section I will prove that the toric  $ G $-constellations are completely described in terms  of $ G $-stairs, which are certain diagrams whose definition I will give in \Cref{stair}.
	
	The following sections (4 and 5), are devoted to the proofs of the two main theorems, while in the last section I will give the above mentioned commutative algebra construction.
	\section*{Acknowledgments}I thank my advisor Ugo Bruzzo  for his guidance and support. I also thank Mario De Marco, Massimo Gisonni and Felipe Luis López Reyes for very helpful discussions. Special thanks to Alastair Craw, Maria Luisa Graffeo and Andrea T. Ricolfi for giving me useful suggestions while writing the paper.
 \section{Preliminaries}\label{sec1}
	Given a finite group $G$ and a representation $\rho:G\rightarrow \GL(n,\C)$, we have an action of $G$ on the polynomial ring $\C[x_1,\ldots,x_n]$, given by
\[
\begin{tikzcd}[row sep=tiny]
G\times\C[x_1,\ldots,x_n] \arrow{r} & \C[x_1,\ldots,x_n]& \\
(g,p)\arrow[mapsto]{r} & p\circ \rho(g)^{-1}
\end{tikzcd}
\]
	where $p$ and $\rho(g)^{-1}$ are thought respectively as a polynomial and a linear function. Out of this, we can build the quotient singularity $$\A^n/G=\Spec\C[x_1,\ldots,x_n]^G $$
	whose points parametrize the set-theoretic orbits of the action of $G$ on $\A^n$ induced by $\rho$.
	
	Given a representation $\rho:G\rightarrow\GL(n,\C)$, a \textit{$\rho$-equivariant sheaf} (or a \textit{$\rho$-sheaf} in the sense of \cite{BKR}) is a coherent sheaf $\Calf\in\Ob\Coh(\A^n)$ together with a lift to $\Calf$ of the $G$-action  on $\A^n$ induced by $\rho$, i.e. for all $g\in G$ there are morphisms $\lambda_g^\Calf:\Calf\rightarrow\rho(g)^*\Calf$ such that:
	\begin{itemize}
		\item $\lambda_{1_G}^\Calf=\id_\Calf$,
		\item $\lambda_{hg}^\Calf=\rho(g)^*(\lambda^\Calf_h)\circ \lambda^\Calf_g$,
	\end{itemize}
where $1_G$ is the unit of $G$.
	In particular, this induces a structure of representation on the vector space $H^0(\A^n,\Calf)$ as above
\[
\begin{tikzcd}[row sep=tiny]
G\times H^0(\A^n,\Calf) \arrow{r} & H^0(\A^n,\Calf) \\
(g,s)\arrow[mapsto]{r} & (\lambda_g^\Calf)^{-1}\circ \rho(g)^*(s).
\end{tikzcd}
\]

	Whenever the representation is an inclusion $G\subset \GL(n,\C)$ we will omit the representation and we will talk about  \textit{$G$-equivariant sheaf} (or \textit{$G$-sheaf}).
	\begin{definition}\label{cluster} Let $G\subset \GL(n,\C)$ be a finite subgroup. A \textit{$G$-cluster} is a zero-dimensional subscheme $Z$ of $\A^n$ such that:
		\begin{itemize}
			\item the structure sheaf $\Calo_Z$ is $G$-equivariant, i.e. the ideal $I_Z$ is invariant with respect to the action of $G$ on $\C[x_1,\ldots,x_n]$, and
			\item if $\rho_{\reg}:G\rightarrow\GL(\C[G])$ is the regular representation, then there is an isomorphism of representations $$\varphi: H^0(Z,\Calo_Z)\rightarrow \C[G],$$
			i.e. $\varphi$ is an isomorphism of vector spaces such that the following diagram:
			\begin{center}
				\begin{tikzpicture}
					\node at (-2,1) {$G\times H^0(Z,\Calo_Z)$};
					\node at (-2,-1) {$G\times \C[G]$};
					\node at (1,1) {$ H^0(Z,\Calo_Z)$};
					\node at (1,-1) {$\C[G]$};
					\node[right] at (1,0) {\small$\varphi$};
					\node[left] at (-2,0) {\small$\id_G\times\varphi$};
					\draw[->] (1,0.7)--(1,-0.7);
					\draw[->] (-2,0.7)--(-2,-0.7);
					\draw[->] (-0.7,1)--(0.1,1);
					\draw[->] (-1.1,-1)--(0.5,-1);
				
					\end{tikzpicture}
			\end{center}
			where the horizontal arrows are the $G$-actions, commutes.
		\end{itemize}
	
		We will denote by $\Hilb^G(\A^n)$ the fine moduli space of $G$-clusters and, by $G$-$\Hilb(\A^n)$ the irreducible component of $\Hilb^G(\A^n)$ containing the free $G$-orbits.
	\end{definition}
   Recall that,  for all $ n \ge 1$ and for all $G\subset \SL(n,\C)$ finite subgroup, the singularities of the form $\A^n/G$ are Gorenstein (cf. \cite{WATANABE}).  
   
	\begin{theorem}[{\cite[Theorem 1.2]{BKR}}]\label{BKR} Let $G\subset \SL(n,\C)$ be a finite subgroup where $n=2,3$. Then, the Hilbert-Chow morphism
		$$Y:=G\mbox{-}\Hilb(\A^n)\xrightarrow{\varepsilon}\A^n/G=:X$$ is a crepant resolution of singularities, i.e. $\omega_Y\cong \varepsilon^*\omega_X$.
	\end{theorem}
\begin{remark} The Hilbert-Chow morphism $\varepsilon$ mentioned in \Cref{BKR} is a $G$-equivariant version of the usual Hilbert-Chow morphism
	$$\overline{\varepsilon}:\Hilb^{|G|}(\A^n)\rightarrow \sym^{|G|}(\A^n).$$
	In particular $\varepsilon$ can be thought of as the restriction of $\overline{\varepsilon}$ to the $G$-invariant subvariety $G$-$\Hilb(\A^n)\subset \Hilb^{|G|}(\A^n)$.
\end{remark}
	A natural generalization of the concept of a $G$-cluster is given in \cite{CRAWISHII}, and it is achieved by consider coherent $\Calo_{\A^n}$-modules which are not necessarily the structure sheaves of zerodimensional subschemes of $\A^n$.
	\begin{definition}[{\cite[Definition 2.1]{CRAWISHII}}]\label{constellation} Let $G\subset \GL(n,\C)$ be a finite subgroup. A \textit{$G$-constellation} is a coherent $\Calo_{\A^n}$-module $\Calf$ on $\A^n$ such that:
		\begin{itemize}
			\item $\Calf$ is $G$-equivariant, i.e. there is a fixed lift on $\Calf$ of the $G$-action  on $\A^n$, and
			\item there is an isomorphism of representations $$\varphi: H^0(\A^n,\Calf)\rightarrow \C[G].$$
	\end{itemize}\end{definition}
	\begin{remark} Since a $G$-constellation $\Calf$ is a coherent sheaf on the affine variety $\A^n$, sometimes, by abuse of notations, we address the name $G$-constellation to the space of global sections $H^0(\A^n,\Calf)$ as well as $\Calf$ and, sometimes, we   treat a $G$-constellation as if it were a $\C[x_1,\ldots,x_n]$-module, meaning that we are working with the space of its global sections.
	\end{remark}
	\begin{remark}\label{freunic}  The $G$-equivariance hypothesis implies that the support of a $G$-constellation is a  union of $G$-orbits. Moreover, for dimensional reasons, the only constellations supported on a free orbit $Z$ are isomorphic to the structure sheaf $\Calo_Z$.
	\end{remark}
	\begin{remark}\label{shur}Recall that ({see, for example, \cite[chapters 1 and 2]{FULTREP}}), given a finite group $G$ and the set of isomorphism classes of its irreducible representations $$\Irr(G)=\{\mbox{Irreducible representations}\}/\mbox{iso},$$
		there is a ring isomorphism $$\Psi:R(G)\xrightarrow{\sim}\underset{\rho\in\Irr(G)}{\bigoplus} \Z\rho, $$
		where $(R(G),\oplus)$ is the Grothendieck group of isomorphism classes of representations of $G$, and the ring structure (on both sides) is induced by tensor product $\otimes$ of representations. Moreover $\Irr(G)=\{\rho_1,\ldots,\rho_s\}$ is finite, and we have the correspondence:
\[
\begin{tikzcd}[row sep=tiny]
R(G) \arrow{r}{\Psi} & \underset{i=1}{\overset{s}{\bigoplus}}\Z\rho_i \\
\C[G]\arrow[mapsto]{r} & (\dim \rho_1,\ldots,\dim \rho_s).
\end{tikzcd}
\]
	\end{remark}
	Following the ideas in \cite{KING}, the above mentioned properties allow one to introduce a notion of stability on the set of $G$-constellations. Given a finite subgroup $G\subset \SL(n,\C)$ (where $n=2,3$), the \textit{space of stability conditions} for $G$-constellations is 
	\[
	\Theta=\Set{ \theta\in\Hom_\Z(R(G),\Q)|\theta(\C[G])=0}
	\]
	\begin{definition}\label{stability} Let $\theta\in\Theta$ be a stability condition. A $G$-constellation $\Calf$ is said to be \textit{$\theta$-(semi)stable} if, for any proper $G$-equivariant subsheaf $0\subsetneq \Cale\subsetneq \Calf$, we have $$\theta(H^0(\A^n,\Cale)) \underset{(\ge)}{>}0.$$
		A stability condition $\theta$ is \textit{generic} if the notion of $\theta$-semistability is equivalent to the notion of $\theta$-stability. Finally, we   denote by $\Theta^{\gen}\subset\Theta$ the subset of generic stability conditions.
	\end{definition}
\begin{definition} A $G$-constellation $\Calf$ is \textit{indecomposable} if it cannot be written as a direct sum
	$$\Calf=\Cale_1\oplus \Cale_2,$$
	where $\Cale_1,\Cale_2$ are proper $G$-subsheaves, and it is \textit{decomposable} otherwise.
\end{definition}
\begin{remark}
	If we think of a $ G $-constellation as its space of global sections, a $G$-constellation $F=H^0(\Calf,\A^n)$ is indecomposable if it cannot be written as a direct sum
	$$F=E_1\oplus E_2,$$
	where $E_1,E_2$ are proper $G$-equivariant $\C[x_1,\ldots,x_n]$-submodules.
\end{remark}
	\begin{remark} If $\Calf$ is decomposable, then it is not $\theta$-stable for any stability condition $\theta\in\Theta$. 
		
		Since, for our purpose, we are interested in indecomposable $G$-constellations, whenever not specified a $G$-constellation will always be indecomposable.
	\end{remark}
\begin{remark}\label{frestable} If $Z\subset\A^n$ is a free orbit, then $\Calo_Z$ does not admit any proper $G$-subsheaf. Therefore, it is $\theta$-stable for all $\theta\in\Theta$.
\end{remark}
	\begin{definition}  Let $\theta\in\Theta^{\gen}$ be a generic stability condition. We denote by $\Calm_\theta$ the (fine) moduli space of $\theta$-stable $G$-constellations.
	\end{definition}
The theorem below brings together results from \cite{CRAWISHII,BKR,prova2}.

	\begin{theorem}\label{CRAWthm} The following results are true for $n=2,3$.
		\begin{itemize}
			\item The subset $\Theta^{\gen}\subset\Theta$ of generic parameters is open and dense.
			It is the disjoint union of finitely many open convex polyhedral cones in $\Theta$ called \textit{chambers}.
			\item For generic $\theta\in\Theta^{\gen}$, the moduli space $\Calm_\theta$ exists and it depends only upon the chamber $C\subset\Theta^{\gen}$ containing $\theta$, so we write $\Calm_C$ in place of $\Calm_\theta$ for any $\theta\in C$. Moreover, the Hilbert--Chow morphism, which associates to each $G$-constellation $\Calf$ its support $\Supp(\Calf)$, $\varepsilon\colon \Calm_C\rightarrow \A^n/G$, is a crepant resolution.
			\item(Craw--Ishii Theorem \cite{CRAWISHII}) Given a finite abelian subgroup $G\subset\SL(n,\C)$, suppose $Y\xrightarrow{\varepsilon} \A^n/G$ is a projective crepant resolution. Then $Y \cong \Calm_C$ for some chamber $C\subset\Theta$ and $\varepsilon=\varepsilon_C$ is the Hilbert-Chow morphism.
			\item(Yamagishi Theorem \cite{prova2}) Given a finite subgroup $G\subset\SL(n,\C)$, suppose $Y\xrightarrow{\varepsilon} \A^n/G$ is a projective crepant resolution. Then $Y \cong \Calm_C$ for some chamber $C\subset\Theta$.
			\item There exists a chamber $C_G\subset \Theta^{\gen}$ such that $\Calm_{C_G}=G$-$\Hilb(\A^n)$.			
		\end{itemize}
	\end{theorem}
		We will adopt the same notation as \cite{CRAWISHII}  for the universal family of $C$-stable $G$-constellations, namely $\Calu_C\in\Ob\Coh(\Calm_C\times\A^n)$, and for the tautological bundle $\Calr_C:={(\pi_{\Calm_C})}_*\Calu_C$.
	\begin{remark} The hypothesis of \Cref{CRAWthm}, \Cref{freunic,frestable} imply, together with the third point of \Cref{CRAWthm}, that if we denote by $U_C=\Calm_C\smallsetminus\exc (\varepsilon_C)$ the complement of the exceptional locus of the Hilbert-Chow morphism then, for any two chambers $C,C'\subset\Theta^{\gen}$, then there is a canonical isomorphism of families over $\A^n/G$-schemes
		\begin{center}
			\begin{tikzpicture}
				\node at (0,0) {${\Calu_C}_{|_{U_{C}\times\A^n}}$};
				\node at (0,-1) {$U_C\times\A^n$};
				\draw (0,-0.2)--(0,-0.8);
				\node at (3,0) {${\Calu_{C'}}_{|_{U_{C'}\times\A^n}}$};
				\node at (3,-1) {$U_{C'}\times\A^n$,};
				\draw (3,-0.2)--(3,-0.8);
				
				\node at (1.5,-0.5) {$\cong$};
			\end{tikzpicture}
		\end{center}
	i.e. there exists a unique isomorphism $\varphi_C:U_C\rightarrow U_{C'}$ such that the diagram
\[
\begin{tikzcd}
U_C \arrow{rr}{\varphi} \arrow[swap]{dr}{{\varepsilon}_C} & & U_{C'}\arrow{dl}{\varepsilon_{C'}}\\
& \A^n/G &
\end{tikzcd}
\]
commutes and ${\Calu_C}_{|_{U_{C}\times\A^n}}\cong (\varphi\times \id_{\A^n})^*{\Calu_{C'}}_{|_{U_{C'}\times\A^n}}$.

	In particular, any $U_C$ parametrizes the free orbits of the $G$-action as the complement of the singular locus of $\A^n/G $ does.
	\end{remark}
 \section{The two-dimensional abelian case}
	In this section we  introduce some notation that we will use throughout the rest of the paper. Moreover, we give a very brief description of the singularities $ A_{|G|-1} $ and of their respective resolutions.
	
	Throughout all the section, we fix  a finite abelian subgroup $ G \subset \SL (n, \C) $.	
	\subsection{The action of $G$}\label{section2.1}
	Whenever $G\subset \SL(n,\C)$ is a finite abelian subgroup, it is well known that its irreducible representations are 1-dimensional and that the group $G$ and the set $\Irr(G)$ are in bijection. Moreover, the map $\Psi$ in \Cref{shur} is such that 
	\[
\begin{tikzcd}[row sep=tiny]
R(G) \arrow{r}{\Psi} & \underset{\rho\in\Irr(G)}{\bigoplus}\Z\rho\\
\C[G]\arrow[mapsto]{r} & (1,\ldots,1).
\end{tikzcd}
\]
In particular, in dimension 2, it is well known that all finite abelian subgroups $G\subset \SL(2,\C)$ are cyclic. Moreover, for any $k\ge1$, there is only one conjugacy class of abelian subgroups of $\SL(2,\C)$ isomorphic to $\Z/k\Z$. In what follows we will choose, as representative of such conjugacy class,
\begin{equation}\label{Zkaction}
    \Z/k\Z\cong G=\left< g_k=\begin{pmatrix}
		\xi_k^{-1}&0\\0&\xi_k
	\end{pmatrix} \right>\subset\SL(2,\C),
\end{equation}
	where $\xi_k$ is a (fixed) primitive $k$-th root of unity.
	
	We  adopt the following notation for the irreducible representations of $G$:
		\[ \Irr(G)=\Set{\begin{matrix}\begin{tikzpicture}
	    \node at (-0.9,0) {$\rho_i:$};
	    \node at (0,0) {$\Z/k\Z$};
	    \node at (1.5,0) {$\C^*$};
	    \node at (0,-0.5) {$g_k$};
	    \node at (1.5,-0.5) {$\xi_k^i$};
	    \draw[->] (0.5,0)--(1.2,0);
	    \draw[|->] (0.3,-0.5)--(1.2,-0.5);
	\end{tikzpicture}\end{matrix} |i=0,\ldots,k-1}. \]
	Sometimes, we will identify $\Irr(G)$ with the set $\{0,\ldots,k-1\}$ according to the bijection $\rho_j\mapsto j$. Notice that, one may also identify $(\Irr(G),\otimes)$ with the abelian group $(\Z/k\Z,+)$, but  in what follows we will mostly deal with $\Irr(G)$ as a set of indices, hence we will ignore the natural group structure on it.
	\subsection{The quotient singularity $\A^2/G$ and its resolution}\label{section22}
	The singularity obtained in this case is the  so-called $A_{k-1}$ singularity, i.e.
	$$A_{k-1}:=\A^2/G.$$
	This is a rational double point. It is well known that it has a unique minimal, in fact crepant, resolution $Y\xrightarrow{\varepsilon}A_{k-1}$ whose exceptional divisor is a chain of $k-1$ smooth $(-2)$-rational projective curves.
	
	As a consequence of \Cref{CRAWthm} and of the uniqueness of the minimal model of a surface, for any chamber $C$, there is an isomorphism of varieties $\varphi_C:\Calm_C\xrightarrow{\sim} Y$ such that the diagram
\[
\begin{tikzcd}
\Calm_C \arrow{rr}{\varphi_C} \arrow[swap]{dr}{{\varepsilon}_C} & & Y\arrow{dl}{\varepsilon}\\
& A_{k-1}&
\end{tikzcd}
\]
	commutes. What changes between two different chambers $C,C'$ is that they have different universal families $\Calu_C,\Calu_{C'}\in\Ob\Coh(Y\times\A^2)$.
 \section{Toric \texorpdfstring{$G$}{}-constellations}
	This section is devoted to the study of toric $G$-constellations, i.e. those $G$-constellations which, in addition to being $ G $-sheaves, are also $ \T^2 $-sheaves. As it usually happens when dealing with $ \T^2 $-modules, we will see that the $ \C [x, y] $-module structure of a toric $G$-constellation is fully described in terms of combinatorial objects, namely the skew Ferrers diagrams.
	
	This way of proceeding in the description of a $ \T^2 $-module is not new, and it is actually adopted very often in the literature; for example in the study of monomial ideals (see \cite{BRIANCON}) or, more generally, in the study of $ \T^2 $-modules of finite length (see \cite{ANDREA}).
	
	Although many statements can be generalized to higher dimension, from now on we will focus on the 2-dimensional case. 
	\subsection{The torus action}
	Recall that $\A^2$ is a toric variety via the standard torus action:

\begin{equation}\label{eq:sttorusaction}
\begin{tikzcd}[row sep=tiny]
\mathbb{T}^2\times\A^2 \arrow{r} & \A^2 \\
((\sigma_1,\sigma_2),(x,y))\arrow[mapsto]{r} & (\sigma_1\cdot x ,\sigma_2\cdot y).
\end{tikzcd}
\end{equation}

	Notice that, under our assumptions, $G$ is a finite subgroup of the torus $\T^2$. Hence, the action of $\T^2$ commutes with the action of the finite abelian (diagonal) subgroup $G\subset\T^2$. 
	
	This implies that, given a $\theta$-stable $G$-constellation $\Calf$ and an element $\sigma\in \mathbb T^2 $, the pullback $\sigma^*\Calf$ is a $\theta$-stable $G$-constellation. Indeed, $\sigma^*$ induces an isomorphism between the global sections of $\sigma^*\Calf$ and $\Calf$ and hence, $\dim H^0(\A^2,\sigma^*\Calf)=k$. Moreover, $\sigma^*\Calf$ is still a $G$-sheaf if we define, for all $ g\in G$, the morphisms $\lambda^{\sigma^*\Calf}_g:\sigma^*\Calf\rightarrow g^*\sigma^*\Calf$ as
	$$\lambda^{\sigma^*\Calf}_g=\sigma^*\lambda^{\Calf}_g.$$
	 Such morphisms are well defined because $\sigma^*$ and $g^*$ commute, i.e. $g^*\sigma^*\Calf\cong\sigma^*g^*\Calf$ for all $ (g,\sigma)\in G\times\T^2$.  Finally, we have to check that $\sigma^*\Calf$ is $\theta$-stable. This follows from the fact that both the groups $G\subset \T^2$ act diagonally and, as a consequence, if $\Cale\subset \Calf$ is a proper $G$-subsheaf and
	 $$H^0(\A^2,\Cale)=\underset{j=1}{\overset{r}{\bigoplus}}\rho_{i_j}$$ as representations, then $\sigma^*\Cale\subset\sigma^*\Calf$ is a proper $G$-subsheaf and
	 $$H^0(\A^2,\sigma^*\Cale)=\underset{j=1}{\overset{r}{\bigoplus}}\rho_{i_j}$$as representations.
	\begin{definition} As explained above, the torus $\T^2$ acts on $\Calm_C$, for any chamber $C$. We  say that a (indecomposable) $G$-constellation $\Calf$  is \textit{toric} if it corresponds to a torus fixed point.
	\end{definition}
\begin{remark}\label{toric torus}  A $G$-constellation $\Calf$ is toric if and only if it    admits  a structure of $\T^2$-sheaf.  Indeed, if $\Calf$ is a torus fixed point  one possible $\T^2$-structure is obtained from the following associations \[
\begin{tikzcd}[row sep=tiny]
\T^2\times H^0(\A^2,\Calf) \arrow{r} & H^0(\A^2,\Calf)& \\
(\sigma,p)\arrow[mapsto]{r} & p\circ \sigma^{-1},
\end{tikzcd}
\] for  $\sigma\in\T^2$ acting  on $\A^2$ as in \eqref{eq:sttorusaction}. We stress that the $\T^2$-equivariant structure on $\Calf$ is not unique. Indeed any such structure can be twisted by characters of $\mathbb T^2$. 
\end{remark}
\begin{definition} We say that a $G$-constellation $\Calf$ is \textit{nilpotent} if the endomorphisms $x\cdot $ and $y\cdot$ of the $\C[x,y]$-module $H^0(\A^2,\Calf)$ are nilpotent.
	\end{definition}
	\begin{remark}\label{suppnilp} A $G$-constellation $\Calf$ is supported at the origin $0\in\A^2$ if and only if it is nilpotent. This follows from the relation between the annihilator of a $\C[x,y]$-module and the support of the sheaf associated to it (see \cite[Section 2.2]{EISENBUD}). Moreover, \Cref{CRAWthm} implies that nilpotent $C$-stable $G$-constellations correspond to points of the exceptional locus of the crepant resolution $\Calm_C$.
	\end{remark}
\begin{remark}\label{modrep} Given a $G$-constellation $F=H^0(\A^2,\Calf)$, we can compare its structures of $G$-representation and of $\C[x,y]$-module. Looking at the induced action of $G$ on $\C[x,y]$, it turns out that, if $s\in\rho_i$ via the isomorphism $F\cong\C[G]$ then:
$$x\cdot s\in\rho_{i+1},$$
and,
$$y\cdot s\in\rho_{i-1}.$$
\end{remark}
	\begin{prop}\label{xyfazero} If $F=H^0(\A^2,\Calf)$ is a nilpotent $G$-constellation then the endomorphism $xy\cdot$ is the zero endomorphism. 
	\end{prop}
	\begin{proof}The $G$-constellation $F $ is a $k$-dimensional $\C$-vector space. Let us pick a basis $$\{v_0,\ldots,v_{k-1}\}$$ of $F$ such that, for all $i=0,\ldots,k-1$, $v_i\in\rho_i$ under the isomorphism $F\cong\C[G]$. As in \Cref{modrep}, for all $ i=0,\ldots,k-1$, we have:
		$$x\cdot v_i\in\rho_{i+1},$$
		and,
		$$y\cdot v_i\in\rho_{i-1}$$
		where the indices are thought modulo $k$. In other words,
		$$x\cdot v_i\in\Span(v_{i+1})\mbox{ and }y\cdot v_i\in\Span(v_{i-1}).$$
		Therefore, we get:
		$$xy\cdot v_i\in\Span(v_i),\quad \forall i=0,\ldots,k-1$$
		i.e.
		$$xy\cdot v_i=\alpha_iv_i,\mbox{ with }\alpha_i\in\C,\quad \forall i=0,\ldots,k-1.$$
		Now, the nilpotency hypothesis implies that $\alpha_i=0$ for all $i=0,\ldots,k-1$.
	\end{proof}

	\begin{remark}\label{toricisnilp}If a $G$-constellation $F=H^0(\A^2,\Calf)$ is toric, then it is also nilpotent. Indeed, following the same logic as in the proof of \Cref{xyfazero} we have
		$$x^k\cdot v_i=\alpha_i v_i,\mbox{ with }\alpha_i\in\C, \quad\forall i=0,\ldots,k-1,$$
		but torus equivariancy implies $\alpha_i =0$ for all $i=0,\ldots,k-1$.
	\end{remark}
	\subsection{Skew Ferrers diagrams and $G$-stairs}\label{section 32}
	The advantage of working with toric $G$-constellations is that their spaces of global sections can be described in terms of monomial ideals whose data are described by means of combinatorial objects.
	
	We can associate, to each element of the natural plane $\N^2$, two labels: namely a monomial and an irreducible representation. We achieve this by saying that \textit{a polynomial $p\in\C[x,y]$ belongs to an irreducible representation $\rho_i$} if 
	$$\forall g\in G,\quad  g \cdot p=\rho_i(g)p  $$
	i.e. $p$ is an eigenfunction for the linear map $g\cdot$ with the complex number $\rho_i(g)$ as eigenvector. In particular, with the notations in \Cref{section2.1}, the monomial $x^iy^j$ belongs to the irreducible representation $\rho_{i-j}$ of the abelian group $G$, where the index is tought modulo $k$. According to this association, we can define the \textit{representation tableau $\Calt_G$} as 
\[\Calt_G=\Set{(i,j,t)\in\N^2\times \Irr(G)|i-j\equiv t\ (\mod k\ )}\subset \N^2\times \Irr(G).\]
	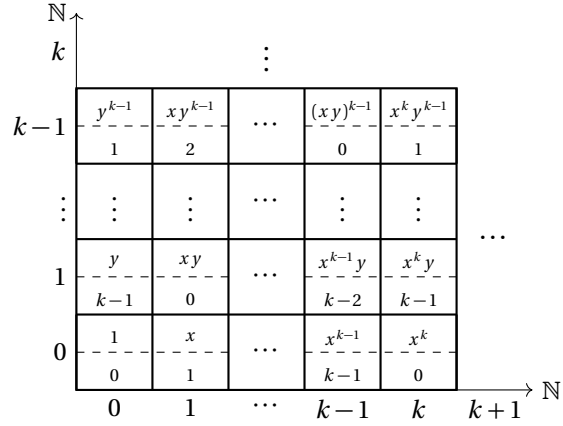
\begin{figure}[H]
		\begin{tikzpicture}
			\node[left] at (0,5) {\small$\N$};
			\node[right] at (6,0) {\small$\N$};
			\draw[<->] (0,5)--(0,0)--(6,0);
			\node at (5.5,2) {$\cdots$};
			\node at (2.5,4.5) {$\vdots$};
			\Quadrant{5}{3}{4}
			\newcommand{\y}{4}
			\draw[dashed] (3,0.5)--(5,0.5);
			\draw[dashed] (3,1.5)--(5,1.5);
			\draw[dashed] (0,0.5)--(2,0.5);
			\draw[dashed] (0,1.5)--(2,1.5);
			
			\draw[dashed] (3,3.5)--(5,3.5);
			\draw[dashed] (0,3.5)--(2,3.5);
			\node at (0.5,0.2) {\tiny$0$};
			\node at (1.5,0.2) {\tiny$1$};
			\node at (2.5,0.5) {\small$\cdots$};
			\node at (0.5,0.7) {\tiny$1$};
			\node at (1.5,0.7) {\tiny$x$};
			\node at (4.5,0.7) {\tiny$x^k$};
			\node at (3.5,0.7) {\tiny$x^{k-1}$};
			\node at (3.5,1.7) {\tiny $x^{k-1}y$};
			\node at (4.5,0.2) {\tiny$0$};
			\node at (3.5,0.2) {\tiny$k-1$};
			\node at (1.5,1.7) {\tiny$xy$};
			\node at (4.5,1.7) {\tiny$x^ky$};
			
			\node at (0.5,1.7) {\tiny$y$};
			\node at (0.5,1.2) {\tiny$k-1$};
			\node at (1.5,1.2) {\tiny$0$};
			\node at (2.5,1.5) {\small$\cdots$};
			\node at (3.5,1.2) {\tiny$k-2$};
			\node at (4.5,1.2) {\tiny$k-1$};
			
			\node at (0.5,3.7) {\tiny$y^{k-1}$};
			\node at (1.5,3.7) {\tiny$xy^{k-1}$};
			\node at (4.5,3.7) {\tiny$x^ky^{k-1}$};
			\node at (3.5,3.7) {\tiny$(xy)^{k-1}$};
			
			\node at (0.5,2.5) {\small$\vdots$};
			\node at (1.5,2.5) {\small$\vdots$};
			\node at (2.5,2.5) {\small$\cdots$};
			\node at (3.5,2.5) {\small$\vdots$};
			\node at (4.5,2.5) {\small$\vdots$};
			
			\node at (0.5,3.2) {\tiny$1$};
			\node at (1.5,3.2) {\tiny$2$};
			\node at (2.5,3.5) {\small$\cdots$};
			\node at (3.5,3.2) {\tiny$0$};
			\node at (4.5,3.2) {\tiny$1$};
			\node[below] at (0.5,0) {\small 0};
			\node[below] at (1.5,0) {\small 1};
			\node[below] at (2.5,0) {\small $\cdots$};
			\node[below] at (3.5,0) {\small $k-1$};
			\node[below] at (4.5,0) {\small $k$};
			\node[below] at (5.5,0) {\small $k+1$};

			\node[left] at (0,0.5) {\small 0};
			\node[left] at (0,1.5) {\small 1};
			\node[left] at (0,2.5) {\small $\vdots$};
			\node[left] at (0,3.5) {\small $k-1$};
			\node[left] at (0,4.5) {\small $k$};
			
		\end{tikzpicture}
\caption{The representation tableau $\Calt_G$.}
\label{Tableau}
\end{figure}
Notice that the labeling with the representation is superfluous because the first projection
$$\pi_{\N^2}:\Calt_G\rightarrow \N^2 $$
is a bijection. In any case, this notation is useful to keep in mind that we are dealing with the representation structure as well as with the module structure.

In summary, the representation tableau has the property that
	\begin{center}
		\textit{moving to the right ``increases" the irreducible representation by 1 $  \pmod k  $}\\
		\textit{moving up ``decreases" the irreducible representation by 1 $  \pmod k  $.}
	\end{center}
\begin{definition}
	A \textit{Ferrers diagram (Fd)} is a subset $A$ of the natural plane $\N^2$ such that
	$$(\N^2\smallsetminus A)+\N^2\subset (\N^2\smallsetminus A)$$
	i.e. there exist $s\ge 0$ and $t_0\ge \cdots\ge t_s\ge 0$ such that
	\[
	A=\Set{(i,j)|i=0,\ldots,s \mbox{ and }\ j=0,\ldots,t_i}.
	\]
\end{definition}
\begin{remark} In the literature there is some ambiguity about the name to be given to such diagrams. Indeed, sometimes, they are also called Young tableaux and, by Ferrers diagrams, something else is meant (for some different notations, see for example \cite{FULTREP,ANDREWS}). In any case, we will adopt the notation in \cite{FEDOU}.
\end{remark}
	Pictorially, we see $s$ consecutive columns of weakly decreasing heights. An example is depicted in \Cref{figure2}.
	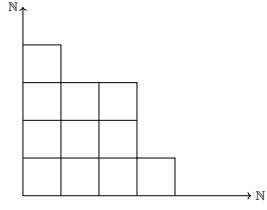
\begin{figure}[H]\scalebox{0.5}{
			\begin{tikzpicture}
				\node[left] at (0,5) {\small$\N$};
				\node[right] at (6,0) {\small$\N$};
				\draw[<->] (0,5)--(0,0)--(6,0);
				\draw[-] (0,4)--(1,4)--(1,3)--(3,3)--(3,1)--(4,1)--(4,0);
				\draw (0,3)--(1,3) --(1,0);
				\draw (0,2)--(3,2);
				\draw (2,3)--(2,0);
				\draw (0,1)--(3,1)--(3,0);
		\end{tikzpicture}}
\caption{An example of Fd where $s=3,t_0=3,t_1=2,t_2=2,t_3=0$.}
\label{figure2}
\end{figure}
\begin{remark}\label{sFdmodulestructre}
	We briefly recall that, starting from a Ferrers diagram $A$, we can build a torus-invariant zero-dimensional subscheme $Z$ of $\A^2$. Indeed, if $B=\N^2\smallsetminus A$ is the complement of $A$, then
	\[I_Z=\Set{x^{b_1}y^{b_2}|(b_1,b_2)\in B}
	\]
	is the ideal of the above mentioned subscheme $Z\subset \A^2$. In particular, the $\C[x,y]$-module structure of $H^0(\A^2,\Calo_Z)=\C[x,y]/I_Z$ is encoded in the Fd, by saying that a box, labeled by the monomial $m\in\C[x,y]$, corresponds to the one-dimensional vector subspace of $H^0(\A^2,\Calo_Z)$ generated by $m$, and
	\begin{center}
		\textit{moving to the right in the Fd is the multiplication by $x$}\\
		\textit{moving up in the Fd is the multiplication by $y$.}
	\end{center}
\end{remark}
\begin{definition}
 Let $\Gamma\subset \N^2$ be a subset of the natural plane. We will say that $\Gamma $ is a  \textit{skew Ferrers diagram (sFd)} if there exist two Ferrers diagrams $\Gamma_1,\Gamma_2\subset\N^2$ such that $\Gamma=\Gamma_1\smallsetminus\Gamma_2$. 
	
	Moreover, we will say that a sFd $\Gamma$ is \textit{connected} if, for any decomposition
	$$\Gamma=\Gamma_1\cup\Gamma_2$$ 
	as disjoint union, there are at least a box in $\Gamma_1$ and a box in $\Gamma_2$ which share an edge.
\end{definition}
	\begin{lemma}\label{lemmatec} 
 Let $A_1,A_2\subset \N^2$ be two Ferrers diagrams and let $\Gamma\subset \N^2$ be the skew Ferrers diagram $\Gamma=A_1\smallsetminus A_2$. Consider, for $i=1,2$, the ideals 
 \[
		I_{A_i}=\left( \Set{x^{b_1}y^{b_2}\in\C[x,y]|(b_1,b_2)\in \N^2\smallsetminus A_i}\right).
 \]
 Then, the isomorphism class of the torus equivariant $\C[x,y]$-module 
$$M_\Gamma=I_{A_2}/I_{A_2}\cap I_{A_1}=I_{A_2}/I_{A_2\cup A_1},$$
 is independent of the choice of $A_1,A_2$. Equivalentely, for any other choice of $A_1',A_2'\subset \N^2$ such that  $\Gamma=A_1\smallsetminus A_2$, the torus equivariant $\C[x,y]$-modules 
$M_\Gamma$ and $I_{A_2'}/I_{A_2'\cup A_1'}$ are isomorphic.  
	\end{lemma}
	\begin{proof} 		
		The fact that $M_\Gamma$ does not depend on the decomposition $\Gamma=A_1\smallsetminus A_2$ follows noticing that, if we pick another decomposition $\Gamma=A_1'\smallsetminus A_2'$, then the isomorphism of $\C$-vector spaces
		$$I_{A_2}/I_{A_2}\cap I_{A_1}\rightarrow I_{A_2'}/I_{A_2'}\cap I_{A_1'},$$
		which associates the class $x^\alpha y^\beta+ I_{A_2}\cap I_{A_1}$ to the class $x^\alpha y^\beta+ I_{A_2'}\cap I_{A_1'}$, is an isomorphism of $\C[x,y]$-modules. 
	\end{proof}

	Now, instead of focusing just on subsets of the natural plane $\N^2$, we introduce more structure by looking at subsets of the representation tableau.
	
	In some instances, we will need to work with abstract sFd's obtained forgetting about the monomials.
	\begin{definition}
		A \textit{$G$-sFd} is a subset $A\subset\Calt_G$ of the representation tableau whose image $\pi_{\N^2}(A)$, under the first projection
		$$\pi_{\N^2}:\Calt_G\rightarrow\N^2,$$
		is a sFd.
		
		An \textit{abstract $G$-sFd} is a diagram $\Gamma$ made of boxes labeled by the irreducible representations of $G$ that can be embedded into the representation tableau as a $G$-sFd.
	\end{definition}
		\begin{example} Consider the $\Z/3\Z$-action on $  \A^2$ defined in \eqref{Zkaction}. In \Cref{figure3} are shown an abstract $G$-sFd and two of its possible realizations as $G$-sFd.
			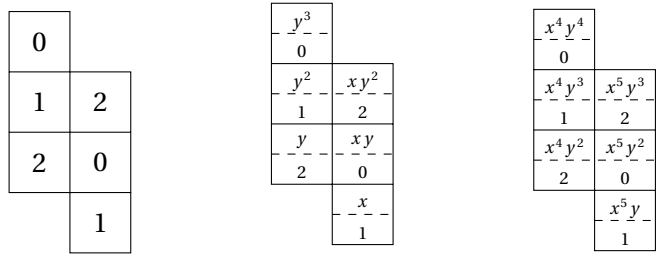
\begin{figure}[H]
				\begin{tikzpicture}[scale=0.8]
					\draw (0,2)--(0,4)-- (1,4)--(1,3)--(0,3);
					\draw (2,1)--(2,3)--(1,3)--(1,2)--(2,2);
					\draw (1,1)--(2,1)--(2,0)--(1,0)--(1,2)--(0,2)--(0,1)--(1,1);
					\node at (1.5,0.5) {1};
					\node at (1.5,1.5) {0};
					\node at (1.5,2.5) {2};
					\node at (0.5,1.5) {2};
					\node at (0.5,2.5) {1};
					\node at (0.5,3.5) {0};
				\end{tikzpicture}  
			\begin{tikzpicture}[scale=0.8]
				\node at (-2,0) {$\ $};
				\node at (4,0) {$\ $};
			\draw (0,2)--(0,4)-- (1,4)--(1,3)--(0,3);
			\draw (2,1)--(2,3)--(1,3)--(1,2)--(2,2);
			\draw (1,1)--(2,1)--(2,0)--(1,0)--(1,2)--(0,2)--(0,1)--(1,1);
			\draw[dashed] (1,0.5)--(2,0.5);
			\draw[dashed] (0,1.5)--(2,1.5);
			\draw[dashed] (0,2.5)--(2,2.5);
			\draw[dashed] (0,3.5)--(1,3.5);
			\node at (1.5,0.2) {\tiny$1$};
			\node at (1.5,1.2) {\tiny$0$};
			\node at (1.5,2.2) {\tiny$2$};
			\node at (0.5,2.2) {\tiny$1$};
			\node at (0.5,3.2) {\tiny$0$};
			\node at (0.5,1.2) {\tiny$2$};

			\node at (0.5,1.7) {\tiny$y$};
			\node at (1.5,1.7) {\tiny$xy$};
			\node at (1.5,2.7) {\tiny$xy^2$};
			\node at (0.5,2.7) {\tiny$y^2$};
			\node at (0.5,3.7) {\tiny$y^3$};
			\node at (1.5,0.7) {\tiny$x$};
		\end{tikzpicture}  
	\begin{tikzpicture}[scale=0.8]
	\draw (0,2)--(0,4)-- (1,4)--(1,3)--(0,3);
	\draw (2,1)--(2,3)--(1,3)--(1,2)--(2,2);
	\draw (1,1)--(2,1)--(2,0)--(1,0)--(1,2)--(0,2)--(0,1)--(1,1);
	\draw[dashed] (1,0.5)--(2,0.5);
	\draw[dashed] (0,1.5)--(2,1.5);
	\draw[dashed] (0,2.5)--(2,2.5);
	\draw[dashed] (0,3.5)--(1,3.5);
	\node at (1.5,0.2) {\tiny$1$};
	\node at (1.5,1.2) {\tiny$0$};
	\node at (1.5,2.2) {\tiny$2$};
	\node at (0.5,2.2) {\tiny$1$};
	\node at (0.5,3.2) {\tiny$0$};
	\node at (0.5,1.2) {\tiny$2$};
	
	\node at (0.5,1.7) {\tiny$x^4y^2$};
	\node at (1.5,1.7) {\tiny$x^5y^2$};
	\node at (1.5,2.7) {\tiny$x^5y^3$};
	\node at (0.5,2.7) {\tiny$x^4y^3$};
	\node at (0.5,3.7) {\tiny$x^4y^4$};
	\node at (1.5,0.7) {\tiny$x^5y$};
\end{tikzpicture}  
\caption{An abstract $\Z/3\Z$-sFd and two of its possible realizations as $\Z/3\Z$-sFd.}
\label{figure3}
\end{figure}
		On the other hand, the diagram in \Cref{figure4} is not an abstract $G$-sFd.
		\begin{figure}[H]
			\begin{tikzpicture}[scale=0.6]
				\draw (1,0)--(1,2)--(0,2)--(0,0)--(2,0)--(2,1)--(0,1);
				\node at (0.5,0.5) {0};
				\node at (1.5,0.5) {2};
				\node at (0.5,1.5) {2};
			\end{tikzpicture}  
\caption{\ }
\label{figure4}
\end{figure}
		\end{example}
	\begin{remark} Given any subset $\Xi$ of the representation tableau and any monomial $x^\alpha y^\beta$ we will denote by $x^\alpha y^\beta\cdot \Xi$ the subset of the representation tableau obtained by translating $\Xi$ $\alpha$ steps to the right and $\beta$ steps up. Notice that this is compatible with the association $\N^2\leftrightarrow\{\mbox{monomials in two variables}\}$ as explained in \Cref{sFdmodulestructre}.
	\end{remark}
\begin{lemma}\label{toric basis} If $\Calf$ is a torus equivariant $G$-constellation then there exists a basis $\{v_0,\ldots,v_{k-1}\}$ of $F=H^0(\A^2,\Calf)$ such that
	\begin{enumerate}
		\item for all $i=0,\ldots,k-1$, we have $v_i\in\rho_i$,
		\item for all $i=0,\ldots,k-1$, the sections $v_i$ are semi-invariant functions with respect some character $\chi_i$ of $\T^2$, i.e. $(a,b)\cdot v_i=\chi_i(a,b) v_i$ for all $(a,b)\in\T^2$,
		\item for all $i=0,\ldots,k-1$,
		$$\begin{cases}
			x\cdot v_i\in \{v_{i+1},0\},\\
			y\cdot v_i\in\{ v_{i-1},0\}.
		\end{cases}$$
	\end{enumerate}
\end{lemma}
\begin{proof} We can always pick a basis $\{\widetilde{v}_0,\ldots,\widetilde{v}_{k-1}\}$ which satisfies \textit{(1)} and \textit{(2)}. Moreover, it follows from \Cref{modrep} that:
	$$\begin{cases}
		x\cdot \widetilde{v}_i\in\Span(\widetilde{v}_{i+1}),\\
		y\cdot \widetilde{v}_i\in\Span(\widetilde{v}_{i-1}),
	\end{cases}$$
	where the indices are thought modulo $k$. The fact that $\Calf$ is toric implies that there are no \virg cycles", i.e. there are no $1<s<k$ and
	\[
	\Set{(i_j,k_j,h_j,\sigma_j)\in \Irr(G)\times \N^2\times \C^*|
\begin{array}{c}
j=1,\ldots,s,\\
i_j\not=i_{j'}\mbox{ for }j\not=j',\\
k_j+h_{j+1}>0
\end{array}}
	\]
	where the indices are thought modulo $s$, such that
\begin{equation}\label{cycle}\begin{cases}
		(x\cdot)^{k_{1}}\widetilde{v}_{i_1}&=\sigma_1(y\cdot)^{h_{2}}\widetilde{v}_{i_2},\\
		(x\cdot)^{k_{2}}\widetilde{v}_{i_2}&=\sigma_2(y\cdot)^{h_{3}}\widetilde{v}_{i_3},\\
		&\vdots\\
		(x\cdot)^{k_{{s-1}}}\widetilde{v}_{i_{s-1}}&=\sigma_{s-1}(y\cdot)^{h_{s}}\widetilde{v}_{i_s},\\
		(x\cdot)^{k_{s}}\widetilde{v}_{i_{s}}&=\sigma_s(y\cdot)^{h_{1}}\widetilde{v}_{i_1}.
	\end{cases}
\end{equation}
	Indeed, $x$ and $y$ are semi-invariant functions with respect to the characters
	\[
\begin{tikzcd}[row sep=tiny]
\T^2  \arrow{r}{\lambda_x} & \C^* \\
(a,b)\arrow[mapsto]{r} & a
\end{tikzcd}
\]
and\[
\begin{tikzcd}[row sep=tiny]
\T^2 \arrow{r}{\lambda_y} & \C^* \\
(a,b)\arrow[mapsto]{r} & b
\end{tikzcd}
\]
	of the torus $\T^2$. Then, if we act on both sides of the Equations \eqref{cycle} with some $(a,b)\in \T^2$, we get:

\begin{equation}\label{cycle1}
		\begin{cases}
		\lambda_x(a,b)^{k_{1}}\chi_{i_1}(a,b)(x\cdot)^{k_{1}}\widetilde{v}_{i_1}=\sigma_1\lambda_y(a,b)^{h_{2}}\chi_{i_2}(a,b)(y\cdot)^{h_{2}}\widetilde{v}_{i_2},\\
		\lambda_x(a,b)^{k_{2}}\chi_{i_2}(a,b)(x\cdot)^{k_{2}}\widetilde{v}_{i_2}=\sigma_2\lambda_y(a,b)^{h_{3}}\chi_{i_3}(a,b)(y\cdot)^{h_{3}}\widetilde{v}_{i_3},\\
		\quad\qquad \quad\quad\quad\quad\quad\quad \quad\vdots\\
		\lambda_x(a,b)^{k_{{s-1}}}\chi_{i_{s-1}}(a,b)(x\cdot)^{k_{{s-1}}}\widetilde{v}_{i_{s-1}}=\sigma_{s-1}\lambda_y(a,b)^{h_{s}}\chi_{i_s}(a,b)(y\cdot)^{h_{s}}\widetilde{v}_{i_s},\\
		\lambda_x(a,b)^{k_{{s}}}\chi_{i_s}(a,b)(x\cdot)^{k_{{s}}}\widetilde{v}_{i_{s}}=\sigma_s\lambda_y(a,b)^{h_{1}}\chi_{i_1}(a,b)(y\cdot)^{h_{1}}\widetilde{v}_{i_1},\\
	\end{cases}
\end{equation}
Now, the System \eqref{cycle1} is equivalent to:
	$$\begin{cases}
		a^{k_{1}}\chi_{i_1}(a,b)=b^{h_{2}}\chi_{i_2}(a,b),\\
		a^{k_{2}}\chi_{i_2}(a,b)=b^{h_{3}}\chi_{i_3}(a,b),\\
		\quad \quad\quad\quad\quad\ \ \ \vdots\\
		a^{k_{{s-1}}}\chi_{i_{s-1}}(a,b)=b^{h_{s}}\chi_{i_s}(a,b),\\
		a^{k_{{s}}}\chi_{i_s}(a,b)=b^{h_{1}}\chi_{i_1}(a,b),\\
	\end{cases}$$ 
	which is equivalent to
	\begin{equation}\label{cycle2}
	     a^{{k_1}+\cdots+{k_s}}=b^{h_{1}+\cdots+h_{s}}\quad \forall (a,b)\in\T^2.
	\end{equation}
	Finally, the only solution of \Cref{cycle2} is 
	$${{k_1}=\cdots={k_s}}={h_{1}=\cdots=h_{s}}=0,$$
	which contradicts the hypothesis $k_i+h_{i+1}>0$ for all $ i=1,\ldots,s$. 
	
	We are now ready to build the requested basis. Let $\{w_1,\ldots,w_\ell\}\subset\{\widetilde{v}_0,\ldots,\widetilde{v}_{k-1}\}$ be a minimal set of generators of the $\C[x,y]$-module $F$, i.e. the set
		\[
		\Set{w_j+\mm\cdot F\in F/\mm\cdot F | j=1,\ldots, \ell  }
		\]
		is a basis of the $\C$-vector space $F/\mm\cdot F$. Let us also denote by $F_j$, for $j=1,\ldots,\ell$, the submodule generated by $w_j$. We start by taking, for all $j=1,\ldots,\ell$, as basis of $F_j$ the set
	\[B_j=\Set{x^\alpha y^\beta w_j|\alpha\cdot\beta=0}. \]
	The problem is that in general the union of all $ B_j$'s is not a basis of $F$ because there can be some relations $x^\alpha w_i=\mu y^\beta w_j$ for $i\not=j$ and $\mu\in\C^*\smallsetminus 1 $. The fact that there are no cycles implies that we can re-scale all the elements in each $B_j$ obtaining new $\overline{B}_j$ so that $\underset{j}{\bigcup}\overline{B}_j$ is a basis of $F$ that verifies properties \textit{(1)}, \textit{(2)}, \textit{(3)}. 
\end{proof}
	\begin{prop}\label{const-sFd}
		Given a, possibly decomposable, torus equivariant $G$-constellation $F=H^0(\A^2,\Calf)$, there is (at least) one $G$-sFd whose associated $\C[x,y]$-module is a $G$-constellation isomorphic to $F$. 
	\end{prop}
	\begin{remark}\label{periodn} If we find one $G$-sFd with the required property, then there are infinitely many of them. Indeed, a special property of the representation tableau is that translations enjoy some periodicity properties.
		
		Let $\Gamma$ be a $G$-sFd, then:
		\begin{enumerate}
			\item multiplication by $x$ has period $k$, i.e there is an isomorphism of $\C[x,y]$-modules
			$$M_\Gamma\xrightarrow{\sim}M_{x^k\cdot\Gamma}$$
			which induces an isomorphism of representations between $M_\Gamma$ and $M_{x^k\cdot\Gamma}$;
			\item multiplication by $y$ has period $k$, i.e there is an isomorphism of $\C[x,y]$-modules
			$$M_\Gamma\xrightarrow{\sim}M_{y^k\cdot\Gamma}$$
			which induces an isomorphism of representations between $M_\Gamma$ and $M_{y^k\cdot\Gamma}$;
			\item multiplication by $xy$ is an isomorphism, i.e there is an isomorphism of $\C[x,y]$-modules
			$$M_\Gamma\xrightarrow{\sim}M_{xy\cdot\Gamma}$$
			which induces an isomorphism of representations between $M_\Gamma$ and $M_{xy\cdot\Gamma}$.
		\end{enumerate} 
In particular, all these $G$-sFd's correspond to the same abstract $G$-sFd.
	\end{remark}
	\begin{proof}( \textit{of \Cref{const-sFd}} ).
		Let  $\{v_0,\ldots,v_{k-1}\}$ be a $\C$-basis of $F$ with the properties listed in \Cref{toric basis}, and let $\set{w_j=v_{i_j}|j=1,\ldots,s}$ be a minimal set of generators of $F$ as a $\C[x,y]$-module (see the proof of \Cref{toric basis}). Denote by $F_j$, for $j=1,\ldots,s$, the $\C[x,y]$-submodule of $F$ generated by $w_j$. We can represent each $F_j$ by using diagrams of the form shown in \Cref{diagfja},
		\begin{figure}[ht]
				{\scalebox{0.7}{\begin{tikzpicture}
					\draw (0,2)--(0,0)--(3,0)--(3,1)--(1,1)--(1,3)--(0,3)--(0,2);
					\draw (0,4)--(0,5)--(1,5)--(1,4)--(0,4)--(0,5);
					\draw (4,0)--(5,0)--(5,1)--(4,1)--(4,0)--(5,0);
					\draw (0,2)--(1,2);
					\draw (2,0)--(2,1);
					\draw (0,1)--(1,1)--(1,0);
					\node at (0.5,0.5) {$w_j$};
					\node at (1.5,0.5) {$v_{i_j+1}$};
					\node at (2.5,0.5) {$v_{i_j+2}$};
					\node at (3.5,0.5) {$\cdots$};
					\node at (4.5,0.5) {$v_{i_j+k_j}$};
					
					\node at (0.5,1.5) {$v_{i_j-1}$};
					\node at (0.5,2.5) {$v_{i_j-2}$};
					\node at (0.5,3.6) {$\vdots$};
					\node at (0.5,4.5) {$v_{i_j-h_j}$};
			\end{tikzpicture}}}
			\caption{\ }
			\label{diagfja}
		\end{figure}
		where the integers $k_j$ and $h_j$ are defined by
		\[k_j=\max\Set{\alpha|(x\cdot)^\alpha w_j\not=0}\] and
		\[h_j=\max\Set{\alpha|(y\cdot)^\alpha w_j\not=0},\] 
		 and they are well defined because any toric $G$-constellation is nilpotent by \Cref{toricisnilp}.
		 
		 The $\C[x,y]$-module structure of $F_j$ is encoded in the fact that the multiplication by $ x $ (resp. $ y $) sends the generator of a box (i.e., the generator of the corresponding vector space) to the generator of the box on the right (resp. above). If there is no box on the right (resp. above) this means that the multiplication by $x$ (resp. $y$) is zero.
		 
		Now, we have to glue these diagrams to form the required $G$-sFd. We glue them along boxes with the same labels. First, notice that, if, for some $j\not=j'$ and $r,t\ge 1$, we have $(x\cdot)^rw_j=(x\cdot)^tw_{j'}$, i.e. $i_j+r=i_{j'}+t$ modulo $k$, then $$(x\cdot)^rw_j=(x\cdot)^tw_{j'}=0.$$ Indeed, if $r<t$ (the case $r\ge t$ is analogous) then, a representation argument (see \Cref{xyfazero}) tells us that  $w_j=(x\cdot)^{t-r}w_{j'}$ which, whenever $(x\cdot)^rw_i\not=0$, contradicts the minimality of the generating set $\{ w_1,\ldots,w_s\}$. Analogously, if, for some $j\not=j'$ and $r,t\ge 1$, we have $(y\cdot)^rw_j=(y\cdot)^tw_{j'}$, then $(y\cdot)^rw_j=0$.
		
		Now we show that, if, for some $j\not=j'$ and $r,t\ge 1$, we have $(x\cdot)^rw_j=(y\cdot)^tw_{j'}$, then $r=k_j$ and $t=h_{j'}$. Suppose, by contradiction, that there exists $1\le r<k_{j}$ such that $(x\cdot)^rw_j=(y\cdot)^{t}w_{j'}$ (the case $1\le t<h_{j'}$ is similar). In particular, the minimality assumption implies $t\ge 1$. Since $r<k_j$, by definition of $k_j$, we have $(x\cdot)^{r+1}w_i\not= 0$. Therefore, we get
		$$0\not=(x\cdot)^{r+1} w_j=x\cdot((x\cdot)^{r} w_j)= x\cdot y^{t}\cdot w_{j'}= (x y)\cdot y^{t-1}\cdot w_{j'}=0 $$
		which gives a contradiction.
		
		We show now that there are no ``cycles". Explicitly, suppose that, up to reordering the $v_i's$, and consequently the $w_i's$, we have already glued $\ell $ diagrams of the form depicted in \cref{diagfja} to a diagram of the form shown in \Cref{diagrammone}.
		\begin{figure}[ht]\scalebox{0.7}{
				\begin{tikzpicture}[scale=0.8]
					\draw (8,-6)--(8,-7)--(10,-7)--(10,-6)--(9,-6)--(9,-5)--(8,-5)--(8,-6);
					\draw (5,-3)--(5,-4)--(7,-4)--(7,-3)--(6,-3)--(6,-2)--(5,-2)--(5,-3);
					\draw (1,1)--(1,0)--(3,0)--(3,1)--(2,1)--(2,2)--(1,2)--(1,1);
					\draw (-2,4)--(-2,3)--(0,3)--(0,4)--(-1,4)--(-1,5)--(-2,5)--(-2,4);
					\draw (1,3)--(1,4)--(2,4)--(2,3)--(1,3)--(1,4);
					\draw (-2,6)--(-2,7)--(-1,7)--(-1,6)--(-2,6)--(-2,7);
					\draw (4,0)--(5,0)--(5,1)--(4,1)--(4,0)--(5,0);
					\draw (5,-1)--(6,-1)--(6,0)--(5,0)--(5,-1)--(6,-1);
					\draw (9,-3)--(9,-4)--(8,-4)--(8,-3)--(9,-3)--(9,-4);
					\draw (5,-3)--(6,-3)--(6,-4);
					\draw (1,1)--(2,1)--(2,0);
					\draw (-2,4)--(-1,4)--(-1,3);
					\node at (3.5,0.5) {$\cdots$};
					\node at (-1.5,5.6) {$\vdots$};
					\node at (5.5,-1.4) {$\vdots$};
					\node at (1.5,2.6) {$\vdots$};
					\node at (7.5,-3.5) {$\cdots$};
					\node at (5.5,0.5) {$\ddots$};
					\node at (-1.5,3.5) {$w_{1}$};
					\node at (1.5,0.5) {$w_{{2}}$};
					\node at (5.5,-3.5) {$w_{{\ell-1}}$};
					\node at (8.5,-6.5) {$w_{{\ell}}$};
					\node at (13.2,-4.5) {$x^{k_{{\ell}}}w_{\ell}$};
					\node[right] at (2,6.5) {$y^{h_{{1}}}w_{1}$};
					\node at (4.5,5) {${y^{h_{{2}}}w_{2} }{=}{  x^{k_{{1}}}w_{1} }$};
					\node at (10.5,-1.5) {${{  y^{h_{{\ell}}}w_{\ell} }{=}}{ x^{k_{{\ell-1}}}w_{{\ell-1}}}$};
					\draw[->] (4.5,4.8) to [out = 270 ,in =0 ]  (1.5,3.5);
					\draw[->] (10.2,-1.7) to [out = 270 ,in =0 ]  (8.5,-3.5);
					\draw[->] (13.2,-4.7) to [out = 270 ,in =0 ]  (11.5,-6.5);
					\node at (10.5,-6.5) {$\cdots$};
					\node at (0.5,3.5) {$\cdots$};
					\node at (8.5,-4.4) {$\vdots$};
					\draw (9,-7)--(9,-6)--(8,-6);
					\draw (11,-6)--(11,-7)--(12,-7)--(12,-6)--(11,-6)--(11,-7);
					\draw[<-] (-1.5,6.5)--(2,6.5);
			\end{tikzpicture}}
			\caption{\ }
			\label{diagrammone}
		\end{figure}
		Then, we want to show that there is no gluing $(x\cdot)^{k_{\ell}}w_{\ell}=\sigma(y\cdot)^{h_{1}}w_{1}$ for some $\sigma\in\C^*$, i.e. no gluing of the first and the last boxes of the above diagram. The presence of this cycle would translate into the following system of equalities
		$$\begin{cases}
			(x\cdot)^{k_{1}}w_{1}&=(y\cdot)^{h_{2}}w_{2},\\
			(x\cdot)^{k_{2}}w_{2}&=(y\cdot)^{h_{3}}w_{3},\\
			&\vdots\\
			(x\cdot)^{k_{{\ell-1}}}w_{{\ell-1}}&=(y\cdot)^{h_{\ell}}w_{\ell},\\
			(x\cdot)^{k_{\ell}}w_{{\ell}}&=\sigma(y\cdot)^{h_{1}}w_{1},
		\end{cases}$$
		which cannot be verified by any toric $G$-constellation as explained in the proof of \Cref{toric basis}.
		 
		So far we have proven that each connected component of the required $G$-sFd have the shape depicted in \Cref{shape}.
		\begin{figure}[ht]\scalebox{0.4}{
				\begin{tikzpicture}
					\draw (8,-6)--(8,-7)--(10,-7)--(10,-6)--(9,-6)--(9,-5)--(8,-5)--(8,-6);
					\draw (5,-3)--(5,-4)--(7,-4)--(7,-3)--(6,-3)--(6,-2)--(5,-2)--(5,-3);
					\draw (1,1)--(1,0)--(3,0)--(3,1)--(2,1)--(2,2)--(1,2)--(1,1);
					\draw (-2,4)--(-2,3)--(0,3)--(0,4)--(-1,4)--(-1,5)--(-2,5)--(-2,4);
					\draw (1,3)--(1,4)--(2,4)--(2,3)--(1,3)--(1,4);
					\draw (-2,6)--(-2,7)--(-1,7)--(-1,6)--(-2,6)--(-2,7);
					\draw (4,0)--(5,0)--(5,1)--(4,1)--(4,0)--(5,0);
					\draw (5,-1)--(6,-1)--(6,0)--(5,0)--(5,-1)--(6,-1);
					\draw (9,-3)--(9,-4)--(8,-4)--(8,-3)--(9,-3)--(9,-4);
					\draw (5,-3)--(6,-3)--(6,-4);
					\draw (1,1)--(2,1)--(2,0);
					\draw (-2,4)--(-1,4)--(-1,3);
					\node at (3.5,0.5) {$\cdots$};
					\node at (-1.5,5.6) {$\vdots$};
					\node at (5.5,-1.4) {$\vdots$};
					\node at (1.5,2.6) {$\vdots$};
					\node at (7.5,-3.5) {$\cdots$};
					\node at (5.5,0.5) {$\ddots$};
					\node at (10.5,-6.5) {$\cdots$};
					\node at (0.5,3.5) {$\cdots$};
					\node at (8.5,-4.4) {$\vdots$};
					\draw (9,-7)--(9,-6)--(8,-6);
					\draw (11,-6)--(11,-7)--(12,-7)--(12,-6)--(11,-6)--(11,-7);
			\end{tikzpicture}}
			\caption{\ }
			\label{shape}
		\end{figure}
	
		Moreover, if we forget about the reordering, each box  contains a label $v_i$ whose index increases by one when moving to the right or downward in the diagram. Since we have chosen $ v_i\in\rho_i $ for $i=0,\ldots,k-1$, this diagram fits in the representation tableau (see \Cref{section 32}), i.e. it is an abstract $G$-sFd. After performing all possible gluings, we obtain a number of abstract $G$-sFd's $A_1,\ldots,A_m$ whose shape is drawn in \Cref{shape}.
		
		The last thing to do is to show that we can realize $A_1,\ldots,A_m$ as subsets $\Gamma_1,\ldots,\Gamma_m$ of the representation tableau to get a $G$-sFd, i.e. in such a way that $$\pi_{\N^2}\left(\underset{i=1}{\overset{m}{\bigcup}}\Gamma_i\right)$$
		is a sFd. This can be done in many ways and we explain one possible way to proceed.
		
		We start by realizing $A_1,\ldots,A_m$ as disjoint $G$-sFd's $\Gamma_1,\ldots,\Gamma_m$. This can always be done because, as we observed, $A_1,\ldots,A_m$ are abstract $G$-sFd's and, from any choice of realizations $\widetilde{\Gamma}_1,\ldots,\widetilde{\Gamma}_m$ of them as non-necessarily disjoint $G$-sFd's, we can obtain disjoint $\Gamma_1,\ldots,\Gamma_m$ by performing the translations described in \Cref{periodn}.
		
		At this point, we have $m$ disjoint $G$-sFd's as described in \Cref{tris},
	\begin{figure}[ht]\scalebox{1}{
			\begin{tikzpicture}
				\node at (-4,0) {\scalebox{0.2}{
						\begin{tikzpicture}
					\draw (8,-6)--(8,-7)--(10,-7)--(10,-6)--(9,-6)--(9,-5)--(8,-5)--(8,-6);
					\draw (5,-3)--(5,-4)--(7,-4)--(7,-3)--(6,-3)--(6,-2)--(5,-2)--(5,-3);
					\draw (1,1)--(1,0)--(3,0)--(3,1)--(2,1)--(2,2)--(1,2)--(1,1);
					\draw (-2,4)--(-2,3)--(0,3)--(0,4)--(-1,4)--(-1,5)--(-2,5)--(-2,4);
					\draw (1,3)--(1,4)--(2,4)--(2,3)--(1,3)--(1,4);
					\draw (-2,6)--(-2,7)--(-1,7)--(-1,6)--(-2,6)--(-2,7);
					\draw (4,0)--(5,0)--(5,1)--(4,1)--(4,0)--(5,0);
					\draw (5,-1)--(6,-1)--(6,0)--(5,0)--(5,-1)--(6,-1);
					\draw (9,-3)--(9,-4)--(8,-4)--(8,-3)--(9,-3)--(9,-4);
							\draw (5,-3)--(6,-3)--(6,-4);
							\draw (1,1)--(2,1)--(2,0);
							\draw (-2,4)--(-1,4)--(-1,3);
							\node at (3.5,0.5) {$\cdots$};
							\node at (-1.5,5.6) {$\vdots$};
							\node at (5.5,-1.4) {$\vdots$};
							\node at (1.5,2.6) {$\vdots$};
							\node at (7.5,-3.5) {$\cdots$};
							\node at (5.5,0.5) {$\ddots$};
							\node at (10.5,-6.5) {$\cdots$};
							\node at (0.5,3.5) {$\cdots$};
							\node at (8.5,-4.4) {$\vdots$};
							\draw (9,-7)--(9,-6)--(8,-6);
							\draw (11,-6)--(11,-7)--(12,-7)--(12,-6)--(11,-6)--(11,-7);
				\end{tikzpicture}}};
				\node at (0,0) {\scalebox{0.2}{
						\begin{tikzpicture}
					\draw (8,-6)--(8,-7)--(10,-7)--(10,-6)--(9,-6)--(9,-5)--(8,-5)--(8,-6);
					\draw (5,-3)--(5,-4)--(7,-4)--(7,-3)--(6,-3)--(6,-2)--(5,-2)--(5,-3);
					\draw (1,1)--(1,0)--(3,0)--(3,1)--(2,1)--(2,2)--(1,2)--(1,1);
					\draw (-2,4)--(-2,3)--(0,3)--(0,4)--(-1,4)--(-1,5)--(-2,5)--(-2,4);
					\draw (1,3)--(1,4)--(2,4)--(2,3)--(1,3)--(1,4);
					\draw (-2,6)--(-2,7)--(-1,7)--(-1,6)--(-2,6)--(-2,7);
					\draw (4,0)--(5,0)--(5,1)--(4,1)--(4,0)--(5,0);
					\draw (5,-1)--(6,-1)--(6,0)--(5,0)--(5,-1)--(6,-1);
					\draw (9,-3)--(9,-4)--(8,-4)--(8,-3)--(9,-3)--(9,-4);
							\draw (5,-3)--(6,-3)--(6,-4);
							\draw (1,1)--(2,1)--(2,0);
							\draw (-2,4)--(-1,4)--(-1,3);
							\node at (3.5,0.5) {$\cdots$};
							\node at (-1.5,5.6) {$\vdots$};
							\node at (5.5,-1.4) {$\vdots$};
							\node at (1.5,2.6) {$\vdots$};
							\node at (7.5,-3.5) {$\cdots$};
							\node at (5.5,0.5) {$\ddots$};
							\node at (10.5,-6.5) {$\cdots$};
							\node at (0.5,3.5) {$\cdots$};
							\node at (8.5,-4.4) {$\vdots$};
							\draw (9,-7)--(9,-6)--(8,-6);
							\draw (11,-6)--(11,-7)--(12,-7)--(12,-6)--(11,-6)--(11,-7);
				\end{tikzpicture}}};
			\node at (6,0) {\scalebox{0.2}{
		\begin{tikzpicture}
					\draw (8,-6)--(8,-7)--(10,-7)--(10,-6)--(9,-6)--(9,-5)--(8,-5)--(8,-6);
					\draw (5,-3)--(5,-4)--(7,-4)--(7,-3)--(6,-3)--(6,-2)--(5,-2)--(5,-3);
					\draw (1,1)--(1,0)--(3,0)--(3,1)--(2,1)--(2,2)--(1,2)--(1,1);
					\draw (-2,4)--(-2,3)--(0,3)--(0,4)--(-1,4)--(-1,5)--(-2,5)--(-2,4);
					\draw (1,3)--(1,4)--(2,4)--(2,3)--(1,3)--(1,4);
					\draw (-2,6)--(-2,7)--(-1,7)--(-1,6)--(-2,6)--(-2,7);
					\draw (4,0)--(5,0)--(5,1)--(4,1)--(4,0)--(5,0);
					\draw (5,-1)--(6,-1)--(6,0)--(5,0)--(5,-1)--(6,-1);
					\draw (9,-3)--(9,-4)--(8,-4)--(8,-3)--(9,-3)--(9,-4);
		\draw (5,-3)--(6,-3)--(6,-4);
		\draw (1,1)--(2,1)--(2,0);
		\draw (-2,4)--(-1,4)--(-1,3);
		\node at (3.5,0.5) {$\cdots$};
		\node at (-1.5,5.6) {$\vdots$};
		\node at (5.5,-1.4) {$\vdots$};
		\node at (1.5,2.6) {$\vdots$};
		\node at (7.5,-3.5) {$\cdots$};
		\node at (5.5,0.5) {$\ddots$};
		\node at (10.5,-6.5) {$\cdots$};
		\node at (0.5,3.5) {$\cdots$};
		\node at (8.5,-4.4) {$\vdots$};
		\draw (9,-7)--(9,-6)--(8,-6);
		\draw (11,-6)--(11,-7)--(12,-7)--(12,-6)--(11,-6)--(11,-7);
	\end{tikzpicture}}};
\node at (3,0) {$\cdots$}; 
				\draw[<-] (1.3,-1.2) to [out=90,in=210] (1.6,-0.6);
				\draw[<-] (-2.7,-1.2) to [out=70,in=210] (-2.4,-0.6);
				\node[right] at (1.5,-0.5) {\small$x^{\alpha_2}y^{\beta_2}$};
				\node[right] at (-2.5,-0.5) {\small$x^{\alpha_1}y^{\beta_1}$};
				\draw[->] (-0.4,1.3) to [out=135,in=0] (-1.1,1.3);
				\node[right] at (-0.5,1.3) {\small$x^{\gamma_2}y^{\delta_2}$};
				
				\draw[->] (5.7,1.3) to [out=135,in=0] (4.9,1.3);
				\node[right] at (5.6,1.3) {\small$x^{\gamma_m}y^{\delta_m}$};
				\node at (-4,-2) {$\Gamma_1$};
				\node at (0,-2) {$\Gamma_2$};
				\node at (6,-2) {$\Gamma_m$};
		\end{tikzpicture}}
		\caption{\ }
		\label{tris}
	\end{figure}
		where just the labels of the boxes we are interested in are shown. The problem is that, in general, the union $\underset{i=1}{\overset{m}{\bigcup}}\Gamma_i$ is not a $G$-sFd, i.e. $\pi_{\N^2}\left(\underset{i=1}{\overset{m}{\bigcup}}\Gamma_i\right)$ is not a sFd. In order to solve this problem, we have to perform some translations, and a possible choice of $G$-sFd is
		 $$\Gamma=\underset{i=1}{\overset{m}{\bigcup}}\overline{\Gamma}_i,$$
		 where
		 $$\overline{\Gamma}_i=x^{k\ssum{j=1}{i-1}\alpha_j}y^{k\ssum{j=1+i}{m}\delta_j}\cdot\Gamma_i\quad\mbox{for }i=1,\ldots,m.$$
		The proof that $\Gamma$ is a $G$-sFd is now an easy check.
	\end{proof}

	As a byproduct of the proof, we also get that any $G$-sFd associated to a toric $G$-constellation has a particular shape.
	\begin{definition}\label{stair} We say that a a connected $G$-sFd $\Gamma$ is a \textit{stair} if 
		$$(m,n)\in\pi_{\N^2}(\Gamma) \Rightarrow (m+1,n+1),(m-1,n-1)\notin \pi_{\N^2}(\Gamma).$$
		
		Moreover,
		\begin{itemize}
			\item a \textit{$G$-stair} is a stair made of $k$ boxes,
			\item an \textit{abstract ($G$-)stair} is an abstract $G$-sFd whose realization in the representation tableau is a ($G$-)stair,
			\item given a stair $\Gamma$, the \textit{(anti)generators} of $\Gamma$ are the boxes positioned in the (top) lower corners of $\Gamma$ (see \Cref{genantigen}),
			\item a substair is any (possibly not connected) subset of a stair.
		\end{itemize}
		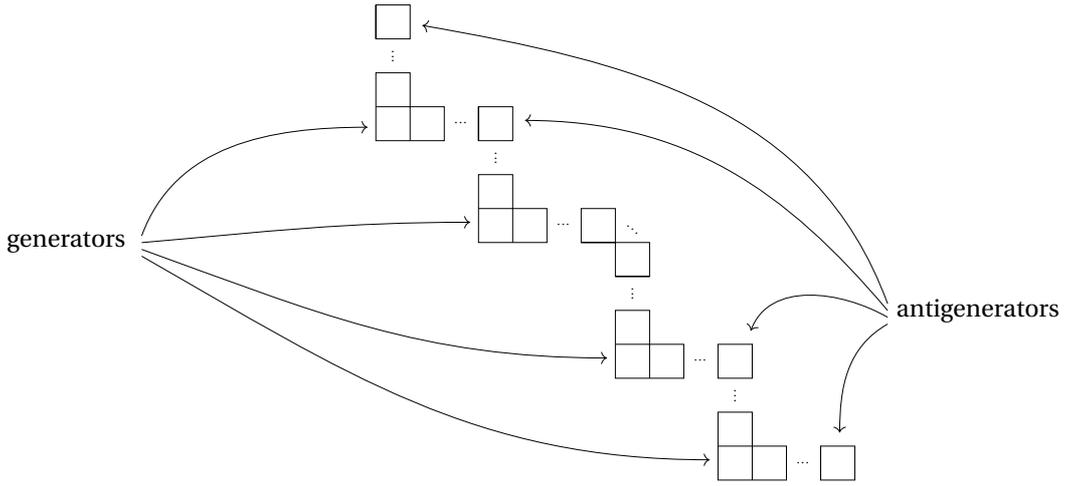
\begin{figure}[H]\scalebox{0.9}{
			\begin{tikzpicture}
				\node at (-4,0) {\scalebox{0.5}{
						\begin{tikzpicture}
					\draw (8,-6)--(8,-7)--(10,-7)--(10,-6)--(9,-6)--(9,-5)--(8,-5)--(8,-6);
					\draw (5,-3)--(5,-4)--(7,-4)--(7,-3)--(6,-3)--(6,-2)--(5,-2)--(5,-3);
					\draw (1,1)--(1,0)--(3,0)--(3,1)--(2,1)--(2,2)--(1,2)--(1,1);
					\draw (-2,4)--(-2,3)--(0,3)--(0,4)--(-1,4)--(-1,5)--(-2,5)--(-2,4);
					\draw (1,3)--(1,4)--(2,4)--(2,3)--(1,3)--(1,4);
					\draw (-2,6)--(-2,7)--(-1,7)--(-1,6)--(-2,6)--(-2,7);
					\draw (4,0)--(5,0)--(5,1)--(4,1)--(4,0)--(5,0);
					\draw (5,-1)--(6,-1)--(6,0)--(5,0)--(5,-1)--(6,-1);
					\draw (9,-3)--(9,-4)--(8,-4)--(8,-3)--(9,-3)--(9,-4);
							\draw (5,-3)--(6,-3)--(6,-4);
							\draw (1,1)--(2,1)--(2,0);
							\draw (-2,4)--(-1,4)--(-1,3);
							\node at (3.5,0.5) {$\cdots$};
							\node at (-1.5,5.6) {$\vdots$};
							\node at (5.5,-1.4) {$\vdots$};
							\node at (1.5,2.6) {$\vdots$};
							\node at (7.5,-3.5) {$\cdots$};
							\node at (5.5,0.5) {$\ddots$};
							\node at (10.5,-6.5) {$\cdots$};
							\node at (0.5,3.5) {$\cdots$};
							\node at (8.5,-4.4) {$\vdots$};
							\draw (9,-7)--(9,-6)--(8,-6);
							\draw (11,-6)--(11,-7)--(12,-7)--(12,-6)--(11,-6)--(11,-7);
				\end{tikzpicture}}};
				
				\draw[<-] (-0.7,-2.8) to [out=90,in=210] (0,-1.2);
				\draw[<-] (-2,-1.3) to [out=70,in=150] (0,-1.1);
				\draw[<-] (-5.3,1.8) to [out=0,in=130] (0,-1);
				\draw[<-] (-6.8,3.2) to [out=-10,in=110] (0,-0.9);
				\node[right] at (0,-1) {antigenerators};
				\node[right] at (-13,0) {generators};
				\draw[->] (-10.9,-0.2) to [out=-30,in=180] (-2.6,-3.2);
				\draw[->] (-10.9,-0.1) to [out=-20,in=180] (-4.1,-1.7);
				\draw[->] (-10.9,0) to [out=5,in=180] (-6.1,0.3);
				\draw[->] (-10.9,0.1) to [out=70,in=180] (-7.6,1.7);
			\end{tikzpicture}}
			\caption{Generators and antigenerators of a stair.}
			\label{genantigen}
		\end{figure}
	\end{definition}
	\begin{remark}\label{orderbox} If $\Calf$ is any torus equivariant $G$-constellation, and $\Gamma_\Calf$ is any $G$-sFd associated to $\Calf$, then $\Gamma_\Calf$ is connected, i.e. it is a $G$-stair, if and only if $\Calf$ is indecomposable, i.e. if it is toric.
		
		In this case we will refer to the upper left box as the first box and we will refer to the lower right box as the last box. In this a way, we provide of a total order the boxes of a $G$-stair and, consequently, we provide of a total order also the irreducible representations of $G$.
	\end{remark}
	\begin{remark} The set of generators of a stair $\Gamma$ corresponds to a minimal set of generators of the $\C[x,y]$-module $M_\Gamma$ associated to $\Gamma$, i.e. $m_1,\ldots,m_s\in M_{\Gamma}$ such that
	\[\Set{m_i+\mm\cdot M_\Gamma\in M_\Gamma/\mm\cdot M_\Gamma|i=1,\ldots,s}
	\]
		is a $\C$-basis of $M_\Gamma/\mm\cdot M_\Gamma$. Antigenerators correspond to one dimensional $\C[x,y]$-submodules of $M_\Gamma$, i.e. they form a $\C$-basis of the so-called socle
		$$(0:_{M_\Gamma}\mm)=\Set{m\in M_\Gamma|\mm\cdot m=0\in M_\Gamma }.$$
		
		Since each irreducible representation of $G$ appears once in a $G$-stair $L$, sometimes, with abuse of notation, we will say that an irreducible representation is a (anti)generator for $L$.
	\end{remark}
	\begin{definition}
		Given a connected $G$-sFd $\Gamma$, we denote  respectively by $\mathfrak{h}(\Gamma)$ and $\mathfrak{w}(\Gamma)$ the \textit{height} and the \textit{width} of $\Gamma$, i.e. the  height and the width of the smallest rectangle in $\N^2$ containing $\pi_{\N^2}(\Gamma)$.
		
		Moreover, the \textit{height} and the \textit{width}, $\mathfrak{h}(\Calf)$ and $\mathfrak{w}(\Calf)$, of a toric $G$-constellation $\Calf$ are  respectively the height and the width of any $G$-stair which represents $\Calf$.
	\end{definition}
 \section{The chamber decomposition of \texorpdfstring{$\Theta$}{} and the moduli spaces \texorpdfstring{$\Calm_C$}{}}
	This section is devoted to the proof of the first main result (\Cref{TEO1}). In the first part of the section we analyze the toric points of $\Calm_C$ and the corresponding $G$-constellations. Then, we show how to construct 1-dimensional families of nilpotent $G$-constellations. Finally, in the last part, we   give the proof of the first main result.
	\subsection{The crepant resolution $\Calm_C$ and its toric points}\label{toric resolution}
	As noticed in \Cref{section22}, the crepant resolution $\Calm_C\xrightarrow{\varepsilon_C}\A^2/G $ does not depend on the chamber $C$, i.e. for all $C,C'\in\Theta^{\gen}$ different chambers, there exists a canonical isomorphism $\varphi:\Calm_C\xrightarrow{\sim}\Calm_C'$ such that the diagram
	\[
\begin{tikzcd}
\Calm_C \arrow{rr}{\varphi} \arrow[swap]{dr}{{\varepsilon}_C} & & \Calm_{C'}\arrow{dl}{\varepsilon_{C'}}\\
& \A^2/G&
\end{tikzcd}
\]
	commutes. 
	
	The varieties $\A^2$, $\A^2/G$ and $\Calm_C$ are toric (see for example \cite[Chapter 10]{COX}  or \cite[Chapter 2]{FULTORI}) and we can rewrite the diagram
		\[
\begin{tikzcd}
 &\A^2 \arrow{d}{\pi}\\
\Calm_C\arrow{r}{\varepsilon_C} &\A^2/G
\end{tikzcd}
\]
	 in terms of fans as follows:
	
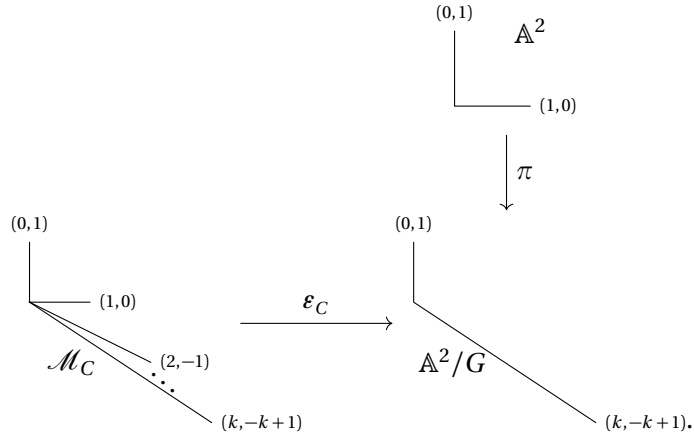
\begin{center}
		\begin{tikzpicture}
			\node at (0.9,0) {\begin{tikzpicture}[scale=0.8]
					\draw[-] (1,0)--(0,0)--(3,-2); 
					\draw[-] (0,1)--(0,0)--(2,-1); 
					\node at (0.7,-1) {$\Calm_C$};
					\node at (2.2,-1.2) {$\ddots$};
					\node[above] at (0,1) {\tiny$(0,1)$};
					\node[right] at (3,-2) {\tiny$(k,-k+1)$};
					\node[right] at (2,-1) {\tiny$(2,-1)$};
					\node[right] at (1,0) {\tiny$(1,0)$};
			\end{tikzpicture}};
			\draw[->] (2,0)--(4,0);
			\node at (6,0) {	\begin{tikzpicture}[scale=0.8]
					\draw[-] (0,1)--(0,0)--(3,-2); 
					\node at (0.65,-1) {$\A^2/G$};
					\node[above] at (0,1) {\tiny$(0,1)$};
					\node[right] at (3,-2) {{\tiny$(k,-k+1)$}.};
			\end{tikzpicture}}; 
			\node at (5.5,3.5) {	\begin{tikzpicture}
					\draw[-] (0,1)--(0,0)--(1,0); 
					\node at (1,1) {$\A^2$};
					\node[right] at (1,0) {\tiny$(1,0)$};
					\node[above] at (0,1) {\tiny$(0,1)$};
			\end{tikzpicture}}; 
			\node[above] at (3,0) {\small$\varepsilon_C$};
			\draw[<-] (5.5,1.5)--(5.5,2.5);
			\node[right] at (5.5,2) {\small$\pi$};
		\end{tikzpicture}
		
	\end{center}
	In particular, $\Calm_C$ is covered by the $k$ toric charts $U_j\cong\A^2$, for $j=1,\ldots,k$, associated to the maximal cones of the  fan for $\Calm_C$ showed above.
	
	Let us identify $\A^2/G$ with the subvariety of $\A^3$
	\[
	\A^2/G=\Set{(\alpha,\beta,\gamma)\in\A^3|\alpha\beta-\gamma^k=0},
	\]
	and let us put (toric) coordinates $a_j,c_j$ on each $U_j$ for $j=1,\ldots,k$. Then, we can encode the diagram above into the following  $k$ diagrams
	\begin{center}
		\begin{tikzpicture}
			\node at (0,0) {$U_j$};
			\node at (3.5,0) {$\A^2/G$};
			\node at (3.5,1) {$\A^2$};
			\draw[->] (0.3,0)--(3,0);
			\draw[|->] (0.6,-0.5)--(1.7,-0.5);
			\draw[->] (3.5,0.8)--(3.5,0.3);
			\draw[|->] (5,0.8)--(5,0.3);
			\node[right] at (3.5,0.6) {\tiny$\pi$};
			\node[above] at (1.5,0) {\tiny$\varepsilon_j$};
			\node at (0,-0.5) {\scriptsize$(a_j,c_j)$};
			\node at (3.5,-0.5) {\scriptsize{$(a_j^{k-j+1}c_j^{k-j},a_j^{j-1}c_j^{j},a_jc_j)$}};
			\node at (5,0) {\scriptsize{$(x^k,y^k,xy)$}};
			\node at (5,1) {\scriptsize{$(x,y)$}};
		\end{tikzpicture}
	\end{center}
for $j=1,\ldots,k$. In this way, we obtain some relations between the coordinates $x,y$ on $\A^2$ and the coordinates $a_j,c_j$ on $U_j$, namely
\begin{equation}\label{toricrelations}
\begin{array}{c}
a_j=x^{j}y^{j-k},\\
c_j=x^{1-j}y^{k-j+1}.
\end{array}
\end{equation}
 Formally, these are relations between regular functions $x,y,a_j,c_j\in \C[a_j,c_j]\underset{\C[x,y]^G}{\otimes}\C[x,y]$ defined on $U_j\underset{\A^2/G}{\times}\A^2= \Spec\left(\left(\C[a_j,c_j]\underset{\C[x,y]^G}{\otimes}\C[x,y]\right)_{\red}\right) $. 
\begin{remark}\label{orderconst}	The toric points of $\Calm_C$ are the origins of the charts $U_j$ and they correspond to the toric $C$-stable $G$-constellations.  Indeed, the torus $\T^2/G$ acts on $\Calm_C$ making it into a toric variety, as described at the beginning of this section, and this toric action coincides with the action
	\[
\begin{tikzcd}[row sep=tiny]
 \T^2\times \Calm_C \arrow{r} & \Calm_C \\
(\sigma ,[\Calf])\arrow[mapsto]{r} & {[}\sigma^*\Calf{]} .
\end{tikzcd}
\] 
	This is a consequence of the universal property of $\Calm _C$. Notice that, outside the exceptional locus of $\Calm_C$, i.e. on the open subset of free orbits, a direct computation is enough to show that the two actions agree.
	
	Hence we have a total order on the toric $G$-constellations over $\Calm_C$, in the sense that the first toric $G$-constellation is the $G$-constellation over the origin of $U_1$, the second one is the $G$-constellation over the origin of $U_2$, and so on.
\end{remark}
\begin{remark}\label{defomations}
	Let $\Gamma$ be a $G$-stair. Then there exists a unique $\sigma\in\Irr(G)$ such that
	$$y\cdot \sigma =0 \mbox{ and } x\cdot \sigma\otimes\rho_{-1}=0$$
	in $\Gamma$. In particular, the representation $\sigma$ corresponds to the first box of $\Gamma$. This representation is important because, if we want to deform in a non-trivial way the $G$-constellation $\Calf_\Gamma$ associated to $\Gamma$ keeping the property of being nilpotent, there are only two ways to do it, namely to modify the $\C[x,y]$-module structure of $\Calf_\Gamma$ by imposing
	$$y\cdot\sigma=\lambda \cdot \sigma\otimes\rho_{-1},\quad \lambda\in\C^*$$
	or
	$$x\cdot\sigma\otimes\rho_{-1}=\mu \cdot \sigma,\quad \mu\in\C^*.$$
	Indeed, if $y\cdot \sigma=\lambda\cdot \sigma\otimes\rho_{-1}$ is not zero, then the nilpotency hypothesis implies 
	$$x\cdot \sigma\otimes\rho_{-1}=\frac{1}{\lambda}xy\cdot \sigma=0,$$
	and the other case is similar. Comparing this with the proof of \Cref{toric basis} one can show that letting $\lambda$ (resp. $\mu$) varying in $\C^*$ all the $G$-constellations so obtained are not isomorphic to each other (as $G$-constellations). In particular $\lambda,\mu$ are coordinates on a chart of $\Calm_C$ around $\Calf_\Gamma$.
\end{remark}
As a consequence of the above remark, we obtain the following lemma.
	\begin{lemma}\label{hwchart} If $\Calf_j$ is the toric $G$-constellation over the origin of the $j$-th chart of some $\Calm_C$, then we have
		$$\mathfrak{h}(\Calf_j)=k-j+1$$
		or, equivalently
		$$\mathfrak{w}(\Calf_j)=j.$$
	\end{lemma}
	\begin{proof}
		Let $\Gamma_j\subset\Calt_G$ be a $G$-stair for $\Calf_j$. In particular, it has the form in \Cref{shapelabel}\begin{figure}[ht]\scalebox{1}{
				\begin{tikzpicture}
					\node at (0,0) {\scalebox{0.3}{
							\begin{tikzpicture}
					\draw (8,-6)--(8,-7)--(10,-7)--(10,-6)--(9,-6)--(9,-5)--(8,-5)--(8,-6);
					\draw (5,-3)--(5,-4)--(7,-4)--(7,-3)--(6,-3)--(6,-2)--(5,-2)--(5,-3);
					\draw (1,1)--(1,0)--(3,0)--(3,1)--(2,1)--(2,2)--(1,2)--(1,1);
					\draw (-2,4)--(-2,3)--(0,3)--(0,4)--(-1,4)--(-1,5)--(-2,5)--(-2,4);
					\draw (1,3)--(1,4)--(2,4)--(2,3)--(1,3)--(1,4);
					\draw (-2,6)--(-2,7)--(-1,7)--(-1,6)--(-2,6)--(-2,7);
					\draw (4,0)--(5,0)--(5,1)--(4,1)--(4,0)--(5,0);
					\draw (5,-1)--(6,-1)--(6,0)--(5,0)--(5,-1)--(6,-1);
					\draw (9,-3)--(9,-4)--(8,-4)--(8,-3)--(9,-3)--(9,-4);
								\draw (5,-3)--(6,-3)--(6,-4);
								\draw (1,1)--(2,1)--(2,0);
								\draw (-2,4)--(-1,4)--(-1,3);
								\node at (3.5,0.5) {$\cdots$};
								\node at (-1.5,5.6) {$\vdots$};
								\node at (5.5,-1.4) {$\vdots$};
								\node at (1.5,2.6) {$\vdots$};
								\node at (7.5,-3.5) {$\cdots$};
								\node at (5.5,0.5) {$\ddots$};
								\node at (10.5,-6.5) {$\cdots$};
								\node at (0.5,3.5) {$\cdots$};
								\node at (8.5,-4.4) {$\vdots$};
								\draw (9,-7)--(9,-6)--(8,-6);
								\draw (11,-6)--(11,-7)--(12,-7)--(12,-6)--(11,-6)--(11,-7);
					\end{tikzpicture}}};
					\draw[<-] (2,-1.8) to [out=90,in=210] (2.3,-1.2);
					\node[right] at (-1.1,1.9){\small$x^{\alpha}y^{\beta}$};
					\draw[->] (-1,1.9) to [out=135,in=0] (-1.8,2);
					\node[right] at (2.2,-1.1)  {\small$x^{\gamma}y^{\delta}$};
					
			\end{tikzpicture}}
			\caption{\ }
			\label{shapelabel}
		\end{figure}
		where just the labels of the boxes we are interested in are shown. Recall, from \Cref{section 32}, that, if we write the skew Ferrers diagram $\pi_{\N^2}(\Gamma_j)=A\smallsetminus B$ as the difference of two Ferrers diagrams $A$ and $B$, then $\Calf_j\cong M_{\Gamma_j}$, where
	$$M_{\Gamma_j}\cong \frac{I_A}{I_A\cap I_B},$$
	and $I_A,I_B$ are as in the proof of \Cref{lemmatec}. Now, if we deform $\Calf_j$ as in \Cref{defomations}, by using the parameters $a_j,c_j\in\C$, we get relations:
	$$x\cdot x^{\gamma}y^{\delta}=a_jx^{\alpha}y^{\beta}$$
	$$y\cdot x^{\alpha}y^{\beta}=c_jx^{\gamma}y^{\delta} $$
	and, the relations \eqref{toricrelations} tell us that
	$$({\gamma-\alpha+1},{\delta-\beta})=({\mathfrak{w}(\Calf)},{-\mathfrak{h}(\Calf)+1})=(j,j-k)\in\N^2$$
	$$({\alpha-\gamma},{\beta-\delta+1})=({-\mathfrak{w}(\Calf)+1},{\mathfrak{h}(\Calf)})=(1-j,k-j+1)
	\in\N^2$$
	which completes the proof.
	\end{proof}

	\begin{remark} \Cref{hwchart} implies that any two toric $G$-constellations of the same height (or equivalently width) cannot belong to the same chamber, i.e. they cannot be $\theta$-stable for the same generic parameter $\theta\in\Theta^{\gen}$ simultaneously.
	\end{remark}
	\subsection{One dimensional families}
	\begin{definition}
		Given a toric $G$-constellation $\Calf$ and its abstract $G$-stair $\Gamma_\Calf$, its \textit{favorite condition} is the stability condition $\theta_\Calf\in\Theta$ defined by:
		$$(\theta_\Calf)_i=\begin{cases}
		-2 &\mbox{if $\rho_i$ is a generator and it is neither the first nor the last box of $\Gamma_\Calf$},\\
		-1 &\mbox{if $\rho_i$ is a generator and it is either the first or the last box of $\Gamma_\Calf$},\\
		2 &\mbox{if $\rho_i$ is an antigenerator and it is neither the first nor the last box of $\Gamma_\Calf$},\\
		1 &\mbox{if $\rho_i$ is an antigenerator and it is either the first or the last box of $\Gamma_\Calf$},\\
		0 & \mbox{otherwise}
		\end{cases}$$
	Moreover, the \textit{cone of good conditions for $\Calf$}, is the cone:
	\[
	\Theta_\Calf=\Set{\theta\in\Theta^{\gen}|\Calf \mbox{ is $\theta$-stable}}.
	\]
	\end{definition}
 \begin{remark}
     It is worth mentioning that the favorite condition $\theta_{\Calf}$ of a toric $G$-constellation $\Calf$ can be understood as the stability condition determined by an appropriate flow on a certain quiver as explained in \cite[\S 6]{INFIRRI}.
 \end{remark}
	\begin{definition} Let $\Gamma$ be a stair and let $\Gamma'\subset \Gamma$ be a substair. We  say that an element $v\in\Gamma'$ is
		\begin{itemize}
			\item a \textit{left internal endpoint} of $\Gamma'$ if there exists $w\in\Gamma\smallsetminus\Gamma'$ such that $x\cdot w=v$ or if $y\cdot v \in\Gamma\smallsetminus\Gamma'$;
			\item a \textit{right internal endpoint} of $\Gamma'$ if there exists $w\in\Gamma\smallsetminus\Gamma'$ such that $y\cdot w=v$ or if $x\cdot v \in\Gamma\smallsetminus\Gamma'$.
		\end{itemize}
		Moreover, we say that
		\begin{itemize}
			\item a left (resp. right) internal endpoint is a \textit{horizontal left (resp. right) cut} if $y\cdot v \in\Gamma\smallsetminus\Gamma'$ (resp. there exists $w\in\Gamma\smallsetminus\Gamma'$ such that  $y\cdot w=v$);
			\item a left (resp. right) internal endpoint is a \textit{vertical left (resp. right) cut} if there exists $w\in\Gamma\smallsetminus\Gamma'$ such that $x\cdot w=v$ (resp. $x\cdot v \in\Gamma\smallsetminus\Gamma'$);
		\end{itemize}
	\end{definition}
	\begin{remark}\label{endpoints}If $\Calf$ is a $G$-constellation and $\Gamma_\Calf$ is a $G$-stair for $\Calf$, then a substair $\Gamma\subset\Gamma_\Calf$ corresponds to a $G$-equivariant $\C[x,y]$-submodule $\Cale_\Gamma$ of $\Calf$ if and only if it has only vertical left cuts and horizontal right cuts. Moreover, if $\Gamma$ is connected and $\theta_\Calf$ is the favorite condition of $\Calf $, then,
		$$\theta_\Calf(\Cale_\Gamma)=\begin{cases}
			1&\mbox{if $\Gamma$ has one internal endpoint},\\
			2&\mbox{if $\Gamma$ has two internal endpoints}.
		\end{cases}$$ 
	\end{remark}
\begin{remark}\label{ororverver} Let $\Calf$ be a toric $G$-constellation with abstract $G$-stair $\Gamma_\Calf$ and let $\Cale<\Calf$ be a subrepresentation, i.e. a $G$-invariant linear subspace, whose substair $\Gamma_\Cale\subset\Gamma_\Calf$ is connected. Then, if $\Gamma_\Cale$ has two horizontal cuts or two vertical cuts and $\theta_\Calf$ is the favorite condition of $\Calf$, we have
	$$\theta_\Calf(\Cale)=0.$$
	
\end{remark}
\begin{remark}\label{propfavorite} The following properties are easy to check for a toric $G$-constellation $\Calf$:
\begin{itemize}
	\item favorite conditions are never generic,
	\item the $G$-constellation $\Calf$ is $\theta_\Calf$-stable,
	\item there exist generic conditions $\theta \in \Theta^{\gen}$ such that $\Calf$ is $\theta $-stable, i.e. the cone of good conditions $\Theta_\Calf$ is not empty.
\end{itemize}
Moreover, given a chamber $C$, we have:
$$C=\underset{[\Calf]\in\Calm_C}{\bigcap}\Theta_\Calf.$$
 For example, one can prove the third property using the openness of the nonempty set $\set{\theta \in\Theta | \Calf\mbox{ is strictly $\theta$-stable}} $ and the denseness of $\Theta^{\gen}$. However, we give here an alternative proof of this fact as  in what follows   we shall need a similar argument. 

Let $\rho_i$ be any irreducible representation, we  denote by $\Calf_{\rho_i}$ the $G$-equivariant $\C[x,y]$-submodule of $\Calf$ generated by $\rho_i$ and, we denote by $\Gamma_{\rho_i}\subset\Gamma _\Calf$ the abstract substair and $G$-stair corresponding to $\Calf_{\rho_i}$ and $\Calf$ respectively.

Consider an $\varepsilon\in\Theta$ with the following properties:
$$\begin{cases}
	\varepsilon_i=0 & \mbox{if }\rho_i\mbox{ is an antigenerator},\\
	\varepsilon_i<0 & \mbox{if }\rho_i\mbox{ is neither a generator nor an antigenerator},\\
	\varepsilon_i=-\ssum{\rho_j\in(\Gamma_{\rho_i}\smallsetminus\rho_i)}{ }\varepsilon_j & \mbox{if }\rho_i\mbox{ is a generator},\\
	\ssum{\mbox{\tiny $\rho_i$ generator}}{ }\varepsilon_i<1.
\end{cases}$$
Then, for any subrepresentation $\Cale<\Calf$, we have $$\varepsilon(\Cale)>-\ssum{\mbox{\tiny $\rho_i$ generator}}{ }\varepsilon_i>-1.$$

Hence, the $G$-constellation $\Calf$ is $(\theta_\Calf+\varepsilon)$-stable. Indeed, \Cref{endpoints} implies that, given an indecomposable proper $G$-equivariant $\C[x,y] $-submodule we have
$$(\theta_\Calf+\varepsilon)(\Cale)>0.$$
On the contrary, if $\Cale $ is not indecomposable then it is a direct sum of indecomposable components and $(\theta_\Calf+\varepsilon)(\Cale)>0$ follows by the additivity of $\theta_\Calf+\varepsilon$ on direct sums.

We conclude by noticing that $\Theta\smallsetminus\Theta^{\gen}$ is a union of hyperplanes and so, there is at least a choice $\varepsilon\in\Theta$ such that $\theta_\Calf+\varepsilon$ is generic.

We will see in the proof of \Cref{TEO1} that there is an easier way, which does not involve any $\varepsilon$, to prove that $\Theta _\Calf$ is not empty.
\end{remark}
	\begin{definition}\label{deflink} An \textit{abstract linking stair} is an abstract stair made of $2k$ boxes obtained from  an abstract $G$-stair $\Gamma$ in either of the following ways:
		\begin{enumerate}
			\item (\textit{decreasing} linking stair of $\Gamma$) take two copies of $\Gamma $ and make a new abstract stair by gluing the right edge of the last box of one copy to the left edge of the first box of the other copy;
			\item (\textit{increasing} linking stair of $\Gamma$) take two copies of $\Gamma $ and make a new abstract stair by gluing the lower edge of the last box of one copy to the upper edge of the first box of the other copy.
		\end{enumerate}
		A \textit{linking stair} is a realization of an abstract linking stair as a subset of the representation tableau.
	\end{definition}
	\begin{remark}
		An abstract linking stair contains exactly $k$ different abstract $G$-stairs.
	\end{remark}
\begin{prop}\label{propcoppia}
	Let $\Gamma$ be the abstract $G$-stair of a $G$-constellation $\Calf$ and let $L$ be its abstract decreasing linking stair. Consider any $G$-stair $\Gamma' \subset L$ and its associated $G$-constellation $\Calf'$. Then, the following are equivalent:
	\begin{enumerate}
		\item there exists at least a chamber $C$ such that both $\Calf $ and $\Calf'$ belong to $C$, i.e. $\Theta_{\Calf}\cap \Theta_{\Calf'}\not=\emptyset$,
		\item $\mathfrak{h}(\Calf')=\mathfrak{h}(\Calf)-1$,
		\item the substair $\Gamma'\subset L$ has a horizontal left cut.
	\end{enumerate}
	In particular, $ \Calf' $ is the $G$-constellation next to $\Calf$ in $\Calm_C$ as per \Cref{orderconst}.
\end{prop}
\begin{example} \Cref{examplefig} describes the situation via an example. Here, we are considering the $\Z/9\Z $-action  on $\A^2 $ given in \eqref{Zkaction}.
	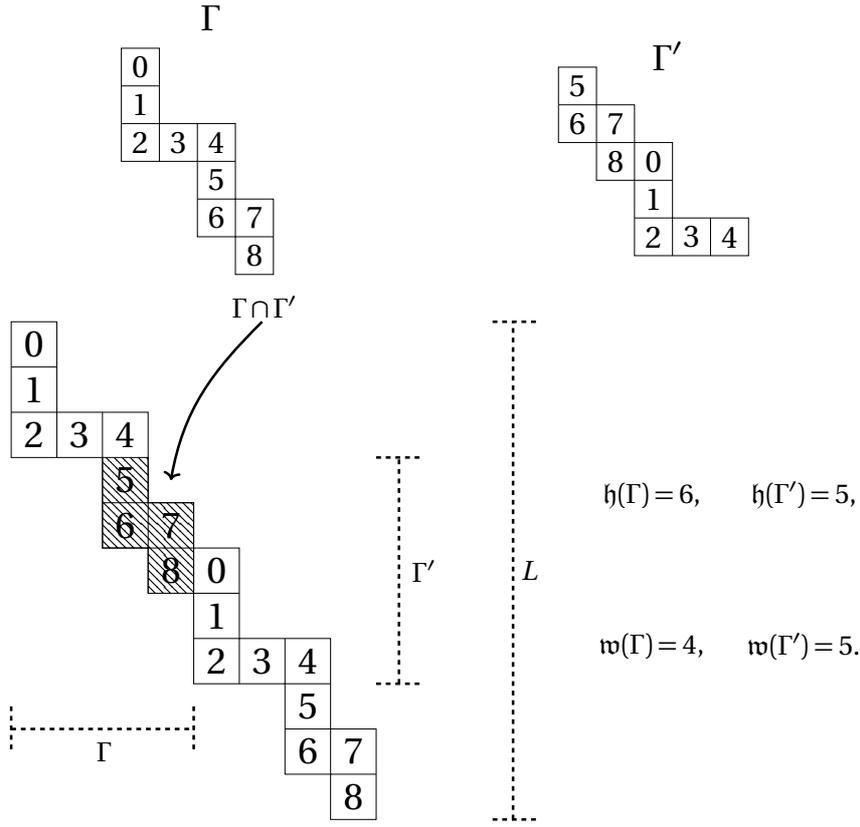
\begin{figure}[H]
		\begin{tikzpicture}\node at (-3,0) {\scalebox{0.5}{
					\begin{tikzpicture}
						\draw (0,-2)--(0,-3)--(2,-3)--(2,-5)--(3,-5)--(3,-6)--(4,-6)--(4,-4)--(3,-4)--(3,-2)--(1,-2)--(1,0)--(0,0)--(0,-2);
						\draw (0,-1)--(1,-1);
						\draw (0,-2)--(1,-2)--(1,-3);
						\draw (2,-2)--(2,-3)--(3,-3);
						\draw (2,-4)--(3,-4)--(3,-5)--(4,-5);
						
						\node at (0.5,-0.5) {\Huge 0};
						\node at (0.5,-1.5) {\Huge 1};
						\node at (0.5,-2.5) {\Huge 2};
						\node at (1.5,-2.5) {\Huge 3};
						\node at (2.5,-2.5) {\Huge 4};
						\node at (2.5,-3.5) {\Huge 5};
						\node at (2.5,-4.5) {\Huge 6};
						\node at (3.5,-4.5) {\Huge 7};
						\node at (3.5,-5.5) {\Huge 8};
						
			\end{tikzpicture}}};
			\node at (3,0) {\scalebox{0.5}{
					\begin{tikzpicture}
						
						\draw (2,-4)--(2,-5)--(3,-5)--(3,-6)--(4,-6)--(4,-4)--(3,-4)--(3,-4);
						\draw (4,-6)--(4,-8)--(7,-8)--(7,-7)--(5,-7)--(5,-5)--(4,-5);
						\draw (2,-4)--(2,-3)--(3,-3)--(3,-4);
						\draw (4,-6)--(5,-6);
						\draw (4,-7)--(5,-7)--(5,-8);
						\draw (2,-4)--(3,-4)--(3,-5)--(4,-5);
						\draw (6,-8)--(6,-7);
						
						\node at (2.5,-3.5) {\Huge 5};
						\node at (2.5,-4.5) {\Huge 6};
						\node at (3.5,-4.5) {\Huge 7};
						\node at (3.5,-5.5) {\Huge 8};
						\node at (4.5,-5.5) {\Huge 0};
						\node at (4.5,-6.5) {\Huge 1};
						\node at (4.5,-7.5) {\Huge 2};
						\node at (5.5,-7.5) {\Huge 3};
						\node at (6.5,-7.5) {\Huge 4};
						
			\end{tikzpicture}}};
			\node at (-2.8,1.9) {{\Large $\Gamma$}};
			\node at (3.2,1.4) {{\Large $\Gamma'$}};
		\end{tikzpicture}
		\begin{tikzpicture}
			\node at (0,0) {\scalebox{0.6}{
					\begin{tikzpicture}
						\draw (0,-2)--(0,-3)--(2,-3)--(2,-5)--(3,-5)--(3,-6)--(4,-6)--(4,-4)--(3,-4)--(3,-2)--(1,-2)--(1,0)--(0,0)--(0,-2);
						\draw (4,-6)--(4,-8)--(6,-8)--(6,-10)--(7,-10)--(7,-11)--(8,-11)--(8,-9)--(7,-9)--(7,-7)--(5,-7)--(5,-5)--(4,-5);
						\draw (0,-1)--(1,-1);
						\draw (0,-2)--(1,-2)--(1,-3);
						\draw (2,-2)--(2,-3);
						\draw (2,-4)--(3,-4)--(3,-5)--(4,-5);
						\draw(4,-6)--(5,-6);
						\draw(4,-7)--(5,-7)--(5,-8);
						\draw(6,-7)--(6,-8)--(7,-8);
						\draw(6,-9)--(7,-9)--(7,-10)--(8,-10);
						\draw[pattern=north west lines] (4,-6)--(4,-4)--(3,-4)--(3,-3)--(2,-3)--(2,-5)--(3,-5)--(3,-6);
						\draw[ultra thick,dashed] (8.5,-8)--(8.5,-3);
						\draw[ultra thick,dashed] (11,-11)--(11,0); 
						\draw[ultra thick,dashed] (8,-8)--(9,-8); 
						\draw[ultra thick,dashed] (8,-3)--(9,-3); 
						\draw[ultra thick,dashed] (11.5,0)--(10.5,0); 
						\draw[ultra thick,dashed] (11.5,-11)--(10.5,-11);
						
						\draw[ultra thick,dashed] (0,-9)--(4,-9); 
						\draw[ultra thick,dashed] (0,-8.5)--(0,-9.5);
						\draw[ultra thick,dashed] (4,-8.5)--(4,-9.5);
						
						\draw[ultra thick,->] (5.5,0) to [out=225 , in =80] (3.5,-3.5); 
						\node at (0.5,-0.5) {\Huge 0};
						\node at (0.5,-1.5) {\Huge 1};
						\node at (0.5,-2.5) {\Huge 2};
						\node at (1.5,-2.5) {\Huge 3};
						\node at (2.5,-2.5) {\Huge 4};
						\node at (2.5,-3.5) {\Huge 5};
						\node at (2.5,-4.5) {\Huge 6};
						\node at (3.5,-4.5) {\Huge 7};
						\node at (3.5,-5.5) {\Huge 8};
						\node at (4.5,-5.5) {\Huge 0};
						\node at (4.5,-6.5) {\Huge 1};
						\node at (4.5,-7.5) {\Huge 2};
						\node at (5.5,-7.5) {\Huge 3};
						\node at (6.5,-7.5) {\Huge 4};
						\node at (6.5,-8.5) {\Huge 5};
						\node at (6.5,-9.5) {\Huge 6};
						\node at (7.5,-9.5) {\Huge 7};
						\node at (7.5,-10.5) {\Huge 8};
						
			\end{tikzpicture}}};
			
			\node at (-2.2,-2.4) {$\Gamma$};
			\node at (-0.1,3.5) {$\Gamma\cap \Gamma'$};
			\node at (2,0) {$\Gamma'$};
			\node at (3.4,0) {$L$};
			\node at (5,1) {$\mathfrak{h}(\Gamma)=6$,};
			\node at (5,-1) {$\mathfrak{w}(\Gamma)=4$,};
			\node at (7,1) {$\mathfrak{h}(\Gamma')=5$,};
			\node at (7,-1) {$\mathfrak{w}(\Gamma')=5$.};
		\end{tikzpicture}
		\caption{The abstract linking stair $L$ of an abstract $G$-stair $\Gamma$ and a substair $\Gamma'$ of $L$ which satisfies the hypotheses of \Cref{propcoppia}.}
		\label{examplefig}
	\end{figure}
\end{example}
	\begin{proof} (of \Cref{propcoppia}).
		We start by introducing some notation.
		
		Let $\Calf,\Calf'$ be two $G$-constellations. Given a proper subrepresentation $\Cale<\Calf$ (resp. $\Cale'<\Calf'$), we  denote by $\Cale'$ (resp. $\Cale$) the corresponding subrepresentation $\Cale'<\Calf'$ (resp. $\Cale<\Calf$). Here, by ``corresponding" we mean that, since $\Cale$ is a subrepresentation of the regular representation $\C[G]$ of an abelian group, it decomposes as a direct sum of distinct indecomposable representations $\Cale\cong\underset{j}{\oplus}\rho_{i_j}$. Then, we denote by $\Cale'$ the subrepresentation of $\Calf'\cong\C[G]$ given by the same summands:
		$$\Cale'\cong\underset{j}{\oplus}\rho_{i_j}.$$
		In particular, for all $ \theta\in\Theta$, the two rational numbers
		$$\theta(\Cale)\ \ \mbox{and}\ \ \theta(\Cale')$$
		are the same. Moreover, we  denote by $\Gamma_{\Cale}\subset\Gamma$ (resp. $\Gamma_{\Cale'}\subset\Gamma'$) the substair associated to $\Cale$ (resp. $\Cale'$).
		
		Notice that, given a proper $G$-equivariant $\C[x,y]$-submodule $\Cale<\Calf$, the subrepresentation $\Cale'$ is not necessarily a $\C[x,y]$-submodule of $\Calf'$. We are now ready to proceed with the proof.
		\begin{itemize}
			\item[(\textit{2})$\Leftrightarrow$(\textit{3})] We omit the easy proof.
			\item[(\textit{1})$\Rightarrow$(\textit{3})] Suppose  by contradiction  that $\Gamma'\subset  L$ has a vertical left cut. Then, by \Cref{endpoints}, the subrepresentation $\Cale_{\Gamma\cap\Gamma'}<\Calf$ is a $\C[x,y]$-submodule because, in $\Gamma$, the substair $\Gamma\cap\Gamma'$ has a vertical left cut by hypothesis and its last box is not internal. At the same time, again by \Cref{endpoints}, $\Cale_{\Gamma\cap\Gamma'}'<\Calf'$ is the complement of a $\C[x,y]$-submodule, because its first box is not internal and it has a vertical right cut. Hence,
			$$C\subset\Theta_{\Calf}\cap\Theta_{\Calf'}\subset \{\theta(\Cale_{\Gamma\cap\Gamma'})>0\}\cap \{-\theta(\Cale_{\Gamma\cap\Gamma'})>0\}=\emptyset,$$
			which contradicts (\textit{1}).
			\item[(\textit{3})$\Rightarrow$(\textit{1})] In order to prove statement (\textit{1}), we need to show that 
			$$\Theta_\Calf\cap\Theta_{\Calf'}\not=\emptyset.$$
		We start by identifying the proper indecomposable $G$-equivariant subsheaves  $\Cale<\Calf$ (resp. $\Cale'<\Calf'$) such that also $\Cale'$ (resp. $\Cale$) is a proper $G$-equivariant subsheaf of $\Calf$ (resp. $\Calf'$).
		
			Let $\Cale'<\Calf'$ be a proper indecomposable $G$-equivariant submodule of $\Calf'$; we consider three different cases.\\
			\textbf{Case 1.} Both the first and the last box of the substair $\Gamma_{\Cale'}\subset\Gamma'$ are internal endpoints. Then, the same happens for $\Gamma_{\Cale}\subset\Gamma$. This is true because $\Gamma$ has a vertical right cut in $L$, by the construction of a decreasing linking stair (see \Cref{deflink}), and hence, the right internal endpoint of $\Gamma_{\Cale'}$ in $\Gamma'$, which is a horizontal cut by \Cref{endpoints}, is different from the right internal endpoint of $\Gamma$ in $L$. Therefore, both internal endpoints of $\Gamma_{\Cale'}$ correspond to internal endpoints of $\Gamma_\Cale$ of the same respective nature. As a consequence, the subrepresentation $\Cale$ is a proper, non necessarily indecomposable, $G$-equivariant submodule of $\Calf$.\\
			\textbf{Case 2.} The substair $\Gamma_{\Cale'}$ has only the vertical left cut in $\Gamma'$, and hence, its last box coincides with the last box of $\Gamma'$. In particular, this box is not the right internal endpoint of $\Gamma $ in $L$. We have to study the nature of the internal endpoints of $\Gamma_\Cale$. Notice first that it is enough to study the right internal endpoint of $\Gamma_\Cale$ because, if $\Gamma_\Cale$ has still left internal endpoint, then it is a vertical left cut. Let $\rho_i$ be the label on the last box of $\Gamma'$, then, the label on the horizontal left cut of $\Gamma'$ (i.e. its first box) is $\rho_{i+1}$. Now, since, by hypothesis (\textit{3}), the box labeled by $\rho_{i+1}$ is a horizontal left cut of $\Gamma'\subset L$, the box labeled by $\rho_i$ in $\Gamma$ has to be a horizontal right cut for the substair $\Gamma_\Cale$. Therefore, $\Gamma_\Cale$ has only vertical left cuts and horizontal right cuts, and so, by \Cref{endpoints}, $\Cale$ is a proper, non necessarily indecomposable, $G$-equivariant submodule.\\
			\textbf{Case 3.} The substair $\Gamma_{\Cale'}\subset \Gamma'$ has only the horizontal right cut, i.e. its first box coincides with the first box of $\Gamma'$. First of all notice that, as for the first analyzed case, the right internal endpoint of $\Gamma_{\Cale'}$ in $\Gamma'$, which is a horizontal cut by hypothesis, is different from the right internal endpoint of $\Gamma$ in $L$, which is vertical by definition of decreasing linking stair. Therefore, the box of $\Gamma$ with the same label as the horizontal right cut of $\Gamma_{\Cale'}$ is an internal endpoint of $\Gamma_\Cale$ and it is a horizontal right cut. Finally, the first box of $\Gamma' $ in $L$ is a left internal endpoint for $\Gamma_\Cale$, and so it is a horizontal left cut by point (\textit{3}) of the statement. As a consequence, $\Gamma_\Cale$ has two horizontal cuts.
			
			In summary, if $\Cale'<\Calf'$ is a proper indecomposable $G$-equivariant submodule of $\Calf'$ such that $\Gamma_{\Cale'}$ has a vertical left cut, then also $\Cale<\Calf$ is a proper, non necessarily indecomposable, $G$-equivariant submodule. While, if $\Gamma_{\Cale'}<\Gamma'$ has only the right horizontal cut, then $\Gamma_\Cale$ has two horizontal cuts.
			
			Following the same logic, if $\Cale<\Calf$ is a proper indecomposable $G$-equivariant submodule of $\Calf$ such that $\Gamma_{\Cale}$ has a horizontal right cut, then also $\Cale'<\Calf'$ is a proper, non necessarily indecomposable, $G$-equivariant submodule. While, if $\Gamma_{\Cale}<\Gamma$ has only the left vertical cut, then $\Gamma_{\Cale'}$ has two vertical cuts.
			
			We are now ready to exhibit a $\theta\in\Theta^{\gen}$ such $\Calf$ and $\Calf'$ are $\theta$-stable. Let $\theta_\Calf$ and $\theta_{\Calf'}$ be the respective favorite conditions for $\Calf $ and $\Calf'$ and let $\theta=\theta_\Calf+\theta_{\Calf'}$ be their sum. Then, both $\Calf$ and $\Calf'$ are $\theta$-stable. Indeed,
			\begin{itemize}
				\item if $\Cale<\Calf$ is a proper indecomposable $G$-equivariant $\C[x,y]$-submodule of $\Calf$ such that also $\Cale'$ is a $\C[x,y]$-submodule of $\Calf'$, then  
				$$\theta(\Cale)=\theta_\Calf(\Cale)+\theta_{\Calf'}(\Cale)=\theta_\Calf(\Cale)+\theta_{\Calf'}(\Cale')>0$$
				follows from the fact that $\Calf $ is $\theta_\Calf$-stable and $\Calf' $ is $\theta_{\Calf'}$-stable (see \Cref{propfavorite});
				\item if $\Cale'<\Calf'$ is a proper indecomposable $G$-equivariant $\C[x,y]$-submodule of $\Calf'$ such that $\Gamma_{\Cale}$ has two horizontal cuts, then $$\theta(\Cale')=\theta_\Calf(\Cale')+\theta_{\Calf'}(\Cale')=\theta_\Calf(\Cale)+\theta_{\Calf'}(\Cale')=\theta_{\Calf'}(\Cale')=1>0$$
				follows from the fact that $\Calf' $ is $\theta_{\Calf'}$-stable (see \Cref{propfavorite}) and from \Cref{endpoints,ororverver};	
				\item if $\Cale<\Calf$ is a proper indecomposable $G$-equivariant $\C[x,y]$-submodule of $\Calf$ such that $\Gamma_{\Cale'}$ has two vertical cuts, then $$\theta(\Cale)=\theta_\Calf(\Cale)+\theta_{\Calf'}(\Cale)=\theta_\Calf(\Cale)+\theta_{\Calf'}(\Cale')=\theta_{\Calf}(\Cale)=1>0$$
				follows from the fact that $\Calf $ is $\theta_{\Calf}$-stable (see \Cref{propfavorite}) and from \Cref{endpoints,ororverver};
				\item if $\Cale<\Calf$ (resp. $\Cale'<\Calf'$) is a proper decomposable $G$-equivariant $\C[x,y]$-submodule, then 
				$$\theta(\Cale)>0$$
				follows by applying the previous points to the indecomposable components of $\Cale$ and from the additivity of $\theta$.
			\end{itemize}
		
		The last issue here is that, in general, such $\theta$ is not generic, i.e. 
		$$\theta \in \overline{\Theta_\Calf\cap\Theta_{\Calf'}}\smallsetminus \Theta_\Calf\cap\Theta_{\Calf'}.$$ In order to solve this problem, we can perturb $\theta_\Calf$ and $\theta_{\Calf'}$ the same way as as we did in \Cref{propfavorite} thus obtaining a generic $\widetilde{\theta}\in\Theta_{\Calf}\cap\Theta_{\Calf'}$. Consider the stability conditions $\varepsilon,\varepsilon'\in\Theta$ defined as follows:
		$$\begin{cases}
			\begin{matrix*}[l]
				\varepsilon_i=0  & \mbox{if }\rho_i\mbox{ is an antigenerator of }\Gamma_\Calf  ,\\
				\varepsilon_i'=0 & \mbox{if }\rho_i\mbox{ is an antigenerator of }\Gamma_{\Calf'},\\
				\varepsilon_i<0 & \mbox{if }\rho_i\mbox{ is neither a generator nor an antigenerator of }\Gamma_\Calf,\\
				\varepsilon_i'<0 & \mbox{if }\rho_i\mbox{ is neither a generator nor an antigenerator of }\Gamma_{\Calf'},\\
				\varepsilon_i=-\ssum{\rho_j\in(\Gamma_{\rho_i}\smallsetminus\rho_i)}{ }\varepsilon_j & \mbox{if }\rho_i\mbox{ is a generator of }\Gamma_\Calf,\\
				\varepsilon_i'=-\ssum{\rho_j\in(\Gamma_{\rho_i}'\smallsetminus\rho_i)}{ }\varepsilon_j' & \mbox{if }\rho_i\mbox{ is a generator of }\Gamma_{\Calf'},
			\end{matrix*}\\
			\ssum{\mbox{\tiny $\rho_i$ generator of }\Gamma_\Calf}{ }\varepsilon_i+\ssum{\mbox{\tiny $\rho_i$ generator of }\Gamma_{\Calf'}}{ }\varepsilon_i'<1,
		\end{cases}$$
	where, as in \Cref{propfavorite}, $\Gamma_{\rho_i}\subset\Gamma$ (resp. $\Gamma_{\rho_i}'\subset\Gamma'$) is the substair associated to the $\C[x,y]$-submodule of $\Calf $ (resp. $ \Calf '$) generated by the irreducible subrepresentation $\rho_i$.
	
	Now, if
	$$\widetilde{\theta}=(\theta_{\Calf}+\varepsilon)+(\theta_{\Calf'}+\varepsilon')$$
	then $\Calf$ and $\Calf'$ are $\widetilde{\theta}$-stable, and $\varepsilon$ and $\varepsilon'$ can be chosen in such a way that $\widetilde{\theta}$ is generic. As a consequence $\Theta_\Calf\cap\Theta_{\Calf'}\not=\emptyset$.
		\end{itemize}	
	\end{proof}
	
	We will see, in the proof of \Cref{TEO1}, that there is an easier way to prove that $\Theta _\Calf\cap\Theta_{\Calf'}$ is not empty. By following the same logic, one can prove a similar statement for the increasing linking stairs.
	
	\begin{prop}\label{propcoppia1}
		Let $\Gamma$ be the abstract $G$-stair of a $G$-constellation $\Calf$ and let $L$ be its abstract increasing linking stair. Consider any $G$-stair $\Gamma'\subset L $ and its associated $G$-constellation $\Calf'$. Then, the following are equivalent:
		\begin{enumerate}
			\item there exists at least a chamber $C$ such that both $\Calf $ and $\Calf'$ belong to $C$, i.e. $\Theta_{\Calf}\cap \Theta_{\Calf'}\not=\emptyset$,
			\item $\mathfrak{h}(\Calf')=\mathfrak{h}(\Calf)+1$,
			\item the substair $\Gamma'\subset L$ has a  vertical right cut.
		\end{enumerate}
	In particular, $ \Calf $ is the $G$-constellation next to $\Calf'$ in $\Calm_C$ in the sense of \Cref{orderconst}.
	\end{prop}
\subsection{Counting the chambers}
	\begin{remark}\label{minoreouguale} Propositions \ref{propcoppia} and \ref{propcoppia1} provide a way to build 1-dimensional families of nilpotent $G$-constellations. In particular, each of this families corresponds to some exceptional line in some $\Calm_C$. Moreover, the two gluings described in the definition of linking stair are  nothing but the two possible ways of deforming a toric $G$-constellation keeping the property of being nilpotent described in \Cref{defomations}. This implies that the families coming from \Cref{propcoppia} and \Cref{propcoppia1} are exactly the 1-dimensional families of nilpotent $G$-constellations appearing in the moduli spaces $\Calm_C$.
		
		An easy combinatorial computation tells us that the maximum number of chambers is $k!$. Indeed, if we start by a $G$-constellation $\Calf_1$ of maximum height $\mathfrak{h}(\Calf)=k$, i.e. $\Calf_1$ has one of the $k$ abstract $G$-stairs shown in \Cref{maxheight},
		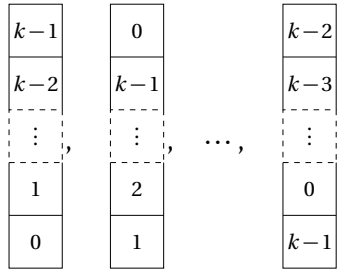
\begin{figure}[ht]$\begin{matrix}\scalebox{0.7}{
					\begin{tikzpicture}
						\node at (0.5,0.5) {\large
							0};
						\node at (0.5,1.5) {\large 1};
						\node at (0.5,2.6) {\large $ \vdots $};
						\node at (0.5,3.5) {\large $k-2$};
						\node at (0.5,4.5) {\large $k-1$};
						\draw (0,3)--(0,5)--(1,5)--(1,3);
						\draw (0,4)--(1,4);
						\draw (0,2)--(0,0)--(1,0)--(1,2);
						\draw (0,1)--(1,1);
						\draw[dashed] (0,2)--(1,2)--(1,3)--(0,3)--(0,2);
				\end{tikzpicture}}
			\end{matrix},\quad \begin{matrix}\scalebox{0.7}{
					\begin{tikzpicture}
						\node at (0.5,0.5) {\large 1};
						\node at (0.5,1.5) {\large 2};
						\node at (0.5,2.6) {\large $ \vdots $};
						\node at (0.5,3.5) {\large $k-1$};
						\node at (0.5,4.5) {\large $0$};
						\draw (0,3)--(0,5)--(1,5)--(1,3);
						\draw (0,4)--(1,4);
						\draw (0,2)--(0,0)--(1,0)--(1,2);
						\draw (0,1)--(1,1);
						\draw[dashed] (0,2)--(1,2)--(1,3)--(0,3)--(0,2);
				\end{tikzpicture}}
			\end{matrix},\quad \cdots,\quad \begin{matrix}\scalebox{0.7}{
					\begin{tikzpicture}
						\node at (0.5,0.5) {\large $k-1$};
						\node at (0.5,1.5) {\large 0};
						\node at (0.5,2.6) {\large $ \vdots $};
						\node at (0.5,3.5) {\large $k-3$};
						\node at (0.5,4.5) {\large $k-2$};
						\draw (0,3)--(0,5)--(1,5)--(1,3);
						\draw (0,4)--(1,4);
						\draw (0,2)--(0,0)--(1,0)--(1,2);
						\draw (0,1)--(1,1);
						\draw[dashed] (0,2)--(1,2)--(1,3)--(0,3)--(0,2);
				\end{tikzpicture}}
			\end{matrix}$
			\caption{The abstract $G$-stairs of maximum height.}
			\label{maxheight}
		\end{figure}
		we can construct toric $G$-constellations $\Calf_2,\ldots,\Calf_k$ with respective abstract $G$-stairs $\Gamma_{j}$ for $j=2,\ldots,k$ by recursively applying the prescriptions in \Cref{propcoppia}. Precisely, for any $j>1$, each $\Gamma_j$ is a connected substair, with horizontal left cut, of the decreasing linking stair of $\Gamma_{j-1}$.
		
		To conclude that the maximum number of chambers is $k!$, we notice that the $j$-th time that we apply \Cref{propcoppia} there are $k-j$ possible $G$-stairs with horizontal left cut in the decreasing linking stair of the abstract $G$-stair of $\Calf_j$. 
	\end{remark}
	\begin{theorem}\label{TEO1}If $G\subset \SL(2,\C)$ is a finite abelian subgroup of cardinality $k=|G|$, then the space of generic stability conditions $\Theta^{\gen}$ is the disjoint union of $k!$ chambers.
	\end{theorem}
\begin{proof} 
		 It is enough to show that, if $\Calf_1,\ldots,\Calf_k$ are as in \Cref{minoreouguale}, then there exists a chamber 
	 $$C=\Theta_{\Calf_1}\cap\Theta_{\Calf_2}\cap\cdots\cap\Theta_{\Calf_k}\not=\emptyset,$$
	 such that $\Calf_j$ is $C$-stable for all $j=1,\ldots,k$. We claim that, if, for all $j=1,\ldots,k$, the favorite condition of $\Calf_j$ is $\theta_{\Calf_j}$, then
	 $$\theta=\ssum{j=1}{k}\theta_{\Calf_j}\in C.$$A priori, in order to prove the claim, we need to show both that $\theta$ is generic and that every $\Calf_j$ is $\theta$-stable. In fact, it is enough to show just that every $\Calf_j$ is $\theta$-stable, because this implies that $\Calm_\theta$ has $k$ torus fixed-points and, as a consequence, that $\theta $ is generic.
	 
	 Let $\Cale_j<\Calf_j$ be a proper $G$-equivariant indecomposable $\C[x,y]$-submodule of $\Calf_j$ with substair $\Gamma_{\Cale_j}\subset\Gamma_{\Calf_j}$. Suppose also that $\Cale_j=\underset{s=m}{\overset{n}{\bigoplus}}\rho_s$, where $0\le m\le n\le k-1$. We denote by $\Cale_i$, for $i=1,\ldots,j-1,j+1,\ldots,k$, the subrepresentation of $\Calf_i$ corresponding to $\Cale_j$, i.e. $$\Cale_i=\underset{s=m}{\overset{n}{\bigoplus}}\rho_s,\ \ \forall i=1,\ldots,j-1,j+1,\ldots,k.$$
	 Notice that
	 \begin{itemize}
	 	\item if $\Gamma_{\Cale_{j+1}}$ has two vertical cuts, then $\Gamma_{\Cale_{i}}$ has two vertical cuts for every $i>j+1$;  
	 	\item if $\Gamma_{\Cale_{j-1}}$ has two horizontal cuts, then $\Gamma_{\Cale_{i}}$ has two horizontal cuts for every $i<j-1$.
	 \end{itemize}
	 This is true because every time we increase (resp. decrease) the index $i$, we perform a horizontal left (resp. vertical right) cut in the decreasing (resp. increasing) linking stair which does not affect the vertical left (resp. horizontal right) cut of $\Gamma_{\Cale_{j+1}}$ (resp. $\Gamma_{\Cale_{j-1}}$).
	 
	 Hence, for all $i =1,\ldots,j-1,j+1,\ldots,k$, we have $\theta_{\Calf_i}(\Cale_j)\ge0 $ and, as a consequence
	 $$\theta(\Cale_j)=\left(\theta_{\Calf_j}+\ssum{i\not=j}{ }\theta_{\Calf_i}\right)(\Cale_j)>0.$$
	 
\end{proof}

\begin{remark} The proof of \Cref{TEO1} provides an alternative way to prove that
	$$\Theta_{\Calf}\not=\emptyset$$
	in \Cref{propfavorite} and, that
	$$\Theta_{\Calf}\cap\Theta_{\Calf'}\not=\emptyset$$
	 in the last part of the third point of the proof of \Cref{propcoppia}.
	 
	 For example, let $\Calf$ be a toric $G$-constellation with abstract $G$-stair of height $\mathfrak{h}(\Calf)=j$. We construct $\Calf_1,\ldots,\Calf_{j-1},\Calf_{j+1},\ldots,\Calf_k$ by recursively applying Propositions \ref{propcoppia} and \ref{propcoppia1}, i.e. 
	 \begin{itemize}
	 	\item if $i>j$, then $\Calf_i$ has, as $G$-stair, a $G$-substair, with a horizontal left cut, of the decreasing linking stair of $\Calf_{i-1}$,
	 	\item if $i<j$, then $\Calf_i$ has, as $G$-stair, a $G$-substair, with a vertical right cut, of the increasing linking stair of $\Calf_{i+1}$.
	 \end{itemize}
 Then, if $\theta = \theta_\Calf+\ssum{ }{ }\theta_{\Calf_i}$ is the sum of all favorite conditions, we have $\theta\in\Theta_\Calf.$
\end{remark}
	\section{Simple chambers}
	In this section we firstly introduce the notion of chamber stair. Roughly speaking, it is a stair that encodes all the data needed to reconstruct a chamber. Then, we define simple chambers, which are a particular kind of chambers with the property that any toric $G$-constellation belongs to at least one of them. Finally, we prove that there are exactly $k\cdot 2^{k-2}$ simple chambers.
	\begin{remark}Given a chamber $C\subset \Theta^{\gen}$ we can make a stair $\Gamma_C$ out of it and  and we say that $\Gamma_C$ is the chamber stair of $C$. 
		
		Let $\Calf_1,\ldots,\Calf_k$ be the toric $G$-constellations in $\Calm_C$. As explained in \Cref{propcoppia} (resp. \Cref{propcoppia1}), the abstract $G$-stairs $\Gamma_j,\Gamma_{j+1}$ of two consecutive $G$-constellations $\Calf_j,\Calf_{j+1}$ are substairs of the same stair $L$, namely the decreasing linking stair of $\Gamma_j$ (resp. the increasing linking stair of $\Gamma_{j+1}$). Moreover they have non-empty intersection in $L$.
		
		Now, if $\Gamma_1,\ldots,\Gamma_k$ are the respective abstract $G$-stairs of $\Calf_1,\ldots,\Calf_k$, we can construct a new abstract stair $\Gamma_C$ by gluing consecutive abstract $G$-stairs along their common parts.
	\end{remark}
	
	\begin{definition}\label{chamberstair}
		The \textit{abstract chamber stair of $C$} or the \textit{abstract $C$-stair} is the abstract stair $\Gamma_C$ obtained as described above.
	\end{definition}
\begin{example} Consider the case $G\cong \Z/5\Z$. \Cref{CHAMBER} explains how to build an abstract $C$-stair starting from the abstract $G$-stairs of the $G$-constellations in some chamber $C$.
	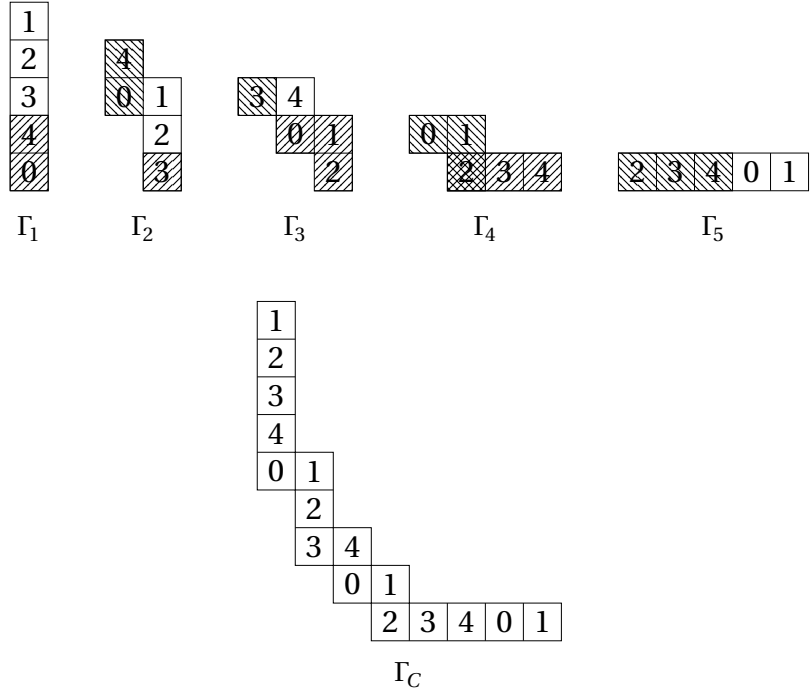
\begin{figure}[ht]
\scalebox{1}{\begin{tikzpicture}
		\node at (0,0) {
			\scalebox{0.5}{\begin{tikzpicture}
					\draw (0,0)--(0,5)--(1,5)--(1,0)--(0,0);
					\draw (0,1)--(1,1);
					\draw (0,2)--(1,2);
					\draw (0,3)--(1,3);
					\draw (0,4)--(1,4);
					\draw[pattern=north east lines] (0,1)--(0,0)--(1,0)--(1,2)--(0,2)--(0,1);
					
					\node at (0.5,0.5) {\Huge 0};
					\node at (0.5,1.5) {\Huge 4};
					\node at (0.5,2.5) {\Huge 3};
					\node at (0.5,3.5) {\Huge 2};
					\node at (0.5,4.5) {\Huge 1};
		\end{tikzpicture}}};
	\node at (1.5,-0.25) {
		\scalebox{0.5}{\begin{tikzpicture}
				\draw (2,0)--(0,0)--(0,2)--(1,2)--(1,1)--(2,1)--(2,-2)--(1,-2)--(1,0)--(0,0);
				\draw (0,1)--(1,1);
				\draw (1,-1)--(2,-1);
				\draw[pattern=north west lines] (0,0)--(0,2)--(1,2)--(1,0)--(0,0);
				\draw[pattern=north east lines] (1,-2)--(2,-2)--(2,-1)--(1,-1)--(1,-2);
				
				\node at (0.5,0.5) {\Huge 0};
				\node at (0.5,1.5) {\Huge 4};
				\node at (1.5,-1.5) {\Huge 3};
				\node at (1.5,-0.5) {\Huge 2};
				\node at (1.5,0.5) {\Huge 1};
	\end{tikzpicture}}};

\node at (3.5,-0.5) {
	\scalebox{0.5}{\begin{tikzpicture}
			\draw (0,0)--(1,0)--(1,-1)--(2,-1)--(2,-2)--(3,-2)--(3,0)--(2,0)--(2,1)--(0,1)--(0,0);
			\draw (1,1)--(1,0)--(2,0);
			\draw (2,0)--(2,-1)--(3,-1);
			\draw[pattern=north west lines] (0,0)--(0,1)--(1,1)--(1,0)--(0,0)--(0,1);
			\draw[pattern=north east lines] (2,-2)--(3,-2)--(3,0)--(1,0)--(1,-1)--(2,-1);
			
			\node at (0.5,0.5) {\Huge 3};
			\node at (2.5,-1.5) {\Huge 2};
			\node at (2.5,-0.5) {\Huge 1};
			\node at (1.5,-0.5) {\Huge 0};
			\node at (1.5,0.5) {\Huge 4};
\end{tikzpicture}}};
\node at (6,-0.75) {
	\scalebox{0.5}{\begin{tikzpicture}
			\draw (1,0)--(1,-1);
			\draw (3,-2)--(3,-1);
			\draw (1,-1)--(2,-1)--(2,-2);
			\draw (1,0)--(2,0)--(2,-1)--(4,-1)--(4,-2)--(1,-2)--(1,-1)--(0,-1)--(0,0)--(1,0);
			\draw[pattern=north west lines] (0,0)--(2,0)--(2,-2)--(1,-2)--(1,-1)--(0,-1)--(0,0);
			\draw[pattern=north east lines] (1,-1)--(4,-1)--(4,-2)--(1,-2);
			
			\node at (2.5,-1.5) {\Huge 3};
			\node at (1.5,-1.5) {\Huge 2};
			\node at (1.5,-0.5) {\Huge 1};
			\node at (0.5,-0.5) {\Huge 0};
			\node at (3.5,-1.5) {\Huge 4};
\end{tikzpicture}}};
\node at ( 9 ,-1) {
	\scalebox{0.5}{\begin{tikzpicture}
			\draw (1,0)--(5,0)--(5,1)--(0,1)--(0,0)--(1,0);
			\draw (1,0)--(1,1);
			\draw (2,0)--(2,1);
			\draw (3,0)--(3,1);
			\draw (4,0)--(4,1);
			\draw[pattern=north west lines] (0,0)--(3,0)--(3,1)--(0,1)--(0,0);
			
			\node at (0.5,0.5) {\Huge 2};
			\node at (1.5,0.5) {\Huge 3};
			\node at (2.5,0.5) {\Huge 4};
			\node at (3.5,0.5) {\Huge 0};
			\node at (4.5,0.5) {\Huge 1};
\end{tikzpicture}}};
\node at (0,-1.75) {$\Gamma_1$};
\node at (1.5,-1.75) {$\Gamma_2$};
\node at (3.5,-1.75) {$\Gamma_3$};
\node at (6,-1.75) {$\Gamma_4$};
\node at (9,-1.75) {$\Gamma_5$};
\end{tikzpicture}}\\$\ $\\
\begin{tikzpicture}
	\node at (0,0) {
		\scalebox{0.5}{\begin{tikzpicture}
				\draw (0,-4)--(0,-5)--(1,-5)--(1,-7)--(2,-7)--(2,-8)--(3,-8)--(3,-9)--(8,-9)--(8,-8)--(4,-8)--(4,-7)--(3,-7)--(3,-6)--(2,-6)--(2,-4)--(1,-4)--(1,0)--(0,0)--(0,-4);
				\draw (0,-1)--(1,-1);
				\draw (0,-2)--(1,-2);
				\draw (0,-3)--(1,-3);
				\draw (0,-4)--(1,-4)--(1,-5)--(2,-5);
				\draw (1,-6)--(2,-6)--(2,-7)--(3,-7)--(3,-8)--(4,-8)--(4,-9);
				\draw (5,-8)--(5,-9);
				\draw (6,-8)--(6,-9);
				\draw (7,-8)--(7,-9);
				\node at (0.5,-0.5) {\Huge 1};
				\node at (0.5,-1.5) {\Huge 2};
				\node at (0.5,-2.5) {\Huge 3};
				\node at (0.5,-3.5) {\Huge 4};
				\node at (0.5,-4.5) {\Huge 0};
				\node at (1.5,-4.5) {\Huge 1};
				\node at (1.5,-5.5) {\Huge 2};
				\node at (1.5,-6.5) {\Huge 3};
				\node at (2.5,-6.5) {\Huge 4};
				\node at (2.5,-7.5) {\Huge 0};
				\node at (3.5,-7.5) {\Huge 1};
				\node at (3.5,-8.5) {\Huge 2};
				\node at (4.5,-8.5) {\Huge 3};
				\node at (5.5,-8.5) {\Huge 4};
				\node at (6.5,-8.5) {\Huge 0};
				\node at (7.5,-8.5) {\Huge 1};
	\end{tikzpicture}}};
\node at (0,-2.7) {$\Gamma_C$};
\end{tikzpicture}
\caption{The abstract $C$-stair $\Gamma_C$ is obtained by gluing, along their common part, the abstract $\Z/5\Z$-stairs $\Gamma_i$ and $\Gamma_{i+1}$ for $i=1,\ldots,4$.}
\label{CHAMBER}
	\end{figure}
In particular, we have glued the boxes \scalebox{0.4}{\begin{tikzpicture}
		\draw[pattern=north east lines] (0,0)--(1,0)--(1,1)--(0,1)--(0,0)--(1,0);
\end{tikzpicture}} of an abstract $G$-stair with the boxes \scalebox{0.4}{\begin{tikzpicture}
\draw[pattern=north west lines] (0,0)--(1,0)--(1,1)--(0,1)--(0,0)--(1,0);
\end{tikzpicture}} of the next abstract $G$-stair.
\end{example}
	\begin{definition} A \textit{chamber stair associated to $C$} or a \textit{$C$-stair} is any realization $\widetilde{\Gamma}_C$ of the abstract chamber stair $\Gamma_C$ associated to $C$ as a subset of the representation tableau.
	\end{definition}
	\begin{remark}\label{CONCLUSION}
	Let $C\subset \Theta^{\gen}$ be a chamber and let $\Gamma_C\subset \Calt_G$ be a $C$-stair. Consider a $G$-stair $\Gamma\subset \Gamma_C$ of width $\mathfrak{w}(\Gamma)=j$ and the associated $G$-constellation $\Calf_\Gamma$. Let us also denote by $b,b'\in\Gamma$ the first and the last box of $\Gamma$. Suppose that $\Calf_\Gamma$ is not $C$-stable. Then, there are two consecutive $C$-stable $G$-constellations $\Calf$ and $\Calf'$ with associated respective $G$-stairs $\Gamma_{\Calf},\Gamma_{\Calf'}\subset\Gamma_C$ such that $b\in\Gamma_{\Calf}$ and $b'\in\Gamma_{\Calf'}$.
	
	Therefore, $\Gamma$ is a substair of both the decreasing linking stair $L$ of $\Gamma_\Calf$ and the increasing linking stair $L'$ of $\Gamma_{\Calf'}$. In particular, as a consequence of Propositions \ref{propcoppia} and \Cref{propcoppia1}, one and only one between the following two possibilities must occur, namely:
	\begin{equation}
	\label{conditions}
	    \begin{array}{l}
	       \mathfrak{w}(\Calf)=j-1,  \mathfrak{w}(\Calf')=j,   \mbox{ and $b$ (resp. $b'$) is a left (resp. right) horizontal cut of $\Gamma$ in $L$},  \\
	       \mathfrak{w}(\Calf)=j ,\mathfrak{w}(\Calf')=j+1,  \mbox{ and $b$ (resp. $b'$) is a right (resp. left) vertical cut of $\Gamma$ in $L$}.
	    \end{array}
	\end{equation}
	On the other hand, again as a consequence of \Cref{propcoppia} and \Cref{propcoppia1}, if $\Calf_\Gamma$ is $C$-stable, none of the conditions in \eqref{conditions} can hold true, and in this case $\Gamma$ has horizontal left cut and vertical right cut in $\Gamma_C$.
	
	Summing up, if $\Gamma\subset\Gamma_C$ is a connected $G$-substair associated to a toric $G$-constellation $\Calf_\Gamma$ then only the following two cases can occur:
	\begin{itemize}
	    \item the $G $-constellation $\Calf_\Gamma$ is $C$-stable and $\Gamma$ has a horizontal left cut and a vertical right cut, or
	    \item the $G $-constellation $\Calf_\Gamma$ is not $C$-stable and $\Gamma$ has two horizontal cuts or two vertical cuts.
	\end{itemize}
	\end{remark}
\begin{remark}\label{unicCstair} Different chambers have different abstract chamber stairs. 
	
	First, recall from \Cref{CONCLUSION} that, as per \Cref{propcoppia}, the $G$-stair of any toric $C$-stable $ G$-constellation has a vertical right cut in the $C$-stair and a horizontal right cut in the decreasing linking stair of the previous $G$-constellation.
	
	Suppose that two chambers $C$ and $C'$ have the same abstract chamber stair $\Gamma$. In particular, from the construction of abstract chamber stairs, it follows that $C$ and $C'$ have the same first (in the sense of \Cref{orderconst}) toric $G$-constellation. Suppose that $C$ and $C'$ differ for the $j$-th toric $G$-constellation. This translates into the fact that, if $\Calf_j$ and $\Calf_j'$ are the respective $j$-th $G$-constellations of $C$ and $C'$ and $\Gamma_j,\Gamma_j'$ are their abstract $G$-stairs, then
	$$\Gamma_j\not=\Gamma_j'.$$
	
	Let us denote by $\Calf_{j-1}$ the $(j-1)$-th toric $G$-constellation of $C$ (and $C'$) and by $\Gamma_{j-1}$ its abstract $G$-stair. Then, both $\Gamma_j$ and $\Gamma_j'$ are substairs of the decreasing linking stair $L_{j-1}$ of $\Gamma_{j-1}$ and they have horizontal right cut in $L_{j-1}$ as noticed above. Since, $\Gamma_{j-1},\Gamma_j$ and $\Gamma_j'$ are connected and $\Gamma_{j-1}\cap \Gamma_j,\Gamma_{j-1}\cap \Gamma_j'\not=\emptyset$ in $L_{j-1}$, it follows that:
	$$\Gamma_{j-1}\cup \Gamma_j\subsetneq \Gamma_{j-1}\cup \Gamma_j'\mbox{ or }\Gamma_{j-1}\cup \Gamma_j\supsetneq \Gamma_{j-1}\cup \Gamma_j'.$$
	
	Finally, if  without loss of generality  we suppose
	$$\Gamma_{j-1}\cup \Gamma_j\subsetneq \Gamma_{j-1}\cup \Gamma_j'\subset \Gamma,$$
	then we get a contradiction. Indeed, as noticed at the beginning, $\Gamma_{j}$ has a vertical right cut in $\Gamma$, but it has to have a horizontal right cut in  $\Gamma_{j-1}\cup \Gamma_j'$ because it is a connected substair of $L_{j-1}$ which strictly contains $\Gamma_{j}$.
	
\end{remark}
	\begin{remark}Since the abstract chamber stair $\Gamma_C$ of a chamber $C$ contains a copy of the abstract $G$-stairs of the toric $C$-stable $G$-constellations, we will think of such abstract $G$-stairs as substairs of $\Gamma_C$.
		
		Similarly, given a $C$-stair $\widetilde{\Gamma}_C\subset\Calt_G$ which realize $\Gamma_C$, we will realize the abstract $G$-stairs associated to the $G$-constellations in $C$ as substairs of $\widetilde{\Gamma}_C$.
	\end{remark}
	\begin{definition}\label{simplechamber}
	Given a chamber $C$, we  say that a toric $C$-stable $G$-constellation is \textit{$C$-characteristic} if its abstract $G$-stair has the same generators as the abstract $C$-stair, see \Cref{stair}.
	
		We say that a chamber $C$ is \textit{simple} if there is a toric $C$-stable $G$-constellation whose abstract $G$-stair has the same generators of the abstract $C$-stair, i.e. if there exists at least one $C$-characteristic $G$-constellation. 
	\end{definition}
	\begin{example}
		An example of a simple chamber is given by the chamber $C_G$ in \Cref{CRAWthm}, i.e. the chamber whose associated moduli space is $G$-$ \Hilb(\A^2) $. In particular, the abstract $C_G$-stair has only one generator, namely $\rho_0$.
	\end{example}
\begin{definition} Let $\Gamma$ be a $G$-stair and let $\rho_i$ and $\rho_j$ be its first and its last generators.
	\begin{itemize}
		\item 
		The \textit{left tail} of $\Gamma$ is the substair of $\Gamma$ given by
		\[\mathfrak{lt}(\Gamma)=\Set{y^s\cdot\rho_i|s>0}.\]
		\item 
		The \textit{right tail} of $\Gamma$ is the substair of $\Gamma$ given by
		\[\mathfrak{rt}(\Gamma)= \Set{x^s\cdot\rho_j|s>0}.\]
		\item The \textit{tail} of $\Gamma$ is the substair of $\Gamma$ given by
		$$\mathfrak{t}(\Gamma)=\mathfrak{lt}(\Gamma)\cup \mathfrak{rt}(\Gamma).$$
	\end{itemize}
Similarly one can define left/right tails for abstract $G$-stairs.
\end{definition}
\begin{remark}\label{stessigen} If two $G$-stairs $\Gamma$ and $\Gamma'$ have the same generators, then they differ by their tails, i.e. the following equality of subsets of the representation tableau holds true:
	$$\Gamma\smallsetminus\mathfrak{t}(\Gamma)=\Gamma'\smallsetminus\mathfrak{t}(\Gamma')$$
	In particular, if a $G$-stair $\Gamma$ has a tail of cardinality $m$, then there are $m+1$ $G$-stairs with the same generators as $\Gamma$.
	
	In simple words, the other $G$-stairs are obtained by moving some boxes from the left tail to the right tail (and viceversa) of $\Gamma$.
\end{remark}
	\begin{prop}\label{propsimple}The following properties are true.
		\begin{enumerate}
			\item Any toric $G$-constellation is $C$-stable for some simple chamber $C$. 
			\item   Given  a simple chamber $C$, and a $C$-characteristic $G$-constellation $\Calf$, there is an algorithm to produce all the toric $C$-stable constellations. 
			\item If $C$ is a simple chamber, all the toric $G$-constellations that admit a $G$-stair with the same generators as the $C$-stair belong to $C$, i.e. they are $C$-stable. In particular, they are $C$-characteristic.
		\end{enumerate} 
	\end{prop}
	\begin{proof} Let $\Gamma_C$ be the abstract $C$-stair. We  prove the first two points in a constructive way. In order to do so, we  show that, given a toric $G$-constellation $\Calf$, there is a unique simple chamber $C$ such that $\Calf$ is $C$-characteristic.
		
		Let $\Calf$ be a toric $G$-constellation with associated abstract $G$-stair $\Gamma_\Calf$ of height $\mathfrak{h}(\Calf)=j$. In order to build a chamber starting from $\Calf$, we have to first recursively apply Propositions \ref{propcoppia} and \ref{propcoppia1} $j-1$ times and $k-j$ times respectively, to obtain $k$ toric constellations
		$$\Calf_1,\ldots,\Calf_{j-1},\Calf,\Calf_{j+1},\ldots,\Calf_k$$
		and, finally, apply \Cref{TEO1} to conclude that there exists a chamber $C$ such that the constellations $\Calf_1,\ldots,\Calf_{j-1},\Calf,\Calf_{j+1},\ldots,\Calf_k$ correspond to the toric points of $\Calm_C$.
		
		The condition that the chamber must be simple translates into the fact that, at every step, no new generators appear. This may be only achieved by performing, every time that we apply \Cref{propcoppia} (resp. \Cref{propcoppia1}), the first (resp. last) possible horizontal (resp. vertical) cut in the decreasing (resp. increasing) linking stair.
		
		In order to prove the last point, we start by considering a $G$-constellation $\Calf$ whose abstract $G$-stair $\Gamma_\Calf$ has the same generators as the $C$-stair and such that it has empty right tail, i.e. $\mathfrak{t}(\Gamma_\Calf)=\mathfrak{lt}(\Gamma_\Calf)$.
		
		Let $m=\#\mathfrak{lt}(\Gamma_\Calf)$ be the cardinality of the left tail of $\Gamma_\Calf$. The first $m$ times we apply \Cref{propcoppia} by performing the first possible horizontal cut we increase the cardinality of $\mathfrak{rt}(\Gamma_\Calf)$ by 1 and, consequently, we decrease the cardinality of $\mathfrak{lt}(\Gamma_\Calf)$ by 1. In this way we find, as explained in \Cref{stessigen}, all the toric $G$-constellations which admit a $G$-stair with the same generators as the $C$-stair and all of them are $C$-stable by \Cref{TEO1}. 
	\end{proof}

\begin{lemma}\label{genconst} Let $\Gamma$ be a $G$-stair. Then $\Gamma$ has at most 
	$$\left\lfloor \frac{k+1}{2} \right\rfloor$$
	generators.
\end{lemma}
\begin{proof}The statement follows from the following observation. If a stair has $r$ generators, then it has at least $2r-1$ boxes, as shown in \Cref{numgens}.
	\begin{figure}[H]\scalebox{0.6}{
			\begin{tikzpicture}[scale=0.7]
				\draw (0,0)--(0,1)--(1,1)--(1,0)--(0,0)--(0,1);
				\draw (2,0)--(2,1)--(3,1)--(3,0)--(2,0)--(2,1);
				\draw (2,-2)--(2,-1)--(3,-1)--(3,-2)--(2,-2)--(2,-1);
				\draw (4,-2)--(4,-1)--(5,-1)--(5,-2)--(4,-2)--(4,-1);
				\draw (5,-3)--(5,-2)--(6,-2)--(6,-3)--(5,-3)--(5,-2);
				\draw (5,-5)--(5,-4)--(6,-4)--(6,-5)--(5,-5)--(5,-4);
				\node at (0.5,1.6) {$\vdots$};
				\node at (1.5,0.5) {$\cdots$};
				\node at (2.5,-0.4) {$\vdots$};
				\node at (3.5,-1.5) {$\cdots$};
				\node at (5.5,-1.5) {$\ddots$};
				\node at (5.5,-3.4) {$\vdots$};
				\node at (6.5,-4.5) {$\cdots$};
				\node at (2,-3) {\rotatebox{-45}{$\underbrace{\hspace{5cm}}$}};
				\node at (5,-0.5) {\rotatebox{-45}{$\overbrace{\hspace{3.5cm}}$}};
				\node[left] at (2,-3.5) {$r$};
				\node[right] at (5,0) {$r-1$};
		\end{tikzpicture}}
		\caption{\ }
		\label{numgens}
	\end{figure}
	Now, a $G$-stair has exactly $k$ boxes. Hence,
	$$r\le \left\lfloor \frac{k+1}{2} \right\rfloor.$$
\end{proof}

	\begin{example}\label{hilbop}
		Non-simple chambers exist.
		
		As already mentioned in \Cref{CRAWthm}, there is a chamber $C_G$ such that $G$-$\Hilb(\A^2)\cong\Calm_{C_G}$ as moduli spaces. In particular,
		\[C_G\subset\Set{\theta\in\Theta | \theta_0<0,\ \theta_i>0\ \forall i=1,\ldots,k-1 },\]
		and the abstract $G$-stairs of its toric constellations are shown in \Cref{torichilb}.
		\begin{figure}[H]
			\scalebox{1}{
			\begin{tikzpicture}
				\node at (0,-2.5) {$\Gamma_{\Calf_1}$};
					\node at (0,0) {$\begin{matrix}
						\begin{tikzpicture}[scale=0.7]
						\draw[dashed](0,2)--(0,3);
						\draw[dashed](1,2)--(1,3);
							\draw (1,1)--(1,2)--(0,2)--(0,0)--(1,0)--(1,1)--(0,1);
							\draw (1,4)--(1,6)--(0,6)--(0,3)--(1,3)--(1,4)--(0,4);
							\draw (1,5)--(0,5);
							\node at (0.5,2.6) {$\vdots$};
							\node at (0.5,0.5) {\small$0$};
							\node at (0.5,1.5) {\small$k$-$1$};
							\node at (0.5,3.5) {\small$3$};
							\node at (0.5,4.5) {\small$2$};
							\node at (0.5,5.5) {\small$1$};
						\end{tikzpicture}
					\end{matrix}$};
				\node at (1.5,-2.5) {$\Gamma_{\Calf_2}$};
				\node at (1.5,-0.35) {$\begin{matrix}
						\begin{tikzpicture}[scale=0.7]
						\draw[dashed](0,2)--(0,3);
						\draw[dashed](1,2)--(1,3);
							\draw (1,0)--(2,0)--(2,1)--(1,1);
							\draw (1,1)--(1,2)--(0,2)--(0,0)--(1,0)--(1,1)--(0,1);
							\draw (1,4)--(1,5)--(0,5)--(0,3)--(1,3)--(1,4)--(0,4);
							\node at (0.5,2.6) {$\vdots$};
							\node at (0.5,0.5) {\small$0$};
							\node at (0.5,1.5) {\small$k$-$1$};
							\node at (0.5,3.5) {\small$3$};
							\node at (0.5,4.5) {\small$2$};
							\node at (1.5,0.5) {\small$1$};
						\end{tikzpicture}
					\end{matrix}$};
				\node at (3.7,-2.5) {$\Gamma_{\Calf_3}$};
			\node at (3.7,-0.7) {$\begin{matrix}
					\begin{tikzpicture}[scale=0.7]
						\draw[dashed](0,2)--(0,3);
						\draw[dashed](1,2)--(1,3);
						\draw (2,0)--(2,1);
						\draw (1,0)--(3,0)--(3,1)--(1,1);
						\draw (1,1)--(1,2)--(0,2)--(0,0)--(1,0)--(1,1)--(0,1);
						\draw (1,3)--(1,4)--(0,4)--(0,3)--(1,3)--(1,4);
						\node at (0.5,2.6) {$\vdots$};
						\node at (0.5,0.5) {\small$0$};
						\node at (0.5,1.5) {\small$k$-$1$};
						\node at (0.5,3.5) {\small$3$};
						\node at (2.5,0.5) {\small$2$};
						\node at (1.5,0.5) {\small$1$};
					\end{tikzpicture}
				\end{matrix}$};
			\node at (5.3,-1.4) {$\cdots$};
			\node at (5.3,-2.5) {$\cdots$};
			\node at (7.7,-2.5) {$\Gamma_{\Calf_{k-1}}$};
			\node at (7.7,-1.4) {$\begin{matrix}
					\begin{tikzpicture}[scale=0.7]
						\draw[dashed](2,0)--(3,0);
						\draw[dashed](2,1)--(3,1);
						\draw (1,1)--(1,2)--(0,2)--(0,0)--(1,0)--(1,1)--(0,1);
						\draw (1,0)--(2,0)--(2,1)--(1,1);
						\node at (2.5,0.5) {$\cdots$};
						\draw (4,1)--(3,1)--(3,0)--(5,0)--(5,1)--(4,1)--(4,0);
						\node at (0.5,0.5) {\small$0$};
						\node at (0.5,1.5) {\small$k$-$1$};
						\node at (4.5,0.5) {\small$k$-$2$};
						\node at (3.5,0.5) {\small$k$-$3$};
						\node at (1.5,0.5) {\small$1$};
					\end{tikzpicture}
				\end{matrix}$};
			\node at (12,-2.5) {$\Gamma_{\Calf_k}$.};
		\node at (12,-1.75) {$\begin{matrix}
				\begin{tikzpicture}[scale=0.7]
						\draw[dashed](2,0)--(3,0);
						\draw[dashed](2,1)--(3,1);
					\draw (1,1)--(1,0)--(0,0)--(0,1)--(1,1);
					\draw (1,0)--(2,0)--(2,1)--(1,1);
					\node at (2.5,0.5) {$\cdots$};
					\draw (4,1)--(3,1)--(3,0)--(5,0)--(5,1)--(4,1)--(4,0);
					\draw (5,0)--(6,0)--(6,1)--(5,1);
					\node at (0.5,0.5) {\small$0$};
					\node at (5.5,0.5) {\small$k$-$1$};
					\node at (4.5,0.5) {\small$k$-$2$};
					\node at (3.5,0.5) {\small$k$-$3$};
					\node at (1.5,0.5) {\small$1$};
				\end{tikzpicture}
			\end{matrix}$};
		\end{tikzpicture}}
			\caption{The abstract $G$-stairs of the $C_G$-stable toric $G$-constellations.}
			\label{torichilb}
		\end{figure}
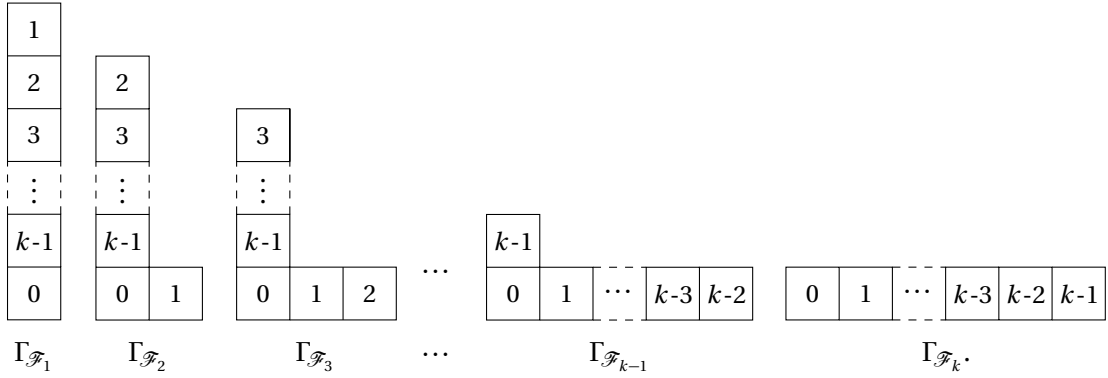 
	Notice that, for $i=1,\ldots,k$ and $j=0,\ldots,k-1$, the favorite conditions $\theta_{\Calf_i}$ are defined by
	$$(\theta_{\Calf_i})_j=\begin{cases}
		-2&\mbox{if } j=0 \ \&\ i\not=1,k,\\
		-1&\mbox{if } j=0 \ \&\ (i=1 \mbox{ or } i=k),\\
		1&\mbox{if } j=i-1\not=0,\\
		1&\mbox{if } j=i ,\\
		0& \mbox{otherwise.}
	\end{cases}$$
	and that the condition $$\theta=\ssum{i=1}{k}\theta_{\Calf_i}=(-2k+2,\underbrace{2,\ldots,2}_{k-1})$$ belongs to $C_G$. More precisely, the moduli space $G$-$\Hilb(\A^2)$ parametrises all the toric $G$-constellations generated by the trivial representation. As a consequence, the abstract $G$-stairs $\Gamma_{\Calf_i}$, for $i=1,\ldots,k$, have  as only generator  the trivial representation.
	
		Let us reverse this property by asking the presence of just one antigenerator, for example, the trivial representation. It is easy to see that there is a chamber $C_G^{\OP}$ whose toric $G$-constellations, as requested, have the abstract $G$-stairs  in \Cref{torichilbop}.
		\begin{figure}[ht]
			\scalebox{1}{
				\begin{tikzpicture}
					\node at (0,-2.5) {$\Gamma_{\Calf_1'}$};
					\node at (0,0) {$\begin{matrix}
							\begin{tikzpicture}[scale=0.7]
						\draw[dashed](0,2)--(0,3);
						\draw[dashed](1,2)--(1,3);
								\draw (1,1)--(1,2)--(0,2)--(0,0)--(1,0)--(1,1)--(0,1);
								\draw (1,4)--(1,6)--(0,6)--(0,3)--(1,3)--(1,4)--(0,4);
								\draw (1,5)--(0,5);
								\node at (0.5,2.6) {$\vdots$};
								\node at (0.5,0.5) {\small$k$-1};
								\node at (0.5,1.5) {\small$k$-2};
								\node at (0.5,3.5) {\small$2$};
								\node at (0.5,4.5) {\small$1$};
								\node at (0.5,5.5) {\small$0$};
							\end{tikzpicture}
						\end{matrix}$};
					\node at (1.5,-2.5) {$\Gamma_{\Calf_2'}$};
					\node at (1.5,-0.35) {$\begin{matrix}
							\begin{tikzpicture}[scale=0.7]
						\draw[dashed](0,2)--(0,3);
						\draw[dashed](1,2)--(1,3);
								\draw (0,5)--(-1,5)--(-1,4)--(0,4);
								\draw (1,1)--(1,2)--(0,2)--(0,0)--(1,0)--(1,1)--(0,1);
								\draw (1,4)--(1,5)--(0,5)--(0,3)--(1,3)--(1,4)--(0,4);
								\node at (0.5,2.6) {$\vdots$};
								\node at (0.5,0.5) {\small$k$-2};
								\node at (0.5,1.5) {\small$k$-3};
								\node at (0.5,3.5) {\small$1$};
								\node at (0.5,4.5) {\small$0$};
								\node at (-0.5,4.5) {\small$k$-1};
							\end{tikzpicture}
						\end{matrix}$};
					\node at (3.7,-2.5) {$\Gamma_{\Calf_3'}$};
					\node at (3.7,-0.7) {$\begin{matrix}
							\begin{tikzpicture}[scale=0.7]
						\draw[dashed](0,2)--(0,1);
						\draw[dashed](1,2)--(1,1);
								\draw (-1,3)--(-1,4);
								\draw (0,3)--(-2,3)--(-2,4)--(0,4);
								\draw (0,3)--(0,2)--(1,2)--(1,3);
								\draw (0,0)--(1,0)--(1,1)--(0,1)--(0,0)--(1,0);
								\draw (1,3)--(1,4)--(0,4)--(0,3)--(1,3);
								\node at (0.5,1.6) {$\vdots$};
								\node at (0.5,0.5) {\small$k$-3};
								\node at (0.5,2.5) {\small1};
								\node at (0.5,3.5) {\small$0$};
								\node at (-0.5,3.5) {\small$k$-1};
								\node at (-1.5,3.5) {\small$k$-2};
							\end{tikzpicture}
						\end{matrix}$};
					\node at (5.3,-1.4) {$\cdots$};
					\node at (5.3,-2.5) {$\cdots$};
					\node at (7.7,-2.5) {$\Gamma_{\Calf_{k-1}'}$};
					\node at (7.7,-1.4) {$\begin{matrix}
							\begin{tikzpicture}[scale=0.7]
						\draw[dashed](2,0)--(3,0);
						\draw[dashed](2,1)--(3,1);
								\draw (4,0)--(4,-1)--(5,-1)--(5,0);
								\draw (0,0)--(1,0)--(1,1)--(0,1)--(0,0)--(1,0);
								\draw (1,0)--(2,0)--(2,1)--(1,1);
								\node at (2.5,0.5) {$\cdots$};
								\draw (4,1)--(3,1)--(3,0)--(5,0)--(5,1)--(4,1)--(4,0);
								\node at (0.5,0.5) {\small2};
								\node at (4.5,-0.5) {\small1};
								\node at (4.5,0.5) {\small0};
								\node at (3.5,0.5) {\small$k$-1};
								\node at (1.5,0.5) {\small3};
							\end{tikzpicture}
						\end{matrix}$};
					\node at (12,-2.5) {$\Gamma_{\Calf_k'}$};
					\node at (12,-1.75) {$\begin{matrix}
							\begin{tikzpicture}[scale=0.7]
						\draw[dashed](2,0)--(3,0);
						\draw[dashed](2,1)--(3,1);
								\draw (1,1)--(1,0)--(0,0)--(0,1)--(1,1);
								\draw (1,0)--(2,0)--(2,1)--(1,1);
								\node at (2.5,0.5) {$\cdots$};
								\draw (4,1)--(3,1)--(3,0)--(5,0)--(5,1)--(4,1)--(4,0);
								\draw (5,0)--(6,0)--(6,1)--(5,1);
								\node at (0.5,0.5) {\small$1$};
								\node at (5.5,0.5) {\small$0$};
								\node at (4.5,0.5) {\small$k$-1};
								\node at (3.5,0.5) {\small$k$-$2$};
								\node at (1.5,0.5) {\small$2$};
							\end{tikzpicture}
						\end{matrix}$};
			\end{tikzpicture}}
			\caption{The abstract $G$-stairs of the $C_G^{\OP}$-stable toric $G$-constellations.}
			\label{torichilbop}
		\end{figure}
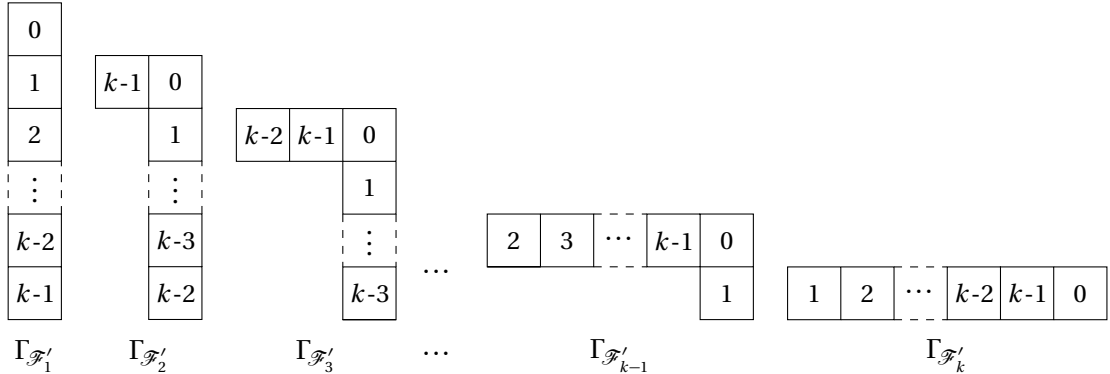 
		In particular,
		\[C_G^{\OP}\subset\Set{\theta\in\Theta|\theta_0>0,\ \theta_i<0\ \forall i=1,\ldots,k-1}.\]
		We denote the associated moduli space by
		$$G\mbox{-}\Hilb^{\OP}(\A^2):= \Calm_{C_G^{\OP}}.$$
		Notice that, while $C_{\Z/3\Z}^{\OP}$ is simple, $C_{\Z/k\Z}^{\OP}$ is not simple for $k\ge 4$ because the number of generators of the $C_{\Z/k\Z}^{\OP}$-stair is
		$$k-1>\left\lfloor\frac{k+1}{2}\right\rfloor \ \ \forall k\ge 4.$$
		Therefore, as a consequence of \Cref{genconst}, there is no $ C_{\Z/k\Z}^{\OP}$-characteristic $G$-constellation.

		We show, as an example, the abstract chamber stairs of $C_G$ and $C_G^{\OP}$ in the case $k=5$.
		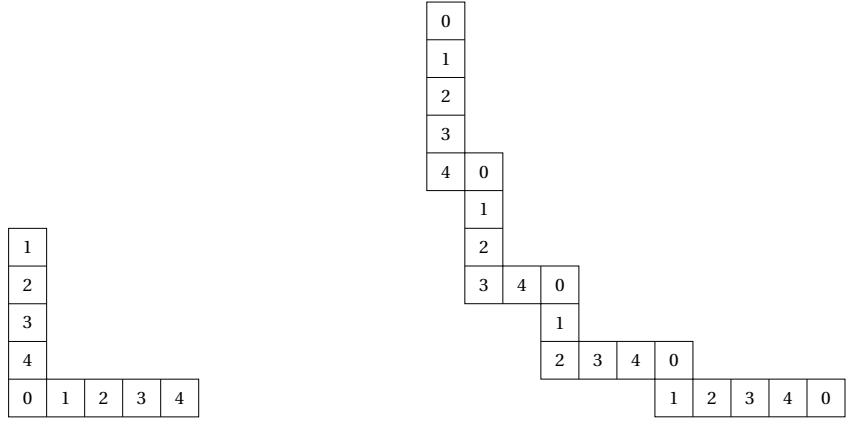
\begin{figure}[H]
			\begin{tikzpicture}
				\node at (0,0) { 	$\begin{matrix}
						\scalebox{0.5}{
							\begin{tikzpicture}
								\draw (0,-4)--(0,-5)--(5,-5)--(5,-4)--(1,-4)--(1,0)--(0,0)--(0,-4);
								\draw (0,-1)--(1,-1);
								\draw (0,-2)--(1,-2);
								\draw (0,-3)--(1,-3);
								\draw (0,-4)--(1,-4)--(1,-5);
								\draw (2,-4)--(2,-5);
								\draw (3,-4)--(3,-5);
								\draw (4,-4)--(4,-5);
								
								\node at (0.5,-0.5) {\Large1};
								\node at (0.5,-1.5) {\Large2};
								\node at (0.5,-2.5) {\Large3};
								\node at (0.5,-3.5) {\Large4};
								\node at (0.5,-4.5) {\Large0};
								\node at (1.5,-4.5) {\Large1};
								\node at (2.5,-4.5) {\Large2};
								\node at (3.5,-4.5) {\Large3};
								\node at (4.5,-4.5) {\Large4};
						\end{tikzpicture}}
					\end{matrix}$};
				\node at (7,1.5 ) {	$\begin{matrix}\scalebox{0.5}{
							\begin{tikzpicture}
								\draw (0,-4)--(0,-5)--(1,-5)--(1,-8)--(3,-8)--(3,-10)--(6,-10)--(6,-11)--(11,-11)--(11,-10)--(7,-10)--(7,-9)--(4,-9)--(4,-7)--(2,-7)--(2,-4)--(1,-4)--(1,0)--(0,0)--(0,-4);
								\draw (0,-1)--(1,-1);
								\draw (0,-2)--(1,-2);
								\draw (0,-3)--(1,-3);
								\draw (0,-4)--(1,-4)--(1,-5)--(2,-5);
								\draw (1,-6)--(2,-6);
								\draw (1,-7)--(2,-7)--(2,-8);
								\draw (3,-7)--(3,-8)--(4,-8);
								\draw (3,-9)--(4,-9)--(4,-10);
								\draw (5,-9)--(5,-10);
								\draw (6,-9)--(6,-10)--(7,-10)--(7,-11);
								\draw (8,-10)--(8,-11);
								\draw (9,-10)--(9,-11);
								\draw (10,-10)--(10,-11);
								
								\node at (0.5,-0.5) {\Large0};
								\node at (0.5,-1.5) {\Large1};
								\node at (0.5,-2.5) {\Large2};
								\node at (0.5,-3.5) {\Large3};
								\node at (0.5,-4.5) {\Large4};
								\node at (1.5,-4.5) {\Large0};
								\node at (1.5,-5.5) {\Large1};
								\node at (1.5,-6.5) {\Large2};
								\node at (1.5,-7.5) {\Large3};
								\node at (2.5,-7.5) {\Large4};
								\node at (3.5,-7.5) {\Large0};
								\node at (3.5,-8.5) {\Large1};
								\node at (3.5,-9.5) {\Large2};
								\node at (4.5,-9.5) {\Large3};
								\node at (5.5,-9.5) {\Large4};
								\node at (6.5,-9.5) {\Large0};
								\node at (6.5,-10.5) {\Large1};
								\node at (7.5,-10.5) {\Large2};
								\node at (8.5,-10.5) {\Large3};
								\node at (9.5,-10.5) {\Large4};
								\node at (10.5,-10.5) {\Large0};
						\end{tikzpicture}}
					\end{matrix}$};
			\end{tikzpicture}
			\caption{ The abstract {$C_{\Z/5\Z}$-stair} and the abstract {$C_{\Z/5\Z}^{\OP}$-stair.}}
		\end{figure}
	\end{example}

	\begin{theorem}\label{teosimple} If $G\subset \SL(2,\C)$ is a finite abelian subgroup of cardinality $k=|G|$, then the space of generic stability conditions $\Theta^{\gen}$ contains $k\cdot 2^{k-2}$ simple chambers.
	\end{theorem}
\begin{proof}Let $\Calb$ be the set of of possible sets of generators for a $G$-stair, i.e.
	\[\Calb=\Set{A\subset \Calt_G|\mbox{there exists a $G$-stair whose generators are the elements in $A$}},\]
	and let $\Calg$ be the set of all $G$-stairs
	\[\Calg=\Set{\Gamma\subset \Calt_G|\mbox{$\Gamma$ is a $G$-stair}}.\]
	Consider the subsemigroup $Z$ of $\Calt_G$
	\[Z=\Set{(\alpha k +\gamma , \beta k+\gamma,\rho_0)\in\Calt_G|\alpha,\beta,\gamma\ge 0  }.\]
	We denote by $\overline{\Calb}$ and $\overline{\Calg}$ the set of equivalence classes
	$$\overline{\Calb}=\Calb/\sim_Z,\ \mbox{and} \  \overline{\Calg}=\Calg/\sim_Z$$
	where, if $A_1,A_2\in \Calb$ (resp. $\Gamma_1,\Gamma_2\in\Calg$), then $A_1\sim_Z A_2$ (resp. $\Gamma_1\sim_Z \Gamma_2$) if there exist $z\in Z$ such that
	$$A_1=A_2+z\mbox{ or }A_2= A_1+z \quad (\mbox{resp. }\Gamma_1=\Gamma_2+z\mbox{ or }\Gamma_2= \Gamma_1+z).$$
	Notice that, if two $G$-stairs are $\sim_Z$-equivalent also their sets of generators are $\sim_Z$-equivalent. However, the contrary is not true.
	
	 Now, the number of simple chambers equals the cardinality of $\overline{\Calb}$. Indeed, \Cref{propsimple} implies that the chamber $C$ is uniquely determined by a constellation $\Calf$ whose $G$-stair is $C$-characteristic. More precisely, $C$ is uniquely determined by the generators of any characteristic $C$-stair $\Gamma_\Calf$. Although there are infinitely many $G$-stairs corresponding to $\Calf$, \Cref{periodn} tells us that two $G$-stairs correspond to the same $G$-constellation if and only if they differ by an element in $Z$, i.e. they are $\sim_Z$-equivalent.
	 
	 Let $\Calg_r$ be the set of $G$-stairs with $r$ generators and let $\overline{\Calg}_r=\Calg_r/\sim_Z$ be the induced quotient. We have a surjective map
	 $$\Psi:\Calg\rightarrow\Calb$$
	 which associates to each $G$-stair its set of generators, and this map descends to the sets of equivalence classes
	 $$\overline{\Psi}:\overline{\Calg}\rightarrow\overline{\Calb},$$
	 because $\sim_Z$-equivalent $G$-stairs correspond to $\sim _Z$-equivalent sets of generators.
	 
	 Now, $\overline{\Calb}$ decomposes as a disjoint union (see \Cref{genconst}) as follows:
	 $$\overline{\Calb}=\underset{r=1}{\overset{\left\lfloor \frac{k+1}{2}\right\rfloor}{\bigsqcup}}\overline{\Psi}(\overline{\Calg}_r).$$
	 
	 Our strategy is to compute $\overline{\Psi}(\overline{\Calg}_r)$ for every $1\le r \le\left\lfloor \frac{k+1}{2}\right\rfloor$ and then sum over all $r$.
	 For $r=1$ we have $|\overline{\Psi}(\overline{\Calg}_1)|=k$. If we impose the presence of $r\ge2$ generators and of a tail of cardinality $j$ then there are
	 $$k\cdot\binom{k-2-j}{2r-3}$$
	 elements in $\overline{\Psi}(\overline{\Calg}_r)$ which comes from $G$-stairs with a tail of cardinality $j$. Indeed, as shown in \Cref{arrangement}, we have $2r-1$ fixed boxes (generators and anti-generators), $j$ boxes contained in the tails (dashed areas) and $k-2r+1-j$ boxes left to arrange in $2r-2$ places (dotted areas).	 
	 \begin{figure}\scalebox{0.8}{
	 		\begin{tikzpicture}[scale=0.6]
				\draw (0,0)--(0,1)--(1,1)--(1,0)--(0,0)--(0,1);
				\draw (2,0)--(2,1)--(3,1)--(3,0)--(2,0)--(2,1);
				\draw (2,-2)--(2,-1)--(3,-1)--(3,-2)--(2,-2)--(2,-1);
				\draw (4,-2)--(4,-1)--(5,-1)--(5,-2)--(4,-2)--(4,-1);
				\draw (5,-3)--(5,-2)--(6,-2)--(6,-3)--(5,-3)--(5,-2);
				\draw (5,-5)--(5,-4)--(6,-4)--(6,-5)--(5,-5)--(5,-4);
	 			\node at (1.5,0.5) {$\cdots$};
	 			\node at (2.5,-0.35) {$\vdots$};
	 			\node at (3.5,-1.5) {$\cdots$};
	 			\node at (5.5,-1.5) {$\ddots$};
	 			\node at (5.5,-3.35) {$\vdots$};
	 			\node at (2.2,-3) {\rotatebox{-45}{$\underbrace{\hspace{4.7cm}}$}};
	 			\node at (4.8,-0.2) {\rotatebox{-45}{$\overbrace{\hspace{3cm}}$}};
	 			\node[left] at (2.2,-3.5) {$r$};
	 			\node[right] at (4.8,0.2) {$r-1$};
	 			\draw[dashed] (0,1)--(0,2);
	 			\draw[dashed] (1,1)--(1,2);
	 			\draw[dashed] (6,-5)--(7,-5);
	 			\draw[dashed] (6,-4)--(7,-4);
	 	\end{tikzpicture}}
	 	\caption{\ }
	 	\label{arrangement}
	 \end{figure}
 The number of possible ways to arrange the boxes is computed via the stars and bars method\footnote{In a more suggestive way, one can say \virg combinations with repetition of $2r-2$ elements of class $k-2r+1-j$".}. In particular, there are
 $$\binom{(2r-2)+(k-2r+1-j)-1}{k-2r+1-j}=\binom{k-2-j}{2r-3}$$
 of them.
 
 Finally, if we sum over all possible $r$ and $j$, we get
 $$k\cdot \left[1+\ssum{r=2}{\left\lfloor\frac{k+1}{2}\right\rfloor}\ssum{j=0}{k-2r+1}\binom{k-2-j}{2r-3}\right]=k\cdot {2^{k-2}}.$$
\end{proof}

\begin{remark}
An easy combinatorial computation shows that the set $\overline{\Calg}$ in the proof of \Cref{teosimple} has cardinality $k\cdot 2^{k-1}$, i.e. that there there are exactly $k\cdot 2^{k-1}$ isomorphisms classes of toric $G$-constellations.
\end{remark}

We conclude this section with two examples which help to understand the notions just introduced.
\begin{example} In this example we treat the case $G\cong \Z/5\Z$.
	
	 The following picture contains a list of the possible shapes of the abstract chamber stairs of simple chambers and, in each case, the shapes of the $G$-stairs associated to the toric $G$-constellations belonging to the respective simple chamber.
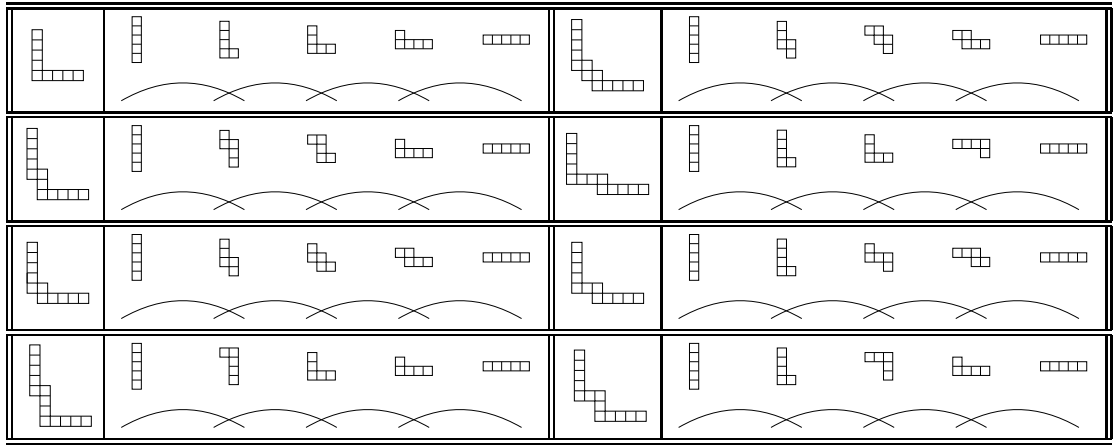
\begin{figure}[H]
	\scalebox{0.9}{
	\begin{tabular}{||c|c||c|c||}
		\hline\hline
		$\begin{matrix}
			{\begin{tikzpicture}[scale=0.15]
			\draw (0,1)--(0,0)--(5,0)--(5,1)--(1,1)--(1,5)--(0,5)--(0,1);
			\draw (0,1)--(1,1)--(1,0);
			\draw (2,0)--(2,1);
			\draw (3,0)--(3,1);
			\draw (4,0)--(4,1);
			\draw (0,4)--(1,4);
			\draw (0,3)--(1,3);
			\draw (0,2)--(1,2);
			\end{tikzpicture}}
		\end{matrix}$&
		$\begin{matrix}\scalebox{0.9}{
			\begin{tikzpicture}
				\draw (0,0) to [out=30,in=150] (2,0);
				\draw (1.5,0) to [out=30,in=150] (3.5,0);
				\draw (3,0) to [out=30,in=150] (5,0);
				\draw (4.5,0) to [out=30,in=150] (6.5,0);
				\node at (0.25,1) {\begin{tikzpicture}[scale=0.15]
			\draw (0,1)--(0,0)--(1,0)--(1,5) --(0,5)--(0,1);
			\draw (0,1)--(1,1);
			\draw (0,4)--(1,4);
			\draw (0,3)--(1,3);
			\draw (0,2)--(1,2);
					\end{tikzpicture}
					
				};
				\node at (1.75,1) {\begin{tikzpicture}[scale=0.15]
			\draw (0,1)--(0,0)--(2,0)--(2,1)--(1,1)--(1,4) --(0,4)--(0,1);
			\draw (0,1)--(1,1)--(1,0);
			\draw (0,3)--(1,3);
			\draw (0,2)--(1,2);
					\end{tikzpicture}
					
				};
				\node at (3.25,1) {\begin{tikzpicture}[scale=0.15]
			\draw (0,1)--(0,0)--(3,0)--(3,1)--(1,1)--(1,3) --(0,3)--(0,1);
			\draw (0,1)--(1,1)--(1,0);
			\draw (0,2)--(1,2);
			\draw (2,0)--(2,1);
					\end{tikzpicture}
					
				};
				\node at (4.75,1) {\begin{tikzpicture}[scale=0.15]
			\draw (0,1)--(0,0)--(4,0)--(4,1)--(1,1)--(1,2) --(0,2)--(0,1);
			\draw (0,1)--(1,1)--(1,0);
			\draw (2,0)--(2,1);
			\draw (3,0)--(3,1);
					\end{tikzpicture}
					
				};
				\node at (6.25,1) {\begin{tikzpicture}[scale=0.15]
			\draw (1,0)--(0,0)--(0,1)--(5,1) --(5,0)--(1,0);
			\draw (1,0)--(1,1);
			\draw (4,0)--(4,1);
			\draw (3,0)--(3,1);
			\draw (2,0)--(2,1);
					\end{tikzpicture}
					
				};
		\end{tikzpicture}}
	\end{matrix}$&
$\begin{matrix}
\scalebox{1}{\begin{tikzpicture}[scale=0.15]
			\draw (0,1)--(0,0)--(1,0)--(1,-1)--(6,-1)--(6,0)--(2,0)--(2,1)--(1,1)--(1,2)--(0,2)--(0,6) --(-1,6)--(-1,1)--(0,1);
			\draw (0,2)--(0,1)--(1,1)--(1,0)--(2,0);
			\draw (0,5)--(-1,5);
			\draw (0,4)--(-1,4);
			\draw (0,3)--(-1,3);
			\draw (0,2)--(-1,2);
			\draw (2,0)--(2,-1);
			\draw (3,0)--(3,-1);
			\draw (4,0)--(4,-1);
			\draw (5,0)--(5,-1);
\end{tikzpicture}}
\end{matrix}$&
$\begin{matrix}\scalebox{0.9}{
	\begin{tikzpicture}
		\draw (0,0) to [out=30,in=150] (2,0);
		\draw (1.5,0) to [out=30,in=150] (3.5,0);
		\draw (3,0) to [out=30,in=150] (5,0);
		\draw (4.5,0) to [out=30,in=150] (6.5,0);
		\node at (0.25,1) {\begin{tikzpicture}[scale=0.15]
			\draw (0,1)--(0,0)--(1,0)--(1,5) --(0,5)--(0,1);
			\draw (0,1)--(1,1);
			\draw (0,4)--(1,4);
			\draw (0,3)--(1,3);
			\draw (0,2)--(1,2);
			\end{tikzpicture}
			
		};
		\node at (4.75,1) {\begin{tikzpicture}[scale=0.15]
		\draw (2,0)--(1,0)--(1,1)--(0,1)--(0,2)--(2,2)--(2,1)--(4,1)--(4,0)--(2,0);
		\draw (1,2)--(1,1)--(2,1)--(2,0);
		\draw (3,1)--(3,0);
			\end{tikzpicture}
			
		};
		\node at (3.25,1) {\begin{tikzpicture}[scale=0.15]
			\draw (0,1)--(0,0)--(1,0)--(1,-1)--(2,-1)--(2,1)--(1,1)--(1,2) --(-1,2)--(-1,1)--(0,1);
			\draw (0,2)--(0,1)--(1,1)--(1,0)--(2,0);
			\end{tikzpicture}
			
		};
		\node at (1.75,1) {\begin{tikzpicture}[scale=0.15]
			\draw (0,1)--(0,0)--(1,0)--(1,-1)--(2,-1)--(2,1)--(1,1)--(1,3) --(0,3)--(0,1);
			\draw (0,1)--(1,1)--(1,0)--(2,0);
			\draw (0,2)--(1,2);
			\end{tikzpicture}
			
		};
		\node at (6.25,1) {\begin{tikzpicture}[scale=0.15]
			\draw (4,0)--(0,0)--(0,1)--(5,1) --(5,0)--(4,0);
			\draw (1,0)--(1,1);
			\draw (4,0)--(4,1);
			\draw (3,0)--(3,1);
			\draw (2,0)--(2,1);
			\end{tikzpicture}
			
		};
\end{tikzpicture}}
\end{matrix}$\\\hline\hline
$\begin{matrix}
\scalebox{1}{\begin{tikzpicture}[scale=0.15]
			\draw (0,1)--(0,0)--(0,-1)--(5,-1)--(5,0)--(1,0)-- (1,2)--(0,2)--(0,6) --(-1,6)--(-1,1)--(0,1);
			\draw (0,2)--(0,1)--(1,1)--(1,0)--(2,0);
			\draw (0,0)--(1,0)--(1,-1);
			\draw (0,5)--(-1,5);
			\draw (0,4)--(-1,4);
			\draw (0,3)--(-1,3);
			\draw (0,2)--(-1,2);
			\draw (2,0)--(2,-1);
			\draw (3,0)--(3,-1);
			\draw (4,0)--(4,-1);
\end{tikzpicture}}
\end{matrix}$&
		$\begin{matrix}\scalebox{0.9}{
			\begin{tikzpicture}
				\draw (0,0) to [out=30,in=150] (2,0);
				\draw (1.5,0) to [out=30,in=150] (3.5,0);
				\draw (3,0) to [out=30,in=150] (5,0);
				\draw (4.5,0) to [out=30,in=150] (6.5,0);
				\node at (0.25,1) {\begin{tikzpicture}[scale=0.15]
			\draw (0,4)--(0,0)--(1,0)--(1,5) --(0,5)--(0,4);
			\draw (0,1)--(1,1);
			\draw (0,4)--(1,4);
			\draw (0,3)--(1,3);
			\draw (0,2)--(1,2);
					\end{tikzpicture}
					
				};
				\node at (4.75,1) {\begin{tikzpicture}[scale=0.15]
			\draw (0,1)--(0,0)--(4,0)--(4,1)--(1,1)--(1,2) --(0,2)--(0,1);
			\draw (0,1)--(1,1)--(1,0);
			\draw (2,0)--(2,1);
			\draw (3,0)--(3,1);
					\end{tikzpicture}
					
				};
				\node at (3.25,1) {\begin{tikzpicture}[scale=0.15]
			\draw (1,1)--(0,1)--(0,0)--(1,0)--(1,-2)--(3,-2)--(3,-1)--(2,-1)--(2,1)--(1,1);
			\draw (1,1)--(1,0)--(2,0);
			\draw (1,-1)--(2,-1)--(2,-2);
					\end{tikzpicture}
					
				};
				\node at (1.75,1) {\begin{tikzpicture}[scale=0.15]
			\draw (0,1)--(0,0)--(1,0)--(1,-2)--(2,-2)--(2,1)--(1,1)--(1,2) --(0,2)--(0,1);
			\draw (0,1)--(1,1)--(1,0)--(2,0);
			\draw (1,-1)--(2,-1);
					\end{tikzpicture}
					
				};
				\node at (6.25,1) {\begin{tikzpicture}[scale=0.15]
			\draw (4,0)--(0,0)--(0,1)--(5,1) --(5,0)--(4,0);
			\draw (1,0)--(1,1);
			\draw (4,0)--(4,1);
			\draw (3,0)--(3,1);
			\draw (2,0)--(2,1);
					\end{tikzpicture}
					
				};
		\end{tikzpicture}}
	\end{matrix}$&
$\begin{matrix}
\scalebox{1}{\begin{tikzpicture}[scale=0.15]
\draw (0,1)--(0,5)--(1,5)--(1,1)--(4,1)--(4,0)--(8,0)--(8,-1)--(3,-1)--(3,0)--(0,0)--(0,1);
\draw (0,1)--(1,1)--(1,0);
\draw (0,2)--(1,2);
\draw (0,3)--(1,3);
\draw (0,4)--(1,4);
\draw (2,1)--(2,0);
\draw (3,1)--(3,0)--(4,0)--(4,-1);
\draw  (5,0)--(5,-1);
\draw  (6,0)--(6,-1);
\draw  (7,0)--(7,-1);
\end{tikzpicture}}
\end{matrix}$&
		$\begin{matrix}\scalebox{0.9}{
			\begin{tikzpicture}
				\draw (0,0) to [out=30,in=150] (2,0);
				\draw (1.5,0) to [out=30,in=150] (3.5,0);
				\draw (3,0) to [out=30,in=150] (5,0);
				\draw (4.5,0) to [out=30,in=150] (6.5,0);
				\node at (0.25,1) {\begin{tikzpicture}[scale=0.15]
			\draw (0,4)--(0,0)--(1,0)--(1,5) --(0,5)--(0,4);
			\draw (0,1)--(1,1);
			\draw (0,4)--(1,4);
			\draw (0,3)--(1,3);
			\draw (0,2)--(1,2);
					\end{tikzpicture}
					
				};
				\node at (4.75,1) {\begin{tikzpicture}[scale=0.15]
				\draw (1,2)--(4,2)--(4,0)--(3,0)--(3,1)--(0,1)--(0,2)--(1,2);
				\draw (4,1)--(3,1)--(3,2);
				\draw (2,1)--(2,2);
				\draw (1,1)--(1,2);
					\end{tikzpicture}
					
				};
				\node at (3.25,1) {\begin{tikzpicture}[scale=0.15]
			\draw (0,1)--(0,0)--(3,0)--(3,1)--(1,1)--(1,3) --(0,3)--(0,1);
			\draw (0,1)--(1,1)--(1,0);
			\draw (0,2)--(1,2);
			\draw (2,0)--(2,1);
					\end{tikzpicture}
					
				};
				\node at (1.75,1) {\begin{tikzpicture}[scale=0.15]
			\draw (0,3)--(0,0)--(2,0)--(2,1)--(1,1)--(1,4) --(0,4)--(0,3);
			\draw (0,1)--(1,1)--(1,0);
			\draw (0,3)--(1,3);
			\draw (0,2)--(1,2);
					\end{tikzpicture}
					
				};
				\node at (6.25,1) {\begin{tikzpicture}[scale=0.15]
			\draw (4,0)--(0,0)--(0,1)--(5,1) --(5,0)--(4,0);
			\draw (1,0)--(1,1);
			\draw (4,0)--(4,1);
			\draw (3,0)--(3,1);
			\draw (2,0)--(2,1);
					\end{tikzpicture}
					
				};
		\end{tikzpicture}}
	\end{matrix}$\\\hline\hline
$\begin{matrix}
{\begin{tikzpicture}[scale=0.15]
\draw (0,1)--(0,5)--(1,5)--(1,1)--(2,1)--(2,0)--(6,0)--(6,-1)--(1,-1)--(1,0)--(0,0)--(0,1);
			\draw (0,2)--(0,1)--(1,1)--(1,0)--(2,0);
			\draw (0,4)--(1,4);
			\draw (0,3)--(1,3);
			\draw (0,2)--(1,2);
			\draw (2,0)--(2,-1);
			\draw (3,0)--(3,-1);
			\draw (4,0)--(4,-1);
			\draw (5,0)--(5,-1);
\end{tikzpicture}}
\end{matrix}$&		
$\begin{matrix}\scalebox{0.9}{
	\begin{tikzpicture}
		\draw (0,0) to [out=30,in=150] (2,0);
		\draw (1.5,0) to [out=30,in=150] (3.5,0);
		\draw (3,0) to [out=30,in=150] (5,0);
		\draw (4.5,0) to [out=30,in=150] (6.5,0);
		\node at (0.25,1) {\begin{tikzpicture}[scale=0.15]
			\draw (0,4)--(0,0)--(1,0)--(1,5) --(0,5)--(0,4);
			\draw (0,1)--(1,1);
			\draw (0,4)--(1,4);
			\draw (0,3)--(1,3);
			\draw (0,2)--(1,2);
			\end{tikzpicture}
			
		};
		\node at (4.75,1) {\begin{tikzpicture}[scale=0.15]
		\draw (2,0)--(1,0)--(1,1)--(0,1)--(0,2)--(2,2)--(2,1)--(4,1)--(4,0)--(2,0);
		\draw (1,2)--(1,1)--(2,1)--(2,0);
		\draw (3,1)--(3,0);
			\end{tikzpicture}
			
		};
		\node at (3.25,1) {\begin{tikzpicture}[scale=0.15]
			\draw (0,1)--(0,0)--(1,0)--(1,-1)--(3,-1)--(3,0)--(2,0)--(2,1)--(1,1)--(1,2) --(0,2)--(0,1);
			\draw (0,1)--(1,1)--(1,0)--(2,0)--(2,-1);
			\end{tikzpicture}
			
		};
		\node at (1.75,1) {\begin{tikzpicture}[scale=0.15]
			\draw (0,2)--(0,0)--(1,0)--(1,-1)--(2,-1)--(2,1)--(1,1)--(1,3) --(0,3)--(0,2);
			\draw (0,1)--(1,1)--(1,0)--(2,0);
			\draw (0,2)--(1,2);
			\end{tikzpicture}
			
		};
		\node at (6.25,1) {\begin{tikzpicture}[scale=0.15]
			\draw (4,0)--(0,0)--(0,1)--(5,1) --(5,0)--(4,0);
			\draw (1,0)--(1,1);
			\draw (4,0)--(4,1);
			\draw (3,0)--(3,1);
			\draw (2,0)--(2,1);
			\end{tikzpicture}
			
		};
\end{tikzpicture}}
\end{matrix}$&
$\begin{matrix}
\scalebox{1}{\begin{tikzpicture}[scale=0.15]
\draw (0,1)--(0,5)--(1,5)--(1,1)--(3,1)--(3,0)--(7,0)--(7,-1)--(2,-1)--(2,0)--(0,0)--(0,1);
\draw (0,2)--(1,2);
\draw (0,3)--(1,3);
\draw (0,4)--(1,4);
\draw (0,1)--(1,1)--(1,0);
\draw (2,1)--(2,0)--(3,0)--(3,-1);
\draw (4,0)--(4,-1);
\draw (5,0)--(5,-1);
\draw (6,0)--(6,-1);
\end{tikzpicture}}
\end{matrix}$&
$\begin{matrix}\scalebox{0.9}{
			\begin{tikzpicture}
				\draw (0,0) to [out=30,in=150] (2,0);
				\draw (1.5,0) to [out=30,in=150] (3.5,0);
				\draw (3,0) to [out=30,in=150] (5,0);
				\draw (4.5,0) to [out=30,in=150] (6.5,0);
				\node at (0.25,1) {\begin{tikzpicture}[scale=0.15]
			\draw (0,4)--(0,0)--(1,0)--(1,5) --(0,5)--(0,4);
			\draw (0,1)--(1,1);
			\draw (0,4)--(1,4);
			\draw (0,3)--(1,3);
			\draw (0,2)--(1,2);
					\end{tikzpicture}
					
				};[scale=0.15]
				\node at (4.75,1) {\begin{tikzpicture}[scale=0.15]
		\draw (3,0)--(2,0)--(2,1)--(0,1)--(0,2)--(3,2)--(3,1)--(4,1)--(4,0)--(3,0);
		\draw (2,2)--(2,1)--(3,1)--(3,0);
		\draw (1,1)--(1,2);
					\end{tikzpicture}
					
				};
				\node at (3.25,1) {\begin{tikzpicture}[scale=0.15]
			\draw (0,1)-- (0,0)--(2,0)--(2,-1)--(3,-1)--(3,1)--(1,1)--(1,2) --(0,2)--(0,1);
			\draw (0,1)--(1,1)--(1,0);
			\draw (2,1)--(2,0)--(3,0);
					\end{tikzpicture}
					
				};
				\node at (1.75,1) {\begin{tikzpicture}[scale=0.15]
			\draw (0,3)-- (0,0)--(2,0)--(2,1)--(1,1)--(1,4) --(0,4)--(0,3);
			\draw (0,1)--(1,1)--(1,0);
			\draw (0,3)--(1,3);
			\draw (0,2)--(1,2);
					\end{tikzpicture}
					
				};
				\node at (6.25,1) {\begin{tikzpicture}[scale=0.15]
			\draw (4,0)--(0,0)--(0,1)--(5,1) --(5,0)--(4,0);
			\draw (1,0)--(1,1);
			\draw (4,0)--(4,1);
			\draw (3,0)--(3,1);
			\draw (2,0)--(2,1);
					\end{tikzpicture}
					
				};
		\end{tikzpicture}}
	\end{matrix}$\\\hline\hline
$\begin{matrix}
\scalebox{1}{\begin{tikzpicture}[scale=0.15]
\draw (0,1)--(0,5)--(1,5)--(1,1)--(2,1)--(2,-2)--(6,-2)--(6,-3)--(1,-3)--(1,0)--(0,0)--(0,1);
\node at (0.5,5) {$\ $}; 
\draw (0,4) -- (1,4); 
\draw (0,3) -- (1,3); 
\draw (0,2) -- (1,2); 
\draw (0,1) -- (1,1)-- (1,0)-- (2,0);
\draw  (1,-1)-- (2,-1);
\draw  (1,-2)-- (2,-2)-- (2,-3);
\draw  (3,-2)-- (3,-3);
\draw  (4,-2)-- (4,-3);
\draw  (5,-2)-- (5,-3);
\end{tikzpicture}}
\end{matrix}$&
$\begin{matrix}\scalebox{0.9}{
			\begin{tikzpicture}
				\draw (0,0) to [out=30,in=150] (2,0);
				\draw (1.5,0) to [out=30,in=150] (3.5,0);
				\draw (3,0) to [out=30,in=150] (5,0);
				\draw (4.5,0) to [out=30,in=150] (6.5,0);
				\node at (0.25,1) {\begin{tikzpicture}[scale=0.15]
			\draw (0,4)--(0,0)--(1,0)--(1,5) --(0,5)--(0,4);
			\draw (0,1)--(1,1);
			\draw (0,4)--(1,4);
			\draw (0,3)--(1,3);
			\draw (0,2)--(1,2);
					\end{tikzpicture}
					
				};
				\node at (4.75,1) {\begin{tikzpicture}[scale=0.15]
			\draw (0,1)--(0,0)--(4,0)--(4,1)--(1,1)--(1,2) --(0,2)--(0,1);
			\draw (0,1)--(1,1)--(1,0);
			\draw (2,0)--(2,1);
			\draw (3,0)--(3,1);
					\end{tikzpicture}
					
				};
				\node at (3.25,1) {\begin{tikzpicture}[scale=0.15]
			\draw (0,1)--(0,0)--(3,0)--(3,1)--(1,1)--(1,3) --(0,3)--(0,1);
			\draw (0,1)--(1,1)--(1,0);
			\draw (0,2)--(1,2);
			\draw (2,0)--(2,1);
					\end{tikzpicture}
					
				};
				\node at (1.75,1) {\begin{tikzpicture}[scale=0.15]
				\draw (2,1)--(2,4)--(0,4)--(0,3)--(1,3)--(1,0)--(2,0)--(2,1);
				\draw (1,1)--(2,1);
				\draw (1,2)--(2,2);
				\draw (1,4)--(1,3)--(2,3);
					\end{tikzpicture}
					
				};
				\node at (6.25,1) {\begin{tikzpicture}[scale=0.15]
			\draw (4,0)--(0,0)--(0,1)--(5,1) --(5,0)--(4,0);
			\draw (1,0)--(1,1);
			\draw (4,0)--(4,1);
			\draw (3,0)--(3,1);
			\draw (2,0)--(2,1);
					\end{tikzpicture}
					
				};
		\end{tikzpicture}}
	\end{matrix}$&
$\begin{matrix}
	\scalebox{1}{\begin{tikzpicture}[scale=0.15]
	\draw (0,1)--(0,5)--(1,5)--(1,1)--(3,1)--(3,-1)--(7,-1)--(7,-2)--(2,-2)--(2,0)--(0,0)--(0,1);
\node at (0.5,5) {$\ $}; 
\draw (0,4) -- (1,4); 
\draw (0,3) -- (1,3); 
\draw (0,2) -- (1,2); 
\draw (0,1) -- (1,1)-- (1,0); 
\draw (2,1) -- (2,0)-- (3,0); 
\draw (2,-1) -- (3,-1)-- (3,-2); 
\draw  (4,-1)-- (4,-2);  
\draw  (5,-1)-- (5,-2);  
\draw  (6,-1)-- (6,-2); 
	\end{tikzpicture}}
\end{matrix}$&
$\begin{matrix}\scalebox{0.9}{
			\begin{tikzpicture}
				\draw (0,0) to [out=30,in=150] (2,0);
				\draw (1.5,0) to [out=30,in=150] (3.5,0);
				\draw (3,0) to [out=30,in=150] (5,0);
				\draw (4.5,0) to [out=30,in=150] (6.5,0);
				\node at (0.25,1) {\begin{tikzpicture}[scale=0.15]
			\draw (0,4)--(0,0)--(1,0)--(1,5) --(0,5)--(0,4);
			\draw (0,1)--(1,1);
			\draw (0,4)--(1,4);
			\draw (0,3)--(1,3);
			\draw (0,2)--(1,2);
					\end{tikzpicture}
					
				};
				\node at (4.75,1) {\begin{tikzpicture}[scale=0.15]
			\draw (0,1)--(0,0)--(4,0)--(4,1)--(1,1)--(1,2) --(0,2)--(0,1);
			\draw (0,1)--(1,1)--(1,0);
			\draw (2,0)--(2,1);
			\draw (3,0)--(3,1);
					\end{tikzpicture}
					
				};
				\node at (3.25,1) {\begin{tikzpicture}[scale=0.15]
				\draw (1,3)--(3,3)--(3,0)--(2,0)--(2,2)--(0,2)--(0,3)--(1,3);
				\draw (1,3)--(1,2);
				\draw (2,3)--(2,2)--(3,2);
				\draw (3,1)--(2,1);
					\end{tikzpicture}
					
				};
				\node at (1.75,1) {\begin{tikzpicture}[scale=0.15]
			\draw (0,3)--(0,0)--(2,0)--(2,1)--(1,1)--(1,4) --(0,4)--(0,3);
			\draw (0,1)--(1,1)--(1,0);
			\draw (0,3)--(1,3);
			\draw (0,2)--(1,2);
					\end{tikzpicture}
					
				};
				\node at (6.25,1) {\begin{tikzpicture}[scale=0.15]
			\draw (4,0)--(0,0)--(0,1)--(5,1) --(5,0)--(4,0);
			\draw (1,0)--(1,1);
			\draw (4,0)--(4,1);
			\draw (3,0)--(3,1);
			\draw (2,0)--(2,1);
					\end{tikzpicture}
					
				};
		\end{tikzpicture}}
	\end{matrix}$\\\hline\hline
	\end{tabular}}
	\caption{Description of the simple chambers for the action of $\Z/5\Z$.}
\end{figure}

As predicted by \Cref{teosimple}, the possible shapes for the chamber stairs of simple chambers are $ 8 = 2 ^ {5-2} $, and there are 5 different ways to label each chamber stair.
\end{example}
\begin{example}In this example we treat the case $G\cong \Z/4\Z$.
	
	The following picture contains a list of the possible shapes of the abstract chamber stairs and, in each case, the shapes of the $G$-stairs associated to the toric $G$-constellations belonging to the respective chamber.
	\begin{figure}[H]
		\begin{tabular}{||c|c||c|c||}
			\hline\hline
			$\begin{matrix}
				\scalebox{1}{\begin{tikzpicture}[scale=0.15]
			\draw (0,1)--(0,0)--(4,0)--(4,1)--(1,1)--(1,4)--(0,4)--(0,1);
			\draw (0,1)--(1,1)--(1,0);
			\draw (2,0)--(2,1);
			\draw (3,0)--(3,1);
			\draw (0,3)--(1,3);
			\draw (0,2)--(1,2);
				\end{tikzpicture}}
			\end{matrix}$&
			$\begin{matrix}\scalebox{0.9}{
					\begin{tikzpicture}
						\draw (0,0) to [out=30,in=150] (2,0);
						\draw (1.5,0) to [out=30,in=150] (3.5,0);
						\draw (3,0) to [out=30,in=150] (5,0);
						\node at (0.25,1) {\begin{tikzpicture}[scale=0.15]
			\draw (0,3)--(0,0)--(1,0)--(1,4) --(0,4)--(0,3);
			\draw (0,1)--(1,1);
			\draw (0,3)--(1,3);
			\draw (0,2)--(1,2);
							\end{tikzpicture}
							
						};
						\node at (1.75,1) {\begin{tikzpicture}[scale=0.15]
			\draw (0,2)--(0,0)--(2,0)--(2,1)--(1,1)--(1,3) --(0,3)--(0,2);
			\draw (0,1)--(1,1)--(1,0);
			\draw (0,2)--(1,2);
							\end{tikzpicture}
							
						};
						\node at (3.25,1) {\begin{tikzpicture}[scale=0.15]
			\draw (0,1)--(0,0)--(3,0)--(3,1)--(1,1)--(1,2) --(0,2)--(0,1);
			\draw (0,1)--(1,1)--(1,0);
			\draw (2,0)--(2,1);
							\end{tikzpicture}
							
						};
						\node at (4.75,1) {\begin{tikzpicture}[scale=0.15]
			\draw (3,0)--(0,0)--(0,1)--(4,1) --(4,0)--(3,0);
			\draw (1,0)--(1,1);
			\draw (3,0)--(3,1);
			\draw (2,0)--(2,1);
							\end{tikzpicture}
							
						};
				\end{tikzpicture}}
			\end{matrix}$&
		$\begin{matrix}
		\scalebox{1}{\begin{tikzpicture}[scale=0.15]
			\draw (0,1)--(0,0)--(0,-1)--(4,-1)--(4,0)--(1,0)-- (1,2)--(0,2)--(0,5) --(-1,5)--(-1,1)--(0,1);
			\draw (0,2)--(0,1)--(1,1)--(1,0)--(2,0);
			\draw (0,0)--(1,0)--(1,-1);
			\draw (0,4)--(-1,4);
			\draw (0,3)--(-1,3);
			\draw (0,2)--(-1,2);
			\draw (2,0)--(2,-1);
			\draw (3,0)--(3,-1);
		\end{tikzpicture}}
	\end{matrix}$&
$\begin{matrix}\scalebox{0.9}{
	\begin{tikzpicture}
		\draw (0,0) to [out=30,in=150] (2,0);
		\draw (1.5,0) to [out=30,in=150] (3.5,0);
		\draw (3,0) to [out=30,in=150] (5,0);
		\node at (0.25,1) {\begin{tikzpicture}[scale=0.15]
			\draw (0,3)--(0,0)--(1,0)--(1,4) --(0,4)--(0,3);
			\draw (0,1)--(1,1);
			\draw (0,3)--(1,3);
			\draw (0,2)--(1,2);
			\end{tikzpicture}
			
		};
		\node at (1.75,1) {\begin{tikzpicture}[scale=0.15]
				\draw (2,1)--(2,4)--(0,4)--(0,3)--(1,3)--(1,1)--(2,1)--(2,1);
				\draw (1,2)--(2,2);
				\draw (1,4)--(1,3)--(2,3);
			\end{tikzpicture}
			
		};
		\node at (3.25,1) {\begin{tikzpicture}[scale=0.15]
			\draw (0,1)--(0,0)--(3,0)--(3,1)--(1,1)--(1,2) --(0,2)--(0,1);
			\draw (0,1)--(1,1)--(1,0);
			\draw (2,0)--(2,1);
			\end{tikzpicture}
			
		};
		\node at (4.75,1) {\begin{tikzpicture}[scale=0.15]
			\draw (3,0)--(0,0)--(0,1)--(4,1) --(4,0)--(3,0);
			\draw (1,0)--(1,1);
			\draw (3,0)--(3,1);
			\draw (2,0)--(2,1);
			\end{tikzpicture}
			
		};
\end{tikzpicture}}
\end{matrix}$\\\hline\hline
$\begin{matrix}
\scalebox{1}{\begin{tikzpicture}[scale=0.15]
\draw (0,1)--(0,4)--(1,4)--(1,1)--(3,1)--(3,0)--(6,0)--(6,-1)--(2,-1)--(2,0)--(0,0)--(0,1);
\draw (0,2)--(1,2);
\draw (0,3)--(1,3);
\draw (0,1)--(1,1)--(1,0);
\draw (2,1)--(2,0)--(3,0)--(3,-1);
\draw (4,0)--(4,-1);
\draw (5,0)--(5,-1);
\end{tikzpicture}}
\end{matrix}$&
$\begin{matrix}\scalebox{0.9}{
	\begin{tikzpicture}
		\draw (0,0) to [out=30,in=150] (2,0);
		\draw (1.5,0) to [out=30,in=150] (3.5,0);
		\draw (3,0) to [out=30,in=150] (5,0);
		\node at (0.25,1) {\begin{tikzpicture}[scale=0.15]
			\draw (0,3)--(0,0)--(1,0)--(1,4) --(0,4)--(0,3);
			\draw (0,1)--(1,1);
			\draw (0,3)--(1,3);
			\draw (0,2)--(1,2);
			\end{tikzpicture}
			
		};
		\node at (1.75,1) {\begin{tikzpicture}[scale=0.15]
			\draw (0,2)--(0,0)--(2,0)--(2,1)--(1,1)--(1,3) --(0,3)--(0,2);
			\draw (0,1)--(1,1)--(1,0);
			\draw (0,2)--(1,2);
			\end{tikzpicture}
			
		};
		\node at (3.25,1) {\begin{tikzpicture}[scale=0.15]
				\draw (1,3)--(3,3)--(3,1)--(2,1)--(2,2)--(0,2)--(0,3)--(1,3);
				\draw (1,3)--(1,2);
				\draw (2,3)--(2,2)--(3,2);
			\end{tikzpicture}
			
		};
		\node at (4.75,1) {\begin{tikzpicture}[scale=0.15]
			\draw (3,0)--(0,0)--(0,1)--(4,1) --(4,0)--(3,0);
			\draw (1,0)--(1,1);
			\draw (3,0)--(3,1);
			\draw (2,0)--(2,1);
			\end{tikzpicture}
			
		};
\end{tikzpicture}}
\end{matrix}$&
$\begin{matrix}
\scalebox{1}{\begin{tikzpicture}[scale=0.15]
\draw (0,1)--(0,4)--(1,4)--(1,1)--(2,1)--(2,0)--(5,0)--(5,-1)--(1,-1)--(1,0)--(0,0)--(0,1);
			\draw (0,2)--(0,1)--(1,1)--(1,0)--(2,0);
			\draw (0,3)--(1,3);
			\draw (0,2)--(1,2);
			\draw (2,0)--(2,-1);
			\draw (3,0)--(3,-1);
			\draw (4,0)--(4,-1);
\end{tikzpicture}}
\end{matrix}$&
$\begin{matrix}\scalebox{0.9}{
	\begin{tikzpicture}
		\draw (0,0) to [out=30,in=150] (2,0);
		\draw (1.5,0) to [out=30,in=150] (3.5,0);
		\draw (3,0) to [out=30,in=150] (5,0);
		\node at (0.25,1) {\begin{tikzpicture}[scale=0.15]
			\draw (0,3)--(0,0)--(1,0)--(1,4) --(0,4)--(0,3);
			\draw (0,1)--(1,1);
			\draw (0,3)--(1,3);
			\draw (0,2)--(1,2);
			\end{tikzpicture}
			
		};
		\node at (1.75,1) {\begin{tikzpicture}[scale=0.15]
		\draw (2,1)--(2,2)--(1,2)--(1,3)--(0,3)--(0,1)--(1,1)--(1,0)--(2,0)--(2,1);
		\draw (2,1)--(1,1)--(1,2)--(0,2);
			\end{tikzpicture}
			
		};
		\node at (3.25,1) {\begin{tikzpicture}[scale=0.15]
		\draw (1,2)--(2,2)--(2,1)--(3,1)--(3,0)--(1,0)--(1,1)--(0,1)--(0,2)--(1,2);
		\draw (1,2)--(1,1)--(2,1)--(2,0);
			\end{tikzpicture}
			
		};
		\node at (4.75,1) {\begin{tikzpicture}[scale=0.15]
			\draw (3,0)--(0,0)--(0,1)--(4,1) --(4,0)--(3,0);
			\draw (1,0)--(1,1);
			\draw (3,0)--(3,1);
			\draw (2,0)--(2,1);
			\end{tikzpicture}
			
		};
\end{tikzpicture}}
\end{matrix}$\\\hline\hline
$\begin{matrix}
\scalebox{1}{\begin{tikzpicture}[scale=0.15]
			\draw (0,1)--(0,0)--(1,0)--(1,-1)--(5,-1)--(5,0)--(2,0)--(2,1)--(1,1)--(1,2)--(0,2)--(0,5) --(-1,5)--(-1,1)--(0,1);
			\draw (0,2)--(0,1)--(1,1)--(1,0)--(2,0);
			\draw (0,4)--(-1,4);
			\draw (0,3)--(-1,3);
			\draw (0,2)--(-1,2);
			\draw (2,0)--(2,-1);
			\draw (3,0)--(3,-1);
			\draw (4,0)--(4,-1);
\end{tikzpicture}}
\end{matrix}$&
$\begin{matrix}\scalebox{0.9}{
	\begin{tikzpicture}
		\draw (0,0) to [out=30,in=150] (2,0);
		\draw (1.5,0) to [out=30,in=150] (3.5,0);
		\draw (3,0) to [out=30,in=150] (5,0);
		\node at (0.25,1) {\begin{tikzpicture}[scale=0.15]
			\draw (0,3)--(0,0)--(1,0)--(1,4) --(0,4)--(0,3);
			\draw (0,1)--(1,1);
			\draw (0,3)--(1,3);
			\draw (0,2)--(1,2);
			\end{tikzpicture}
			
		};
		\node at (1.75,1) {\begin{tikzpicture}[scale=0.15]
		\draw (2,1)--(2,2)--(1,2)--(1,3)--(0,3)--(0,1)--(1,1)--(1,0)--(2,0)--(2,1);
		\draw (2,1)--(1,1)--(1,2)--(0,2);
			\end{tikzpicture}
			
		};
		\node at (3.25,1) {\begin{tikzpicture}[scale=0.15]
		\draw (1,2)--(2,2)--(2,1)--(3,1)--(3,0)--(1,0)--(1,1)--(0,1)--(0,2)--(1,2);
		\draw (1,2)--(1,1)--(2,1)--(2,0);
			\end{tikzpicture}
			
		};
		\node at (4.75,1) {\begin{tikzpicture}[scale=0.15]
			\draw (3,0)--(0,0)--(0,1)--(4,1) --(4,0)--(3,0);
			\draw (1,0)--(1,1);
			\draw (3,0)--(3,1);
			\draw (2,0)--(2,1);
			\end{tikzpicture}
			
		};
\end{tikzpicture}}
\end{matrix}$&
$\begin{matrix}
\scalebox{1}{\begin{tikzpicture}[scale=0.15]
\draw (0,1)--(0,4)--(1,4)--(1,1)--(2,1)--(2,-1)--(4,-1)--(4,-2)--(7,-2)--(7,-3)--(3,-3)--(3,-2)--(1,-2)--(1,0)--(0,0)--(0,1);
\node at (0.5,4) {$\ $}; 
\draw (0,3)--(1,3);
\draw (0,2)--(1,2);
\draw (0,1)--(1,1)--(1,0)--(2,0);
\draw (1,-1)--(2,-1)--(2,-2);
\draw (3,-1)--(3,-2)--(4,-2)--(4,-3);
\draw (5,-2)--(5,-3);
\draw (6,-2)--(6,-3);
\end{tikzpicture}}
\end{matrix}$&
$\begin{matrix}\scalebox{0.9}{
	\begin{tikzpicture}
		\draw (0,0) to [out=30,in=150] (2,0);
		\draw (1.5,0) to [out=30,in=150] (3.5,0);
		\draw (3,0) to [out=30,in=150] (5,0);
		\node at (0.25,1) {\begin{tikzpicture}[scale=0.15]
			\draw (0,3)--(0,0)--(1,0)--(1,4) --(0,4)--(0,3);
			\draw (0,1)--(1,1);
			\draw (0,3)--(1,3);
			\draw (0,2)--(1,2);
			\end{tikzpicture}
			
		};
		\node at (1.75,1) {\begin{tikzpicture}[scale=0.15]
				\draw (2,1)--(2,4)--(0,4)--(0,3)--(1,3)--(1,1)--(2,1)--(2,1);
				\draw (1,2)--(2,2);
				\draw (1,4)--(1,3)--(2,3);
			\end{tikzpicture}
			
		};
		\node at (3.25,1) {\begin{tikzpicture}[scale=0.15]
				\draw (1,3)--(3,3)--(3,1)--(2,1)--(2,2)--(0,2)--(0,3)--(1,3);
				\draw (1,3)--(1,2);
				\draw (2,3)--(2,2)--(3,2);
			\end{tikzpicture}
			
		};
		\node at (4.75,1) {\begin{tikzpicture}[scale=0.15]
			\draw (3,0)--(0,0)--(0,1)--(4,1) --(4,0)--(3,0);
			\draw (1,0)--(1,1);
			\draw (3,0)--(3,1);
			\draw (2,0)--(2,1);
			\end{tikzpicture}
			
		};
\end{tikzpicture}}
\end{matrix}$\\\hline\hline
		\end{tabular}
		\caption{Description of the chambers for the action of $\Z/4\Z$.}
\label{camerez4}
	\end{figure}
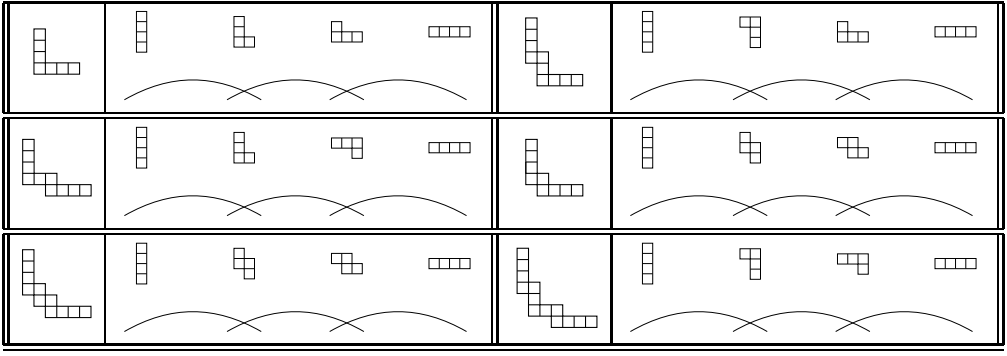

Notice that the first $4=2^{4-2}$ pictures correspond to simple chambers. Moreover, as predicted by \Cref{TEO1}, the possible shapes for the chamber stairs are $ 6 = (4-1) ! $, and there are 4 different ways to label each chamber stair.

Note also that, after having labeled each box appropriately, the first and last chambers in \Cref{camerez4} correspond to $C_G$ and $C_G^{\OP}$ respectively (see \Cref{hilbop}).
\end{example}
\section{The costruction of the tautological bundles \texorpdfstring{$\Calr_C$}{}}
The quasi projective variety $\Calm_C$ is a fine moduli space obtained by GIT as described in \cite{KING} by King. In particular, there exists a universal family $\Calu_C\in\Ob\Coh(\Calm_C\times\A^2)$. The tautological bundle is the pushforward 
$$\Calr_C=(\pi_{\Calm_C})_*(\Calu_C).$$
It is a vector bundle of rank $k=|G|$ whose fibers are $G$-constellations and, more precisely, over each point $[\Calf]\in\Calm_C$ the fiber $(\Calr_C)_{[\Calf]}$ is canonically isomorphic to the space of global sections $H^0(\A^2,\Calf)$.

In this section we give an explicit construction of the tautological bundles $\Calr_C$ for all chambers $C\subset\Theta^{\gen}$ in terms of their chamber stairs. We will adopt the same notation as in \Cref{toric resolution}.

\begin{notation}
From now on, given a coherent monomial ideal sheaf $\Calk\subset\Calo_{\A^2}$, we denote by $\widetilde{\Calk}$ the $\Calo_Y$-module defined by
\[
\widetilde{\Calk}= \varepsilon^*\pi_*\Calk/\tor_{\Calo_Y}\varepsilon^*\pi_*\Calk.
\]
\end{notation}
\begin{lemma}\label{GSVsullecarte} Suppose that $\Calk$ is generated by the monomials $x^{\alpha_1}y^{\beta_1},\ldots,x^{\alpha_s}y^{\beta_s}$. Then, on each toric chart $U_j\subset Y$ with coordinates $(a_j,c_j)$, the sheaf $\widetilde{\Calk}$ agrees with the sheaf $\Calh_j$ associated to the $\C[a_j,c_j]$-module:
	$$H_j= \frac{K_j}{K_j\cap I_j}\subset \frac{\C[a_j,c_j,x,y]}{K_j\cap I_j},$$
	where $K_j$ and $I_j$ are the ideals of $\C[a_j,c_j,x,y]$ given by
	$$K_j=(x^{\alpha_1}y^{\beta_1},\ldots,x^{\alpha_s}y^{\beta_s})$$
	and
	$$I_j=(a_jy^{k-j}-x^{j},c_jx^{j-1}-y^{k-j+1},a_jc_j-xy),$$
	and the gluings on the intersections $U_i\cap U_j$, for $ 1\le i ,j \le k $, are given by:
\[
\begin{tikzcd}[row sep=tiny]
\Gamma(U_i\cap U_j,\Calh_i) \arrow{r}{\varphi_{ij}} &  \Gamma(U_i\cap U_j,\Calh_j) \\
x\arrow[mapsto]{r} & x,\\
y\arrow[mapsto]{r} & y,\\
a_i\arrow[mapsto]{r} & a_j^{i-j+1}c_j^{i-j},\\
c_i\arrow[mapsto]{r} & a_j^{j-i}c_j^{j-i+1}.
\end{tikzcd}
\]
\end{lemma}
\begin{proof} The proof is achieved by direct computation, after noticing that the gluings on the intersections are deduced from the toric description of the toric quasiprojective variety $\Calm_C$ given at the beginning of \Cref{toric resolution} and, in particular, from Equations \eqref{toricrelations}.
\end{proof}
\begin{remark}
     Using the the relations given in \eqref{toricrelations}, the modules $H_j$, for $j=1,\ldots,k$, can be regarded as $\C[a_j , c_j ]$-submodules of the rational function field $\C(x, y)$.
\end{remark}
\begin{remark}\label{GSFDGIUSTO}
 If $x^{\alpha_1}y^{\beta_1},\ldots,x^{\alpha_s}y^{\beta_s}$ are the generators of some $C$-stair $\Gamma_C$ and $\Calk$ is defined as in \Cref{GSVsullecarte}, all the $G$-sFd associated to the toric fibers of $\widetilde{\Calk}$ are substairs of $\Gamma_C$. This is a consequence of Nakayama's Lemma together with the following three facts:
		\begin{equation}\label{app}
		\begin{array}{cc}
		    \forall j=1,\ldots,k,\forall i=1,\ldots,s& x^{\alpha_i+1}y^{\beta_i+1}\in(K_j\cap I_j)+(a_j,c_j),\\
		    \forall j=1,\ldots,k,& x^{\alpha_1}y^{\beta_1+k}\in(K_j\cap I_j)+(a_j,c_j),\\
		\forall j=1,\ldots,k,& x^{\alpha_s+k}y^{\beta_s}\in(K_j\cap I_j)+(a_j,c_j).
		\end{array}
		\end{equation}
		The relations \eqref{app} follow from the easy observations that
		$$x^{\alpha_i}y^{\beta_i}\cdot(a_jc_j-xy)=a_jc_jx^{\alpha_i}y^{\beta_i}-x^{\alpha_i+1}y^{\beta_i+1}\in K_j\cap I_j,$$
		$$y^{j-1}\cdot x^{\alpha_1}y^{\beta_1}\cdot(c_jx^{j-1}-y^{k-j+1})=c_jx^{\alpha_1+j-1}y^{\beta_1+j-1}-x^{\alpha_1}y^{\beta_1+k}\in K_j\cap I_j,$$
		$$x^{k-j}\cdot x^{\alpha_s}y^{\beta_s}\cdot(a_jy^{k-j}-x^{j})=a_jx^{\alpha_s+k-j }y^{\beta_s+k-j}-x^{\alpha_s+k}y^{\beta_s}\in K_j\cap I_j.$$
\end{remark}
In this last part of the paper we state and prove the last main theorem. Before to give the proof, we   also state and prove some corollaries and results needed in the proof.
\begin{theorem}\label{costruzione}
	Let $C\subset\Theta^{\gen}$ be a chamber and let $\Gamma_C\subset\Calt_G$ be a $C$-stair. Suppose that $\Gamma_C$ has $s\ge1$ ordered (see \Cref{orderbox}) generators $v_1,\ldots,v_s$ with associated monomials 
	$$x^{\alpha_1}y^{\beta_1},\ldots,x^{\alpha_s}y^{\beta_s}\in\C[x,y].$$ 
	Consider the ideal sheaf $\Calk= (x^{\alpha_1}y^{\beta_1},\ldots,x^{\alpha_s}y^{\beta_s})\Calo_{\A^2}$, then  
	$$\Calr_C\cong\varepsilon^*\pi_*\Calk /\tor_{\Calo_{\Calm_C}}(\varepsilon^*\pi_*\Calk).$$
\end{theorem}
The following corollaries are  direct consequences of \Cref{costruzione}.
\begin{corollary}\label{sullecarte} On each toric chart $U_j\subset\Calm_C$ with coordinates $(a_j,c_j)$, the tautological bundle ${\Calr_C}_{|_{U_j}}$ agrees with the sheaf $\Calh_j$ associated to the $\C[a_j,c_j]$-module $H_j$ in \Cref{GSVsullecarte}.
\end{corollary}
 \begin{corollary}\label{nuovo GSV}
	    In the hypotheses of \Cref{costruzione}, the $\Calo_Y$-module $\widetilde{\Calk}$ is locally free of rank $|G|$.
	\end{corollary} 
\begin{corollary}\label{quot}
Let $C$ and $\Calk$ be as in \Cref{costruzione}. Then, $\Calm_C$ can be identified with a closed $G$-invariant subvariety of $\Quot_{\Calk}^{|G|}(\A^2)$.
\end{corollary}
\begin{remark}
    \Cref{nuovo GSV} is a generalisation of {\cite[Proposition 2.4.]{GSV}} in the abelian setting. 
\end{remark}
\begin{remark}\label{GSVhilb} For the trivial ideal $K=(1)=\C[x,y]$  \Cref{sullecarte} recovers Nakamura's description of the $G$-Hilbert scheme when $G$ is abelian (see \cite{NAKAMURA}).
\end{remark}
\begin{remark}\label{onefibreisok} Notice that, over the origin of the first and the last charts, the  $\Calo_{U_1}$-module $\Calh_1$ and the $\Calo_{U_k}$-module $\Calh_k$ have, as toric fibers, the expected $G$-constellations $\Calf_1$ and $\Calf_k$, i.e 	
$$\Calf_1\cong {\Calh_1}_{0_1}\cong\frac{(x^{\alpha_1}y^{\beta_1})}{(x^{\alpha_1}y^{\beta_1+k},x^{\alpha_1+1}y^{\beta_1})}\subset \frac{\C[x,y]}{(x^{\alpha_1}y^{\beta_1+k},x^{\alpha_1+1}y^{\beta_1})}$$
and
$$\Calf_k\cong {\Calh_k}_{0_k}\cong\frac{(x^{\alpha_s}y^{\beta_s})}{(x^{\alpha_s+k}y^{\beta_s},x^{\alpha_s}y^{\beta_s+1})}\subset \frac{\C[x,y]}{(x^{\alpha_s+k}y^{\beta_s},x^{\alpha_s}y^{\beta_s+1})},$$
where $0_i\in U_i$ is, for $i=1,k$, the origin.

We prove this only for the origin of $U_k$, the other proof is similar. We start by showing that 
$$x^{\alpha_i}y^{\beta_i}\in (K_k\cap I_k)+(a_k,c_k) \ \mbox{ for }i=1,\ldots,s-1. $$
Notice that, for all $ i=1,\ldots,s-1$, we have
$$\alpha_i\ge0,\quad\beta_i>\beta_{i+1}>\beta_{s}\ge 0,\quad\alpha_{i}+k-1\ge \alpha_{i+1} .$$
Therefore, we can write:
$$c_kx^{\alpha_i+k-1}y^{\beta_i-1}-x^{\beta_i}y^{\alpha_i}=
\begin{cases}
	c_kx^{\alpha_i+k-1-\alpha_{i+1}}y^{\beta_i-1-\beta_{i+1}}(x^{\alpha_{i+1}}y^{\beta_{i+1}})-x^{\alpha_i}y^{\beta_i}\\
	(x^{\alpha_i}y^{\beta_i-1})(c_kx^{k-1}-y),
\end{cases}$$
which implies
$$x^{\alpha_i}y^{\beta_i}\in(K_k\cap I_k )+(a_k,c_k)\ \forall i=1,\ldots,s-1.$$ 
Now, we have
$$K_k\cap I_k+(a_k,c_k)=(x^{\alpha_s}y^{\beta_s})\cap I_k+(a_k,c_k)=(x^{\alpha_s}y^{\beta_s})\cdot I_k+(a_k,c_k)=(x^{\alpha_s+k}y^{\beta_s},x^{\alpha_s}y^{\beta_s+1},a_k,c_k),$$
which gives
\[
 {\Calh_k}_{0_k}\cong \frac{(x^{\alpha_s}y^{\beta_s})}{(x^{\alpha_s+k}y^{\beta_s},x^{\alpha_s}y^{\beta_s+1},a_k,c_k)}\subset \frac{\C[x,y,a_k,c_k]}{(x^{\alpha_s+k}y^{\beta_s},x^{\alpha_s}y^{\beta_s+1},a_k,c_k)}.
\]
\end{remark}
\begin{definition}\label{defajcj}
Let $K\subset\C[x,y]$ be the ideal generated by the (ordered) set of monomials 
\[
\Set{x^{\alpha_i}y^{\beta_i}|i=1,\ldots,s}
\] associated to the generators of some chamber stair $\Gamma_C$ and let $\Gamma_K=\Set{(m,i)\in \Calt_G|m\in K}$ be the subset of the representation tableau corresopnding to $K$. Given a monomial $m_b\in K$ corresponding to a box $b\in \Gamma_C\subset \Gamma_K$, we   say that:
\begin{itemize}
    \item the property \textbf{(A$_j$)} holds for $m_b$ (or for $b$) if 
    \[x^{-j}y^{k-j}\cdot m_b \in\Gamma_K,\]
    \item the property \textbf{(C$_j$)} holds for $m_b$ (or for $b$) if 
    \[x^{j-1}y^{-k+j-1}\cdot m_b \in\Gamma_K.\]
\end{itemize}
\end{definition}
\begin{lemma}\label{fine}
If the property \textbf{(A$_j$)} (resp. \textbf{(C$_j$)}) holds for a box $b\in\Gamma_C$ then it holds also for the box after (resp. before) $b$. 
\end{lemma}
\begin{proof}
Let $m_b=x^\alpha y^\beta$ be the monomial associated to the box $b$. From \Cref{defajcj}, it follows immediately that, if the property \textbf{(A$_j$)} (resp. \textbf{(C$_j$)}) holds for $b$, then it holds for all the monomials $x^\gamma y^\delta$ such that $\gamma \ge\alpha$ and $\delta\ge \beta$. This proves the Lemma in the case in which the box after (resp. before) $b$ is on the right (resp. above) $b$.

We prove the remaining case for the property \textbf{(C$_j$)} and we leave the similar proof for \textbf{(A$_j$)}. We have to prove that, if two monomials of the form $x^\alpha y^\beta,x^{\alpha-1} y^\beta$ correspond to some successive boxes in $\Gamma_C$ and the property \textbf{(C$_j$)} holds for $x^\alpha y^\beta$ then it holds also for $x^{\alpha-1} y^\beta$. In other words, we suppose that 
\[m_1=x^{\alpha +j-1} y^{\beta-k+j-1}\in K,\] 
and we want to prove that  
\[m_2=x^{\alpha +j-2} y^{\beta-k+j-1}\in K.\]
Let $b_1,b_2$ be the boxes corresponding to $m_1,m_2$ and let $b$ be the box corresponding to $x^{\alpha-1}y^\beta$. If $b_1\in\Gamma_K\smallsetminus\Gamma_C$ it follows easily that $b_2\in\Gamma_K$. Suppose  $b_1\in \Gamma_C$ and consider the connected substair $\Gamma\subset\Gamma_C$ whose first box is $b$ and whose last box is $b_1$. We have, by construction,
\[\mathfrak{w}(\Gamma)=j\mbox{ and }\mathfrak{h}(\Gamma)=k-j+2,\]
which imply that $\Gamma$ contains $k+1$ boxes.

Let $\Gamma'=\Gamma\smallsetminus \{b_1\}$ be the connected $G$-substair of $\Gamma_C$ obtained by removing the last box from $\Gamma$ and let $b'\in\Gamma_C$ be the last box of $\Gamma'$. Now, by construction, $b$ is a vertical left cut for $\Gamma'$ in $\Gamma_C$ and, as a consequence of \Cref{CONCLUSION} also $b'$ is a vertical cut. Therefore $b'$ must correspond to the monomial $m_2$ from which it directly follows
\[b'=b_2\in \Gamma_C. \]
Which implies the thesis.
\end{proof}
\begin{proof}( \textit{of \Cref{costruzione}} ). If we endow the product $\Calm_C\times \A^2$ with the $G$-action defined by
	\[
\begin{tikzcd}[row sep=tiny]
G\times\Calm_C\times \A^2 \arrow{r} &\Calm_C\times \A^2 \\
(g_k^i,p,(x,y))\arrow[mapsto]{r} & (p,(\xi_k^{-i}x,\xi_k^{i}y)).
\end{tikzcd}
\]
	where $g_k$ is the (fixed) generator of the cyclic group $G$ (see subsection \ref{section2.1}), it turns out that the $\Calo_{\Calm_C\times \A^2}$-module $\widetilde{\Calk}$ is $G$-equivariant with respect to this action.
	
	To prove the theorem, we   use the description of $\widetilde{\Calk}$ given in \Cref{sullecarte}.
	We know from \Cref{freunic} that the tautological bundles $\Calr_C$ and $\Calr_{C_G}$ agree on the complement $U_C$ of the exceptional locus of $\Calm_C$. Moreover, we have, as a consequence of the construction of $\widetilde{\Calk}$ and of \Cref{GSVhilb}, isomorphisms
	\[
	{\Calr_{C}}_{|_{U_C}}\cong{\Calr_{C_G}}_{|_{U_C}}\cong{\widetilde{\Calk}}_{|_{U_C}}\cong \Calo_{U_C}^{\oplus k}.
	\]
	
	Now we show that the fibers of $\Calr_C$ and $\widetilde{\Calk}$ over the toric points of $\Calm_C$ are the same $G$-constellations. This will be enough to prove the statement, because each chamber is uniquely identified by its toric $G$-constellations. We split this part in several steps:
	\begin{itemize}
		\item[STEP 0] Over each point of $p\in \Calm_C$ the fiber $\widetilde{\Calk}_p$ is a $G$-equivariant $\C[x,y]$-module and, over each origin $0_j\in U_j$ the fibre $\widetilde{\Calk}_{0_j}$ is also $\T^2$-equivariant. This follows from the fact that the ideal $K_j$ is generated by monomials and that the ideal $I_j$ is generated by $G$-eigenbinomials (recall that the group $G$ acts trivially on $U_j$) of positive degrees in the variables $a_j,c_j$.
		\item[STEP 1] All the $G$-sFd associated to the toric fibers of $\widetilde{\Calk}$ are substairs of the $C$-stair $\Gamma_C$. For this, see \Cref{GSFDGIUSTO}.
		\item[STEP 2] For all $j=1,\ldots,k$, the $j$-th torus equivariant $G$-module $\widetilde{\Calk}_{0_j}$ is indecomposable. Let $\Gamma_j\subset\Gamma_C$ be the $G$-sFd associated to $\widetilde{\Calk}_{0_j}$. Then, the $G$-constellation $\widetilde{\Calk}_{0_j}$ is indecomposable if and only if $\Gamma_j$ is connected. 
		
		First observe that, for a box $b\in\Gamma_C$ both the properties \textbf{(A$_j$)} and \textbf{(C$_j$)} implies that the corresponding monomial $m_b$ belongs to $(K _j\cap I_j)+(a_j,c_j)$. This is true because, if $m_b=x^\alpha y^\beta$, then
		\begin{equation}
		\label{motivi}
		\begin{array}{cc}
		     \mbox{\textbf{(A$_j$)}}\Rightarrow & a_jx^{\alpha-j}y^{\beta +k-j}-x^\alpha y^\beta\in K_j\cap I_j,  \\
		     \mbox{\textbf{(C$_j$)}}\Rightarrow & c_jx^{\alpha+j-1}y^{\beta -k+j-1}-x^\alpha y^\beta\in K_j\cap I_j.
		\end{array}
		\end{equation}
		On the other hand, $b\in\Gamma_C\smallsetminus \Gamma_j$ if and only if $m_b\in (K_j \cap I_j)+(a_j,c_j)$. In particular, by construction, at least one of the following relations is true.
		\begin{enumerate}
		    \item $a_jx^{\alpha-j}y^{\beta +k-j}-x^\alpha y^\beta\in K_j\cap I_j$,
		    \item $  c_jx^{\alpha+j-1}y^{\beta -k+j-1}-x^\alpha y^\beta\in K_j\cap I_j$,
		    \item\label{3} $ a_jc_jx^{\alpha-1}y^{\beta -1}-x^\alpha y^\beta\in K_j\cap I_j$.
		\end{enumerate}
		Notice that $b\in\Gamma_C$ implies (see STEP 1) that  \eqref{3}  can not hold true. Therefore, given $b\in\Gamma_C$, it belongs to $\Gamma_j$ if and only if one among the two properties $\mbox{\textbf{(A$_j$)}}$ and $\mbox{\textbf{(C$_j$)}}$ holds for $b$. Now, the connectedness of $\Gamma_j$ is a consequence of \Cref{fine}.
		\item[STEP 3] 
    Let, for all $j=1,\ldots,k$, $\mathfrak m_j \subset \Calo_{\Calm_C,0_j}$ be the maximal ideal, and let 
  \[
  F_j= \widetilde{\Calk}_{0_j}/\mathfrak m_j =\widetilde{\Calk}_{0_j}\underset{\Calo_{\Calm_C,0_j}}{\otimes}(\Calo_{\Calm_C,0_j} /\mathfrak m_j) 
  \]
  be the fibre of the sheaf $\widetilde{\Calk}$ over the point $0_j$.  We show now that
  \[
  \dim_\C F_j= k.
  \]
  This implies, together with the previous step that, for all $j=1,\ldots,k$, the $G$-module $ \widetilde{\Calk}_{0_j} $ is a toric $G$-constellation.
  First notice that, by semicontinuity of the dimension of the fibers of a coherent sheaf (cf. \cite[Example 12.7.2]{HARTSHORNE}), we have
  \begin{equation}
      \label{primainequ}
  \dim_\C F_j\ge  k.
  \end{equation}
Let us suppose $j\ge 2$, the case $j=1$ was shown in \Cref{onefibreisok}.
  Let $\Gamma_j\subset\Gamma_C$ be, as in the previous step, the $G$-sFd associated to $\widetilde{\Calk}_{0_j}$, and let $x^\alpha y^\beta$ be the monomial in $\C[x,y]\subset\C[a_j,c_j,x,y]$ corresponding to the first box of $\Gamma_j$. Then, if the monomial $x^{\alpha+a}y^{\beta+b}$ corresponds to a box of $\Gamma_C$ for some $a\ge j$ and $j-k\le b \le 0$, it has the property  \textbf{(A$_j$)}. As in STEP 2, \Cref{fine} implies that $x^{\alpha+a}y^{\beta+b}\notin\Gamma_i$. Suppose that, $x^{\gamma}y^{\beta-k+j-1}\in\Gamma_i$ for some $\alpha+1\le\gamma\le \alpha+j$. Then, if we have $x^{\gamma'}y^\beta\in\Gamma_i$ for some $\gamma-j+1\le\gamma'\le \gamma$, the following relation
			\[
		    \begin{array}{cc}
		           c_jx^{\gamma'+j-1} y^{\beta -k+j-1}-x^{\gamma'} y^{\beta} \in K_j\cap I_j,
		    \end{array}
		    \]
		implies that $ x^{\gamma'} y^{\beta} \in K_j\cap I_j+(a_j,c_j)$. Now, by construction we have $\alpha-j+2\le\gamma'\le\alpha+j$ and we have fixed $j\ge 2$. Thus, $\gamma'=\alpha$ gives the contraddiction $x^\alpha y^\beta\notin\Gamma_i$.
  
	As a consequence, $x^{\gamma}y^{\beta-k+j-1}\notin\Gamma_i$ for all $\alpha+1\le\gamma\le \alpha+j$. Now, thanks to the connectedness proven in STEP 2, we have
		$$ \mathfrak{w}(\widetilde{\Calk}_{0_j})\le j \quad \mbox{ and }\quad \mathfrak{h}(\widetilde{\Calk}_{0_j})\le k-j+1, $$ which together imply
  \[
\dim_\C F_j=  \mathfrak{w}(\widetilde{\Calk}_{0_j})+ \mathfrak{h}(\widetilde{\Calk}_{0_j})-1\le k.
  \]
		The equality $\dim_\C F_j= k$ follows now from \Cref{primainequ}.
  
		\item[STEP 4] As an immediate consequence of the previous step, for all $j=1,\ldots,k$, the $G$-constellation $\widetilde{\Calk}_{0_j}$ has width $\mathfrak{w}(\widetilde{\Calk}_{0_j})= j$, see \Cref{hwchart}. Hence, for $j=1,\ldots,k$, they are different to each other.  
	\end{itemize}
Now, the above listed properties imply that $\widetilde{\Calk}$ is the tautological bundle $\Calr_{C'}$ of some chamber $C'\subset\Theta^{\gen}$ which admits $\Gamma_C$ as $C'$-stair and this, by \Cref{unicCstair}, implies $C'=C$.
\end{proof}
\begin{remark}
As expected, in dimension 3  \Cref{costruzione} is  in general false. For instance, given the $(\Z/2\Z)^2$-action over $\A^3$ defined by the inclusion
\[
\begin{tikzcd}[row sep=tiny]
(\Z/2\Z)^2\arrow{r}{ } &  \SL(3,\C) \\
(1,0)\arrow[mapsto]{r} & \begin{pmatrix} -1&0&0\\0&1&0\\0&0&-1
\end{pmatrix},\\
(0,1)\arrow[mapsto]{r} & \begin{pmatrix} 1&0&0\\0&-1&0\\0&0&-1
\end{pmatrix},
\end{tikzcd}\]
the quotient singularity $X=\A^3/(\Z/2\Z)^2$ admits four different crepant resolutions $\varepsilon_i:Y_i\rightarrow X$, for $i=1,\ldots,4$. All of them are toric and they are described by the planar graphs in \Cref{Z2Z2}.
    	\begin{figure}[H]
    	\scalebox{1}{
			\begin{tikzpicture}[scale=0.65]
	\draw (-1,2)--(1,2)--(0,0)--(-1,2);
	\draw (-2,0)--(2,0)--(0,4)--(-2,0);
	\fill (-2,0) circle (1.2pt);
	\fill (2,0) circle (1.2pt);
	\fill (0,4) circle (1.2pt);
	\fill (-1,2) circle (1.2pt);
	\fill (0,0) circle (1.2pt);
	\fill (1,2) circle (1.2pt);
	\node at (-2.2,-0.2) {$e_1$};
	\node at (2.2,-0.2) {$e_2$};
	\node at (0,4.2) {$e_3$};
	\node at (1.2,2.2) {$v_1$};
	\node at (-1.3,2.2) {$v_2$};
	\node at (0,-0.3) {$v_3$};
	\node at (0,-1.5) {$Y_1$};
	\end{tikzpicture}
			}
			\scalebox{1}{
			\begin{tikzpicture}[scale=0.65]
		\draw (-1,2)--(1,2)--(0,0);
		\draw (1,2)--(-2,0);
		\draw (-2,0)--(2,0)--(0,4)--(-2,0);
		\fill (-2,0) circle (1.2pt);
		\fill (2,0) circle (1.2pt);
		\fill (0,4) circle (1.2pt);
		\fill (-1,2) circle (1.2pt);
		\fill (0,0) circle (1.2pt);
		\fill (1,2) circle (1.2pt);
		\node at (-2.2,-0.2) {$e_1$};
		\node at (2.2,-0.2) {$e_2$};
		\node at (0,4.2) {$e_3$};
		\node at (1.2,2.2) {$v_1$};
	\node at (-1.3,2.2) {$v_2$};
	\node at (0,-0.3) {$v_3$};
	\node at (0,-1.5) {$Y_2$};
		\end{tikzpicture}
			}
			\scalebox{1}{
			\begin{tikzpicture}[scale=0.65]
		\draw (0,0)--(-1,2)--(1,2);
		\draw (-1,2)--(2,0);
		\draw (-2,0)--(2,0)--(0,4)--(-2,0);
		\fill (-2,0) circle (1.2pt);
		\fill (2,0) circle (1.2pt);
		\fill (0,4) circle (1.2pt);
		\fill (-1,2) circle (1.2pt);
		\fill (0,0) circle (1.2pt);
		\fill (1,2) circle (1.2pt);
		\node at (-2.2,-0.2) {$e_1$};
		\node at (2.2,-0.2) {$e_2$};
		\node at (0,4.2) {$e_3$};
		\node at (1.2,2.2) {$v_1$};
	\node at (-1.3,2.2) {$v_2$};
	\node at (0,-0.3) {$v_3$};
	\node at (0,-1.5) {$Y_3$};
		\end{tikzpicture}
			}
			\scalebox{1}{
			\begin{tikzpicture}[scale=0.65]
		\draw (-1,2)--(0,0)--(1,2);
		\draw (0,0)--(0,4);
		\draw (-2,0)--(2,0)--(0,4)--(-2,0);
		\fill (-2,0) circle (1.2pt);
		\fill (2,0) circle (1.2pt);
		\fill (0,4) circle (1.2pt);
		\fill (-1,2) circle (1.2pt);
		\fill (0,0) circle (1.2pt);
		\fill (1,2) circle (1.2pt);
		\node at (-2.2,-0.2) {$e_1$};
		\node at (2.2,-0.2) {$e_2$};
		\node at (0,4.2) {$e_3$};
		\node at (1.2,2.2) {$v_1$};
	\node at (-1.3,2.2) {$v_2$};
	\node at (0,-0.3) {$v_3$};
	\node at (0,-1.5) {$Y_4$};
		\end{tikzpicture}
			}
			\caption{Toric description of the crepant resolutions of $\A^3/(\Z/2\Z)^2$.}
			\label{Z2Z2}
		\end{figure}
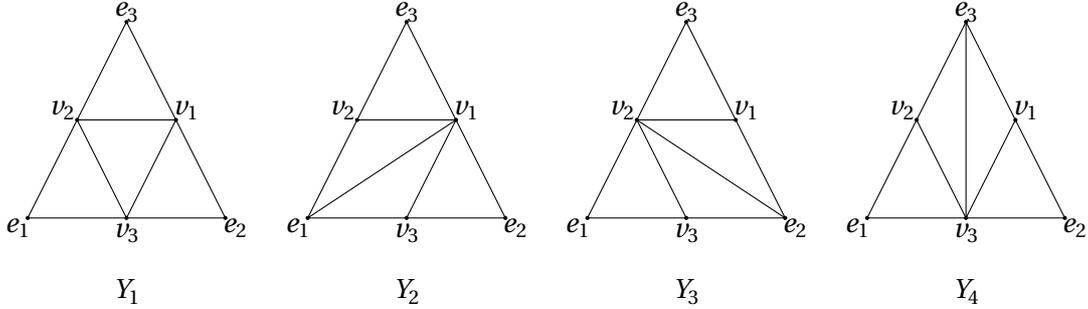 
These diagrams are obtained by considering a fan $\Sigma_i$ for each resolution $Y_i$ then, each simplex in the planar graph is the intersection of a cone in $\Sigma_i$, with the plane containing the heads of the rays that generate $\Sigma_i$. Notice that $Y_1$ differs from the other resolutions by just one flop 
$\begin{tikzcd}[row sep=tiny]
{Y_1}\arrow[dashed,leftrightarrow]{r}{\sigma_i} & Y_i
\end{tikzcd}$ for $i=2,3,4$.

Now, let $\widetilde{\Calo}_i$, for $i=1,\ldots,4$, be the torsion free $\Calo_{Y_i}$-module defined by 
\[
\widetilde{\Calo}_i=\varepsilon_i^*\pi_*\Calo_{\A^3}/\tor_{\Calo_{Y_i}}\varepsilon_i^*\pi_*\Calo_{\A^3},
\]
where $\pi:\A^3\rightarrow X$ is the canonical projection. A direct computation shows that only $\widetilde{\Calo}_1$ is locally free, and, for $i=2,3,4$, the locus where $\widetilde{\Calo}_i$ fails to be locally free coincides with the line flopped by $\sigma_i$. In this setting, it can be shown that the pair $(Y_1,\widetilde{\Calo}_i)$ is canonically isomorphic to the pair $((\Z/2\Z)^2-\Hilb(\A^3),\Calr)$ where $\Calr$ is the tautological bundle.

My future project is to work out conditions on an ideal sheaf $\Calk \subset \Calo_{\A^3}$ and a crepant resolution $Y$ of $\A^3/G$, for $G\subset \SL(3,\C)$ finite subgroup, in order to have $\widetilde \Calk$ locally free and isomorphic to $\Calo_{\A^3}[G]$.
\end{remark}
\providecommand{\bysame}{\leavevmode\hbox to3em{\hrulefill}\thinspace}
\providecommand{\MR}{\relax\ifhmode\unskip\space\fi MR }
\providecommand{\MRhref}[2]{%
	\href{http://www.ams.org/mathscinet-getitem?mr=#1}{#2}
}
\providecommand{\href}[2]{#2}

\bigskip

\medskip
\noindent
\emph{Michele Graffeo}, \texttt{mgraffeo@sissa.it} \\
\textsc{Scuola Internazionale Superiore di Studi Avanzati (SISSA), Via Bonomea 265, 34136 Trieste, Italy}
\end{document}